\documentclass[11pt]{article}

\usepackage[a4paper,margin=1in]{geometry}

\usepackage[T1]{fontenc}
\usepackage[utf8]{inputenc} 
\usepackage[english]{babel}
\usepackage{microtype}

\usepackage{amsmath,amssymb,amsfonts,amsthm,mathtools}

\numberwithin{equation}{section}
\newtheorem{theorem}{Theorem}[section]
\newtheorem{lemma}[theorem]{Lemma}
\newtheorem{proposition}[theorem]{Proposition}
\newtheorem{corollary}[theorem]{Corollary}
\theoremstyle{definition}
\newtheorem{definition}[theorem]{Definition}
\theoremstyle{remark}
\newtheorem{remark}[theorem]{Remark}
\newtheorem{example}[theorem]{Example}

\usepackage{graphicx}
\DeclareGraphicsExtensions{.pdf,.png,.jpg}
\usepackage{subcaption} 
\usepackage{booktabs}
\usepackage{float}

\usepackage{enumitem}
\usepackage{xcolor}
\usepackage{csquotes}

\usepackage{tikz,pgfplots}
\pgfplotsset{compat=1.18}
\usetikzlibrary{graphs,matrix,decorations.pathreplacing,calc}

\usepackage[ruled]{algorithm2e}

\usepackage{hyperref}
\hypersetup{
  colorlinks=true,
  linkcolor=blue!60!black,
  citecolor=blue!60!black,
  urlcolor=blue!60!black,
  pdfauthor={N. Vater, K. R. Foglia, V. Colao, A. Borzì},
  pdftitle={A Low-Rank Symplectic Gradient Adjustment Method for Computing Nash Equilibria}
}
\usepackage[capitalise,noabbrev]{cleveref}

\newcommand{\norm}[1]{\left\lVert#1\right\rVert}
\newcommand{\specnorm}[1]{\norm{#1}_2}

\def\R{\mathbb{R}}

\usetikzlibrary{graphs,matrix,decorations.pathreplacing,calc}

\SetKwInput{KwInput}{Input}                
\newlength\mylen

\SetKwInput{KwOutput}{Output}

\usepackage[draft]{fixme}
\fxsetup{
  marginface=\normalfont\tiny,
theme=color
}

\newcommand{\vertiii}[1]{{\left\vert\kern-0.25ex\left\vert\kern-0.25ex\left\vert #1 
    \right\vert\kern-0.25ex\right\vert\kern-0.25ex\right\vert}}

\def\R{\mathbb{R}}

\def\cM{\cal{M}}

\def\fplrsga{T_{\eta,\tau}^*}
\def\lipDdxf{L_{D \partial_x f}}
\def\lipDdyg{L_{D \partial_y g}}

\title{A Low-Rank Symplectic Gradient Adjustment Method for Computing Nash Equilibria}

\author{N.~Vater\thanks{Institut f\"ur Mathematik, Universit\"at W\"urzburg, Emil-Fischer-Strasse 30, 97074 W\"urzburg, Germany. \texttt{nadja.vater@mathematik.uni-wuerzburg.de}} \and
K.~R.~Foglia\thanks{Department of Mathematics and Computer Science, University of Calabria, Ponte P. Bucci, 30B, 87036 Arcavacata di Rende (CS), Italy. \texttt{katherine.foglia@unical.it}} \and
V.~Colao\thanks{Corresponding author. Department of Mathematics and Computer Science, University of Calabria, Ponte P. Bucci, 30B, 87036 Arcavacata di Rende (CS), Italy. \texttt{vittorio.colao@unical.it}} \and
A.~Borz{\`{\i}}\thanks{Institut f\"ur Mathematik, Universit\"at W\"urzburg, Emil-Fischer-Strasse 30, 97074 W\"urzburg, Germany. \texttt{alfio.borzi@mathematik.uni-wuerzburg.de}}
}

\date{} 

\newcommand{\theinstitute}{}
\newcommand{\institute}[1]{\gdef\theinstitute{#1}}
\apptocmd{\maketitle}{%
  \begingroup
  \def\at{\par}%
  \def\and{\par\medskip}%
  \begin{center}\small\theinstitute\end{center}%
  \endgroup
}{}{}
\newcommand{\keywords}[1]{\par\noindent\textbf{Keywords:} #1\par}
\newcommand{\subclass}[1]{\par\noindent\textbf{MSC (2020):} #1\par}
\newenvironment{acknowledgements}
  {\section*{Acknowledgements}}
  {}

\newif\ifrevcolors
\revcolorstrue
\revcolorsfalse
\ifrevcolors
  \colorlet{revblue}{blue}
  \colorlet{revgreen}{blue}
  \colorlet{revviolet}{blue}
\else
  \colorlet{revblue}{black}
  \colorlet{revgreen}{black}
  \colorlet{revviolet}{black}
\fi

\begin{document}




\title{A low-rank symplectic gradient adjustment method for computing Nash equilibria} 
\author{N. Vater \and
K. R. Foglia \and 
V. Colao$^*$\thanks{$^*$Corresponding author} \and 
A. Borz{\`{\i}}}

\institute{N. Vater \at 
    Institut f\"ur Mathematik, Universit\"at W\"urzburg \\
    Emil-Fischer-Strasse 30, W\"urzburg, 97074, Germany\\
    nadja.vater@mathematik.uni-wuerzburg.de
    \and 
    K. R. Foglia \at
    Department of Mathematics and Computer Science, University of Calabria \\
    Ponte P. Bucci, 30B, Arcavacata di Rende (CS), 87036, Italy\\
    katherine.foglia@unical.it
    \and 
    V. Colao \at 
    Department of Mathematics and Computer Science, University of Calabria \\
    Ponte P. Bucci, 30B, Arcavacata di Rende (CS), 87036, Italy\\
    vittorio.colao@unical.it
    \and
    A. Borz{\`{\i}} \at 
    Institut f\"ur Mathematik, Universit\"at W\"urzburg \\
    Emil-Fischer-Strasse 30, W\"urzburg, 97074, Germany\\
    alfio.borzi@mathematik.uni-wuerzburg.de
    }
\maketitle

\begin{abstract}
This work presents a theoretical and numerical investigation 
of the symplectic gradient adjustment (SGA) method and of a low-rank SGA (LRSGA) 
method for efficiently solving revviolet optimization problems arising from two-player Nash games. The SGA method outperforms the gradient method 
by including second-order mixed derivatives computed at each iterate, which 
requires considerably larger computational effort. 
For this reason, an LRSGA method is proposed where the approximation to 
second-order mixed derivatives is obtained by rank-one updates. 
The theoretical analysis presented in this work focuses on novel convergence 
estimates for the SGA and LRSGA methods, including parameter bounds. 
The numerical experiments complement the theory by
studying the behavior of LRSGA on explicit deterministic games with known
equilibria and by evaluating its computational advantage over exact SGA on a
CLIP-inspired neural-network training task, where LRSGA achieves comparable loss
values lower CPU time than SGA with explicitly
assembled mixed-derivative blocks.
\end{abstract}




 
\subclass{49M15 \and 65H04 \and 65H10 \and 90C30 \and 47H05}





\section{Introduction}
\label{sec:intro}

Multi-objective optimization problems in a game framework appear in many fields of application ranging from aerodynamics to topology optimization and 
machine learning; see, e.g., \cite{Gemp2018,Goodfellow2014,Habbal2004,Navon2022,SchaeferAnandkumar2019,Tang2007,Tang2016,Tennenholtz2002}. In this framework, Nash equilibria (NE) problems are in the focus of these applications or they are instrumental 
for identifying special solutions along the Pareto front, usually by means of 
bargaining criteria \cite{Binois2020,Cala2021,Navon2022}. This fact motivates 
a great effort in developing efficient numerical methods for solving NE 
problems, and the purpose of our work is to contribute to this ongoing research 
with the analysis and development of gradient-based techniques 
applied to the following class of NE problems:
Find $(x_*, y_*) \in X \times Y$ such that
\begin{equation}
\begin{split}
f(x_*,y_*)   & = \min_{x \in X} f(x,y_*) , \\
g (x_*,y_*) & = \min_{y \in Y} g (x_*,y) ,
\end{split}
\label{GNEprob}
\end{equation}
where $X \subset \R^m $ and $Y \subset \R^n $ are two compact and convex subsets and $f, g \in C^3(\Omega, \mathbb{R})$ with $\Omega \subseteq X \times Y$ being a convex domain.

We remark that this \textcolor{revviolet}{optimization problem arising from two-player Nash games} can
be seen as a generalization of a Nash game of two players having  
mixed strategies \cite{Osborne2004}. In particular, we consider 
zero-sum games where $g(x,y)=-f(x,y)$, and player one's gain 
with strategy $x$ corresponds to player two's loss with strategy $y$. 

Assuming differentiability of the loss (objective) functions $f$ and $g$, 
a natural first-order approach for solving \eqref{GNEprob} is to apply the following simultaneous update:
\begin{equation}
\begin{split}
    x_{k+1} &= x_k - \eta \, \partial_x f(x_k,y_k) \\
     y_{k+1} &= y_{k}  - \eta \, \partial_y g(x_k,y_k),
    \end{split} 
         \label{eGD}
\end{equation}
where $\eta > 0$ represents a stepsize.
\textcolor{revgreen}{Although \eqref{eGD} is written in a loss-minimization
form, the game-optimization literature usually refers to this first-order
scheme as gradient descent-ascent (GDA), especially in min--max formulations
where the second loss \(g\) is replaced by the payoff \(-g\) (see e.g. \cite{Balduzzi2018,Letcher2019}). We therefore use
the term GDA for \eqref{eGD} in the sequel.}
Unfortunately, the \textcolor{revgreen}{GDA method} has the disadvantage that, although it may
converge for sufficiently small stepsize $\eta >0$, it produces iterates that 
spiral around the NE solution, thus considerably slowing down convergence.
We illustrate this fact considering the following loss functions:
\begin{equation}
    f(x, y) = \frac{1}{2} x^2 + xy,
    \qquad
    g(x, y) = \frac{1}{2} y^2 - xy.
    \label{eNEsimple}
\end{equation}
As discussed in the next section, the NE problem \eqref{GNEprob}
with these loss functions has the unique NE solution
given by $(x_*, y_*) = (0,0)$. Now, we apply the method \eqref{eGD}
with initialization $(x_0, y_0) = (1,1)$ and $\eta=0.7$ and obtain
the spiraling iterates depicted in Figure \ref{figGD}  (left). One can also see
that, by choosing $\eta=1$, all iterates take sequentially the values of
the coordinates $(\pm 1, \pm 1)$, and never reach the NE point; see Figure \ref{figGD}  (right).
\begin{figure}[h]
    \centering
    \includegraphics[width=0.4\textwidth]{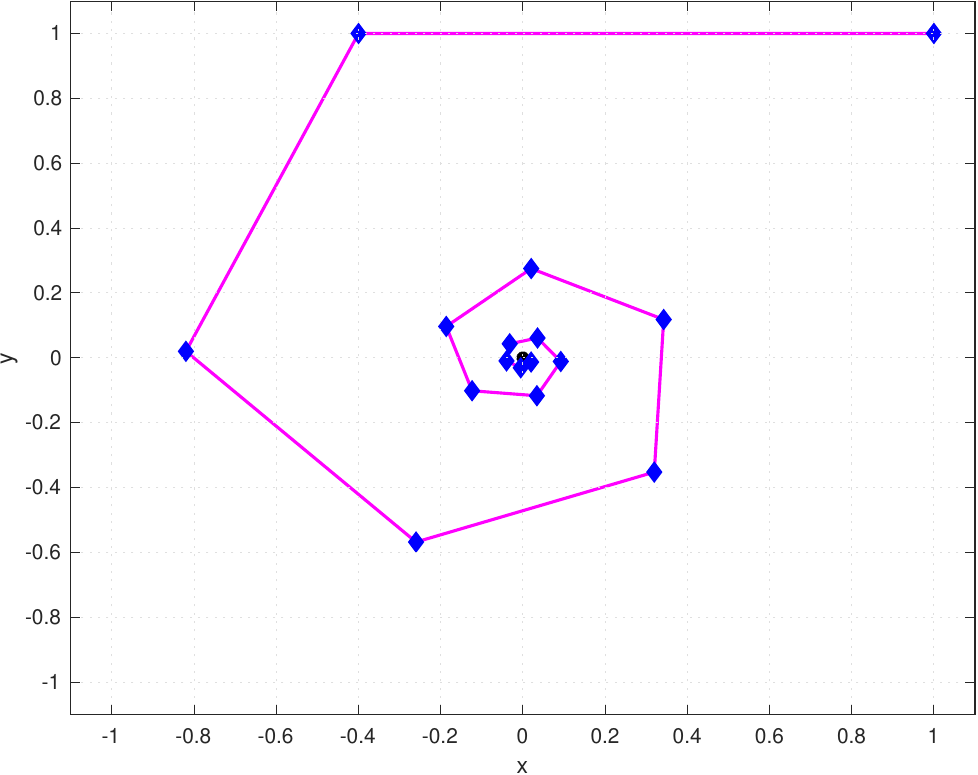}
    \includegraphics[width=0.4\textwidth]{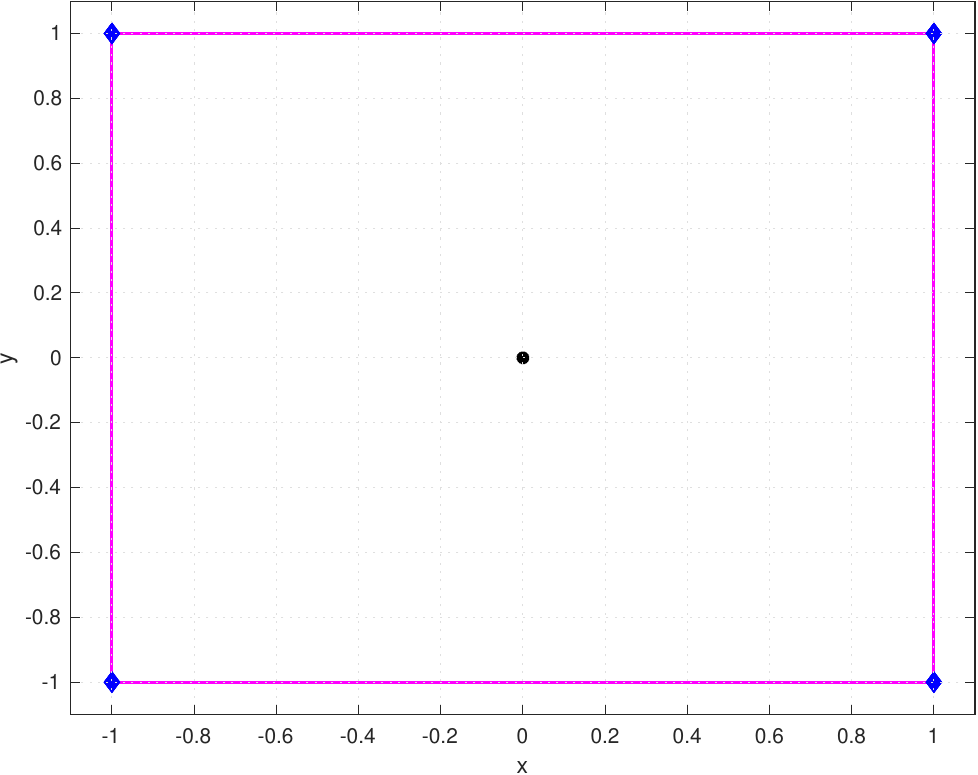}
    \caption{Iterates of the \textcolor{revgreen}{GDA method} with $\eta=0.7$ (left) and $\eta=1$ (right).}
    \label{figGD}
\end{figure}

The phenomenon of spiraling of the \textcolor{revgreen}{GDA} iterates has inspired recent
research work to improve its understanding in a more general setting and has been
a source of inspiration for designing modified \textcolor{revgreen}{GDA} methods that do not spiral
but converge `straight' to the NE solution. We refer to \cite{SchaeferAnandkumar2019} 
for a detailed review and references on modified \textcolor{revgreen}{GDA} methods. One approach is to formulate iterative schemes using second-order
information that results in a preconditioned GDA method as follows:
\begin{equation}
    \label{modifiedGD}
    \begin{pmatrix} x_{k+1} \\ y_{k+1} \end{pmatrix}
    = \begin{pmatrix} x_{k} \\ y_{k} \end{pmatrix}
    - \eta \, \cM_k 
    \begin{pmatrix} \partial_x f_k \\ \partial_y g_k  \end{pmatrix},
\end{equation}
where $\partial_x f_k:=\partial_x f (x_k,y_k)$ and 
$\partial_y g_k:=\partial_y g (x_k,y_k)$, and the entries of the matrix $\cM_k  \in \mathbb{R}^{(m+n) \times (m+n)}$ are determined by second-order derivatives of $f$ and $g$ at the current iterate 
$(x_k,y_k)$; see \cite{Balduzzi2018,Letcher2019,SchaeferAnandkumar2019}. Specifically, 
in \cite{SchaeferAnandkumar2019} a competitive GD (CGD) method is proposed, 
where 
\begin{equation}
    \label{MCGD}
\cM_{CGD, k}= \begin{pmatrix} I_m & \eta \,  \partial_{xy}^2 f_k \\ 
     \eta \, \partial_{yx}^2 g_k  & I_n \end{pmatrix} ^{\text{$-1$}} \approx
     \begin{pmatrix} I_m & - \eta \,  \partial_{xy}^2 f_k \\ 
    - \eta \, \partial_{yx}^2 g_k  & I_n \end{pmatrix},
\end{equation}
where $I_m$ (resp. $I_n$) denotes the identity matrix in $\mathbb{R}^{m \times m}$ (resp. $\mathbb{R}^{n \times n}$), and the given approximation, for $\eta$ sufficiently small,
is also proposed in \cite{SchaeferAnandkumar2019}.
\textcolor{revblue}{This approximation is the first-order Neumann expansion of the
block inverse. Writing the matrix in \eqref{MCGD} as \(I+\eta B\) with
\(B=\left(\begin{smallmatrix} 0 & \partial_{xy}^2 f_k \\ \partial_{yx}^2 g_k & 0\end{smallmatrix}\right)\),
we have \((I+\eta B)^{-1}=I-\eta B+\mathcal O(\eta^2)\), and neglecting the
\(\mathcal O(\eta^2)\) term yields exactly the right-hand side of \eqref{MCGD}.}
Further, in \cite{Balduzzi2018,Letcher2019} a symplectic gradient adjustment (SGA)
method is presented, where $\cM_k$, including a parameter $\tau>0$, is given by 
\begin{equation}
    \label{MSGA}
\cM_{SGA, k}= \begin{pmatrix} I_m & - \frac{\tau}{2}\, ( \partial_{xy}^2 f_k - \partial_{yx}^2 g_k^T)\\ 
    - \frac{\tau}{2}\, (\partial_{yx}^2 g_k - \partial_{xy}^2 f_k^T) & I_n \end{pmatrix} .
\end{equation}
\textcolor{revgreen}{For the motivating example in \eqref{eNEsimple}, the
effect of second-order derivative corrections can be observed directly. With
the same stepsize \(\eta=1\) used in Figure~\ref{figGD}, GDA cycles around the
Nash equilibrium, whereas exact CGD reaches the equilibrium in one step in this
linear-quadratic case, and SGA with \(\tau=0.5\) damps the rotational component
and converges. In this particular example, exact CGD is favored by the special
algebraic structure of the game; therefore, the one-step convergence should be
read as an illustrative property of this linear-quadratic case rather than as a
general comparison between CGD and SGA.}

\begin{figure}[h]
    \centering
    \includegraphics[width=0.8\textwidth]{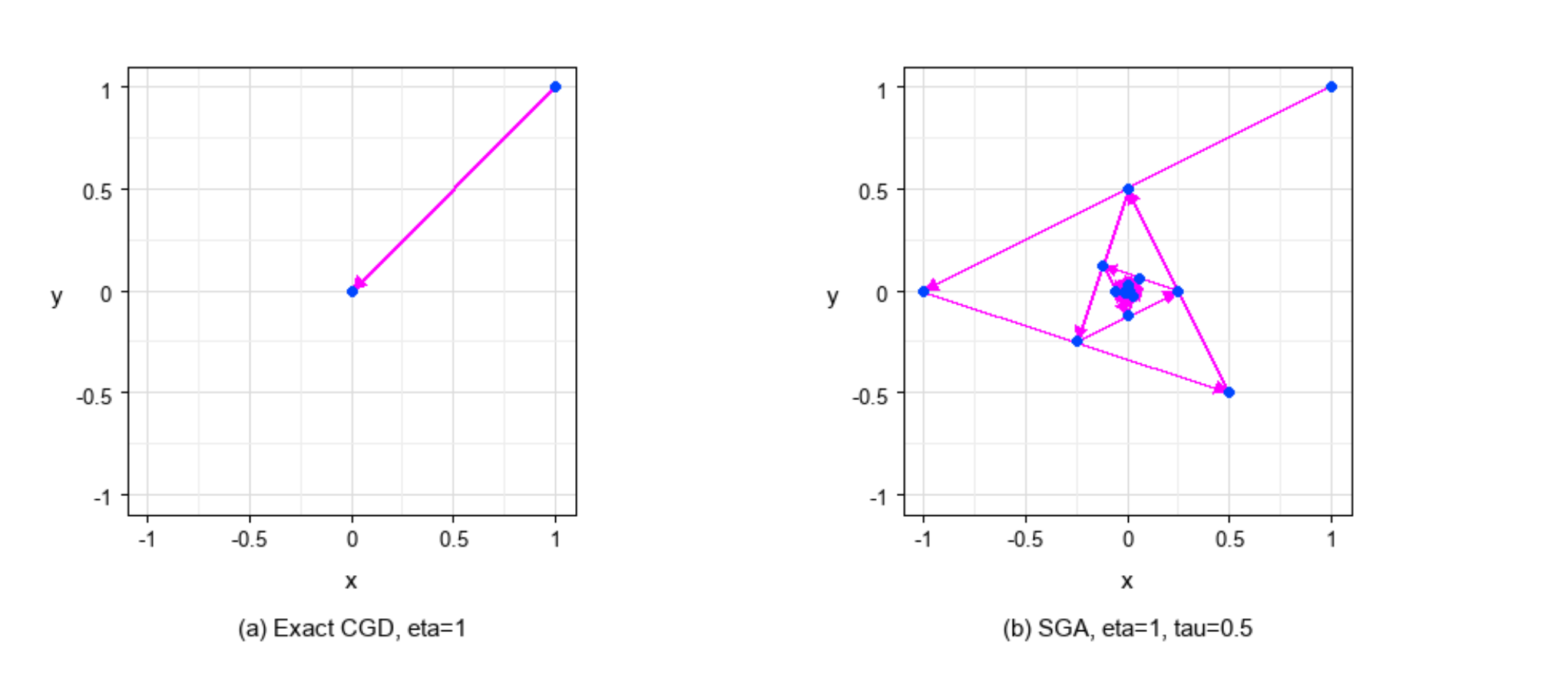}
    \caption{\textcolor{revgreen}{Iterates of exact CGD with \(\eta=1\)
    (left) and SGA with \(\eta=1\), \(\tau=0.5\) (right) for the motivating game
    in \eqref{eNEsimple}.}}
    \label{figCGDSGA}
\end{figure}

Notice that the two methods above, although derived from different perspectives, 
are somewhat similar. In fact, they have similar convergence performance. 

However, while the inclusion of second-order information with 
the mixed derivatives $\partial_{xy}^2 f_k$ and $\textcolor{revblue}{\partial_{yx}^2 g_k}$ 
considerably improves the convergence behaviour of the CGD and 
SGA methods in comparison to the GDA method, it requires
additional computational effort that can make the application 
of these methods to large-size problems prohibitive, especially 
in the case of game problems involving neural networks, where the computation 
of second-order derivatives involves back-propagation processes.

This fact motivates our investigation that aims at developing effective rank-one
updates to the mixed derivatives matrices that are based only on 
first-order information. Furthermore, we support our development 
with theoretical results that can be applied to the original methods 
discussed above and also accommodate our low-rank variant. 
However, in order to keep this work to a reasonable size, we present our results and techniques for the SGA method and its low-rank-update version. 

In the context of two-player NE problems, we investigate the \textcolor{revblue}{convergence of the SGA method from a metric point of view, particularly emphasizing monotonicity properties}. Further, we extend our analysis to our low-rank SGA (LRSGA) by proving its convergence under suitable assumptions. 
Our theoretical work is complemented by results of numerical
experiments comparing the computational performance of the SGA and LRSGA methods.

\textcolor{revblue}{More precisely, our convergence analysis is first carried out in
the broader class of real, possibly infinite-dimensional, Hilbert spaces, where
the convergence of these schemes is shown to follow from metric properties of a
skew-corrected operator; the finite-dimensional game results are then obtained as
a special case. Within this framework we study three iterations. The first is the
SGA method
\[
z_{k+1}=z_k-\eta\,(I-\tau A(z_k))\,F(z_k),
\]
where \(F\) denotes the game gradient and \(A\) the antisymmetric part of the
game Hessian (both introduced in Section~\ref{sec:Nash}). The second is its
frozen-coefficient variant, the frozen SGA method, in which the skew correction
is evaluated once at the equilibrium \(w_*\),
\[
v_{k+1}=v_k-\eta\,(I-\tau A(w_*))\,F(v_k);
\]
\textcolor{revviolet}{The frozen SGA iteration is introduced solely as an analytical tool and
is not intended for practical implementation, since it requires
knowledge of the equilibrium \(w_*\).}
The third is the low-rank method LRSGA,
\[
w_{k+1}=w_k-\eta\,(I-\tau \alpha_k)\,F(w_k),
\]
where \(\alpha_k\) is a skew-symmetric approximation of \(A(w_*)\) built from
rank-one secant updates using first-order information only.}


In the next section, we define \textcolor{revblue}{two-player} Nash games and the corresponding solution concept of Nash equilibria. Sufficient conditions for the existence of NE solutions are given. In Section~\ref{sec:SGAandLR}, we illustrate the construction of the SGA method 
and our low-rank variant LRSGA method. Section~\ref{sec:SGAtheory} is devoted to the analysis of the SGA method, where we propose a theoretical framework for a convergence analysis of the SGA method that is based on metric properties of the involved operators\textcolor{revblue}{, formulated in a general, possibly infinite-dimensional, Hilbert space and then specialized to the finite-dimensional game}. In particular, we provide explicit bounds for the parameters involved in the SGA method, addressing both global convergence for quadratic objective functions and local convergence near equilibrium points. Our results are derived through an analysis of the properties of the operators involved, such as strong monotonicity and inverse strong monotonicity. In Section~\ref{sec:LRSGAtheory}, this analysis is extended to the LRSGA method.
\textcolor{revgreen}{In Section~\ref{sec:numerical-experiments}, we report
three numerical experiments. The first two experiments consider explicit
deterministic games and evaluate the qualitative behavior of the proposed LRSGA
method near known equilibria. The third experiment is a CLIP-inspired
neural-network toy example, where we compare LRSGA with exact SGA and show that
LRSGA achieves comparable loss values with a substantial reduction in training
time. A section of conclusions completes the paper.}

\section{Nash equilibria and their characterization}
\label{sec:Nash}

We consider coupled optimization problems involving two sets 
of optimization variables, representing the players' (mixed) strategies, 
and two corresponding loss functions. In a non-cooperative setting, the objectives are such that 
minimizing one function with respect to one set of variables may result 
in an increase of the other function and vice-versa. In this competitive setting, 
a widely used concept of optimality is called Nash equilibrium \cite{Nash1948,Nash1951}. 
This equilibrium configuration defines the solution of a non-cooperative Nash game. 

As already mentioned in the introduction, we consider 
game problems with two players whose strategies are identified with the 
variables $x \in X \subset \R^m $ and $y \in Y \subset \R^n $, where $X$ and $Y$ are two compact and convex subsets. The purpose of each player is to minimize 
the corresponding loss function that depends 
on both strategies $x$ and $y$. We denote these functions with 
$f$ for the player $P_1$, and $g$ for the player $P_2$. In this 
context, a Nash equilibrium is defined as follows.

\begin{definition}\label{DefNE}
A pair $(x_*, y_*) \in X \times Y$ is a Nash equilibrium (NE) of a \textcolor{revblue}{two-player game} if 

\begin{enumerate} 
\item[(1)] The strategy $x_* \in X$ is the best choice of player $P_1$ in reply to the strategy 
$y_* \in Y$ adopted by player $P_2$ as follows: 
$$
f(x_*,y_*) = \min_{x \in X} f(x,y_*) ;
$$
\label{NEa}
\item[(2)] The strategy $y_* \in Y$ is the best choice of player $P_2$ in reply to the strategy 
$x_* \in X$ adopted by player $P_1$ as follows: 
$$
g (x_*,y_*) = \min_{y \in Y} g (x_*,y) .
$$
\label{NEb}
\end{enumerate}

\end{definition}

In this setting, it is assumed that each player has knowledge of both sets 
of strategies $X$ and $Y$, and of the objectives $f$ and $g$. \textcolor{revblue}{With the definition above, an equilibrium is achieved when no player can improve its objective by a unilateral change of strategy.} 

A zero-sum game corresponds to the setting $g = -f$. 
In this case, the first player with strategy $x \in X$ aims at 
minimizing $f$, while the second player aims at maximizing it using strategy $y \in Y$. 
Because in this game the gain for one player is equivalent to the loss for the other, 
we have a strictly competitive game. Zero-sum games represent the starting 
point of modern game theory introduced by John von Neumann \cite{vonNeumann1928}.

The NE defined above can be analyzed based on 
the best-response maps. For player $P_1$, the best-response correspondence $R_1 : Y \rightarrow 
P(X)$, where $P(X)$ denotes the power set of $X$, is given by 
$$
y \mapsto R_1(y):= \{ x \in X \, : \, f(x,y) \le f(x',y), \, x' \in X  \} .
$$
For player $P_2$, the best-response correspondence $R_2 : X \rightarrow 
P(Y)$ is given by 
$$
x \mapsto R_2(x):= \{ y \in Y \, : \, g(x,y) \le g(x,y'), \, y' \in Y  \} .
$$
Hence \textcolor{revblue}{we define} the multifunction $R : X \times Y \rightrightarrows X \times Y$ \textcolor{revblue}{by} $R(x,y) := R_1(y) \times R_2(x)$ so that a NE point $(x_*,y_*)$ 
corresponds to a fixed point of $R$ as follows: 
$$
(x_*,y_*) \in R(x_*,y_*) . 
$$

Notice that we can also define a 
local NE where (1) and (2) in Definition \ref{DefNE} hold for all $x \in X$ in an 
open neighbourhood of $x_*$, \textcolor{revblue}{respectively,} for all $y \in Y$ in an 
open neighbourhood of $y_*$. A local NE is isolated if there exists a neighbourhood 
of $(x_*,y_*)$ where this NE point is unique. 
We have the following existence result, whose proof is based 
on Kakutani's fixed-point theorem \cite{Debreu1952,Fan1952,Glicksberg1952}:

\begin{theorem}
Assume that $X$ and $Y$ are compact and convex subsets. 
Let $f$ and $g$ be continuous, and assume that the map 
$x \mapsto f(x,y)$ is a convex function of $x$, for each fixed $y \in Y$; 
further assume that the map 
$y \mapsto g(x,y)$ is a convex function of $y$, for each fixed $x \in X$. 
Then there exists a Nash equilibrium. 
\label{theorem-Nash}
\end{theorem}
In the case of a zero-sum game, the same theorem applies with 
$\phi(x,y):=f(x,y)=-g(x,y)$, assuming that the map 
$x \mapsto \phi(x,y)$ is a convex function of $x$, for each fixed $y \in Y$, 
and the map 
$y \mapsto \phi(x,y)$ is a concave function of $y$, for each fixed $x \in X$. 
In his original work \cite{vonNeumann1928}, von Neumann considered
\textcolor{revblue}{\(f(x,y)=x^T A y\), where} $A \in \R^{m \times n}$ is called the payoff or game matrix.

We assume that $f$ and $g$ are both twice continuously differentiable.
\textcolor{revblue}{From this point on, whenever first-order conditions are used,
we restrict attention to equilibria in the interior of the feasible set
\[
K:=X\times Y,\qquad w_*=(x_*,y_*)\in\operatorname{int}(K).
\]
The stationarity equations below are therefore used only as interior
first-order necessary conditions. Boundary equilibria would require a
variational inequality, or equivalently a normal-cone formulation, which is not
the setting considered in the convergence analysis.}
Then, we can characterize the NE point by the following 
necessary first-order equilibrium conditions; see, e.g., \cite{Ratliff2016}: 
\begin{equation}
\partial_{x} f(x_*,y_*) =0, \qquad \partial_{y} g (x_*,y_*) =0 .
\label{equiGradNE}
\end{equation}
We define the game gradient as the mapping $F: \mathbb{R}^{m} \times \mathbb{R}^{n} \to \mathbb{R}^{m + n}$ given by
\begin{equation}
    \label{eq:DefGameGradient}
    F(x, y) = \begin{pmatrix} \partial_{x} f(x, y) \\ \partial_{y} g(x, y) \end{pmatrix} .
\end{equation}
For a shorter notation, we also write 
$w = (x, y) \in \mathbb{R}^{m} \times \mathbb{R}^{n}$, 
which is intended as a row or column vector, depending on the context. Then, the necessary conditions \eqref{equiGradNE} for $w_* \in \mathbb{R}^{m + n}$ to be a Nash equilibrium read as $F(w_*) = 0$.

Necessary second-order equilibrium conditions are given by 
the following positive semi-definiteness of the Hessians: 
\begin{equation}
{\partial_{x x}^2} f(x_*,y_*) \succeq 0, \qquad 
{\partial_{y y}^2} g(x_*,y_*) \succeq 0 . 
\label{equiHess1NE}
\end{equation}
If these inequalities are strict and the game gradient is zero, 
then $(x_*,y_*) $ is a \textcolor{revgreen}{local} Nash equilibrium; see \cite{Ratliff2016}.
However, these conditions do not imply that the NE point is isolated. 

Now, consider the mapping $H: \mathbb{R}^{m} \times \mathbb{R}^n \to \mathbb{R}^{(m+n) \times (m+n)}$ defined by 
\begin{equation}
    \label{eq:DefGameHessian}
    H(x, y) = D F (x, y) = 
    \begin{pmatrix}
        {\partial_{x x}^2} f(x,y) & \textcolor{revblue}{\partial_{x y}^2 f(x,y)} \\
        \textcolor{revblue}{\partial_{y x}^2 g(x,y)} & {\partial_{y y}^2} g(x,y)
    \end{pmatrix}  .
\end{equation}
This matrix is called the game Hessian, and $D F$ is the Jacobian of the game
gradient $F$. We consider the decomposition of the game Hessian 
as $H(x, y) = S(x,y) + A(x,y)$, where $S=S^T=(H+H^T)/2$ is the symmetric part of \textcolor{revblue}{\(H\)}, and 
$A=-A^T= (H-H^T)/2$ represents its antisymmetric part.

We remark that, if $(x_*,y_*) $ is a NE point, and $H(x_*, y_*)$ is invertible, then this NE point is isolated \cite{Ratliff2016}. For this reason, in 
order to rule out the occurrence of a continuum of NE points, which are usually unstable, 
we always assume that $H(x_*, y_*)$ is invertible. Further, we assume that 
$H(x_*, y_*)$ is positive semi-definite, $H(x_*,y_*) \succeq 0$, in the 
sense that $\langle w, H(w_*)\, w \rangle \ge 0$, $w \in \R^{m+n}$,
\textcolor{revblue}{where} $\langle \cdot, \cdot \rangle$ \textcolor{revblue}{denotes} the Euclidean scalar product and $\specnorm{\cdot}$ the corresponding Euclidean norm. Therefore, we define stable NE (SNE) points as follows \cite{Letcher2019}:

\begin{definition}\label{DefSNEpoint}
A NE point $w_*=(x_*, y_*)$ is said to be stable, if $H(x_*, y_*)$ is invertible and
positive semi-definite.
\end{definition}

\textcolor{revviolet}{The notion of stability adopted here is motivated by the local dynamics
associated with the game gradient. Let \(w_*\) be a Nash equilibrium,
so that \(F(w_*)=0\), and consider a perturbation
\[
z=w-w_*.
\]
Since
\[
F(w)=H(w_*)\,z+o(\|z\|),
\]
the continuous gradient dynamics
\[
\dot w = -F(w)
\]
is locally approximated by the linear system
\[
\dot z=-H(w_*)\,z.
\]
Writing
\[
H(w_*)=S(w_*)+A(w_*),
\]
where
\[
S(w_*)=\frac{1}{2}\big(H(w_*)+H(w_*)^T\big)
\]
and
\[
A(w_*)=\frac{1}{2}\big(H(w_*)-H(w_*)^T\big),
\]
we obtain
\[
\frac{1}{2}\frac{d}{dt}\|z\|^2
= -z^T H(w_*) z
= -z^T S(w_*) z,
\]
since
\[
z^T A(w_*) z=0.
\]
Therefore, the symmetric part of the game Hessian completely determines
the local dissipation properties of the linearized dynamics, whereas the
antisymmetric part contributes only rotational effects. In particular,
the condition
\[
H(w_*)\succeq0
\]
implies that the linearized dynamics is nonexpansive, while
\[
H(w_*)\succ0
\]
yields strict local dissipation. This observation motivates the notion
of stable Nash equilibrium adopted in the SGA literature
\cite{Balduzzi2018,Letcher2019}.}

\textcolor{revblue}{The invertibility of
\(H(w_*)\) guarantees that \(w_*\) is an isolated equilibrium, while the
\textcolor{revviolet}{condition \(H(w_*)\succeq 0\) is stronger than the standard blockwise
second-order Nash conditions and is the condition adopted in the SGA
literature to characterize stable equilibria.}}
\textcolor{revblue}{We remark that the term stable is not meant to assert
structural stability of the game, nor a stability notion for variational
inequalities; it refers only to the dissipation properties of the linearized
dynamics discussed above.}

In our work, we focus on the convergence properties of the SGA and LRSGA methods
for computing SNE points. This is a large class of solutions which include 
those arising in zero-sum games, as stated in the following lemma. 
(See \cite{Balduzzi2018} for a similar statement.)

\begin{lemma}
Let $w_*=(x_*, y_*)$ be a NE point of a zero-sum game, and 
$H(x_*, y_*)$ be invertible. Then $w_*=(x_*, y_*)$ is a stable NE point.
\end{lemma}
\begin{proof}
Consider the decomposition $H(w_*) = S(w_*) + A(w_*)$ for the case 
$g(x,y)=-f(x,y)$. We have
$$
    S(x_*, y_*) = 
    \begin{pmatrix}
        {\partial_{x x}^2} f(x_*, y_*) & 0 \\
        0 & {-\partial_{y y}^2} f(x_*, y_*)
    \end{pmatrix}  .
$$
We see that the symmetric part satisfies $\langle w, S(w_*)\, w \rangle \ge 0$ 
because of \eqref{equiHess1NE}, that is, \textcolor{revblue}{${\partial_{x x}^2} f(\textcolor{revgreen}{x_*,y_*})\succeq 0$} and $ {-\partial_{y y}^2} f (\textcolor{revgreen}{x_*,y_*})={\partial_{y y}^2} g(x_*,y_*) \succeq 0$. 
\textcolor{revblue}{Then \(H(w_*)\) is positive semidefinite since}
\[
\langle w, H(w_*)\, w \rangle
= \langle w, S(w_*)\, w \rangle +\langle w, A(w_*)\, w \rangle
= \langle w, S(w_*)\, w \rangle \ge 0 .
\]
\textcolor{revblue}{The claim follows from this result and the invertibility assumption.}
\end{proof}

\section{The SGA and LRSGA methods}
\label{sec:SGAandLR}

We remark that the condition $H(w_*) \succeq 0$ becomes necessary for 
obtaining convergence of several methods, even in the case of quadratic 
loss functions. In particular, this is true for the SGA method, as illustrated by the following example. 

\begin{example}
\label{ex:quadratic_game}
        Consider a two-player NE game with 
        $f(x, y) = x^2 + 3 xy$ and $g(x, y) = y^2 + 3 xy$. 
        The point $(x_*, y_*) = (0, 0)$ satisfies $F(x_*, y_*) = \begin{pmatrix} 2x_* + 3 y_* \\ 2y_* + 3 x_* \end{pmatrix} = 0$. \textcolor{revblue}{Moreover,}
        $\partial_{xx}^2 f(x_*, y_*) =2$ and $\partial_{yy}^2 g(x_*, y_*) =2$, 
        which are both strictly positive. Thus, the point $(x_*, y_*)$ is a Nash equilibrium. The game Hessian is given by
       $$
       H(x, y) = \begin{pmatrix} 2 & 3 \\3 &
                2 \end{pmatrix}, 
      $$
         which is invertible but not positive semi-definite. In fact, for any $(x,y)\in \R^2$, 
         the eigenvalues of $H(x, y)$ are given by $\lambda_1 = -1$ and
        \textcolor{revblue}{$\lambda_2 =5$}. Hence, the game Hessian is indefinite at the NE.
        
        Since $A(x,y)=0$, the SGA method coincides with the \textcolor{revgreen}{GDA method}
(the matrix $\cM_{SGA,k}$ is the identity). We have the \textcolor{revgreen}{corresponding first-order} iterates given by
\[
    \begin{pmatrix} x_{k+1} \\ y_{k+1}\end{pmatrix} = 
  \begin{pmatrix} 1 - 2 \eta & - 3 \eta \\ -3 \eta & 1 - 2 \eta  \end{pmatrix} \begin{pmatrix} x_k \\ y_k \end{pmatrix}.
\]
In fact, $(\varepsilon, - \varepsilon)$ is an eigenvector of the iteration matrix with corresponding eigenvalue $1 + \eta > 1$. Therefore, this sequence diverges for any initial choice $(x_0,y_0)=(\varepsilon, -\varepsilon)$ with $\varepsilon\neq 0$ and any choice of the parameters $\eta>0$. 

\end{example}
Notice that the above is an example of potential game \cite{MondererShapley1996}, with potential $\phi(x,y)=x^2+y^2+3xy$, and a counterexample to the statement
in \cite{Balduzzi2018,Letcher2019} that the \textcolor{revgreen}{GDA} method always converges for this class of games.

Next, we illustrate the SGA and LRSGA methods for determining SNE 
points in \textcolor{revblue}{two-player} games. The SGA method is an iterative scheme given by
\begin{equation}
    \label{eq:def_sga}
    \begin{pmatrix} x_{k+1} \\ y_{k+1} \end{pmatrix}
    = \begin{pmatrix} x_{k} \\ y_{k} \end{pmatrix}
    - \eta \, \begin{pmatrix} I_m & - \frac{\tau}{2}\, ( \partial_{xy}^2 f_k - \partial_{yx}^2 g_k^T)\\ 
    - \frac{\tau}{2}\, (\partial_{yx}^2 g_k - \partial_{xy}^2 f_k^T) & I_n \end{pmatrix}
    \begin{pmatrix} \partial_x f_k\\ \partial_y g_k \end{pmatrix},
\end{equation}
where $I_m$ (resp. $I_n$) denotes the identity matrix in $\R^{m \times m}$ (resp. $\R^{n \times n}$), $\eta >0$ represents the stepsize, and $\tau > 0$ is \textcolor{revblue}{an} auxiliary parameter.

The SGA method can be defined in terms of the game gradient $F$ and game Hessian $H$, 
and the decomposition $H(x, y)=S(x, y) + A(x, y)$, where $S: \mathbb{R}^{m} \times \mathbb{R}^n \to \mathbb{R}^{(m+n) \times (m+n)}$ and $A: \mathbb{R}^{m} \times \mathbb{R}^n \to \mathbb{R}^{(m+n) \times (m+n)}$ map $(x, y) \in \mathbb{R}^{m} \times \mathbb{R}^n$ to the symmetric and anti-symmetric components of the game Hessian. For a general 
\textcolor{revblue}{two-player} game, we have 
    \[
        S(x,y)=\begin{pmatrix}
            \partial_{xx} ^2 f(x, y)  & \frac{1}{2}(\partial_{xy}^2 f(x, y) + \partial_{yx}^2 g(x, y)^T) \\
            \frac{1}{2}(\partial_{xy}^2 f(x, y)^T + \partial_{yx}^2 g(x, y)) & \partial_{yy} ^2 g(x, y) 
        \end{pmatrix}
    \]
and
\[
    A(x,y)=\begin{pmatrix}
        0  & \frac{1}{2}(\partial_{xy}^2 f(x, y) - \partial_{yx}^2 g(x, y)^T) \\
        -\frac{1}{2}(\partial_{xy}^2 f(x, y)^T - \partial_{yx}^2 g(x, y)) & 0 
    \end{pmatrix}.
\]
Then, using the notation $z_k := (x_k, y_k)$, the SGA method can be written as follows:
\begin{equation}
    \label{eq:DefSGAReformulated}
    z_{k+1} = z_k - \eta \,  (I - \tau \, A(z_k)) \, F(z_k),
\end{equation}
where $I$ \textcolor{revblue}{denotes} the identity matrix in the appropriate Euclidean space.
By defining the operator $\Gamma_\tau: \mathbb{R}^{m+n} \to \mathbb{R}^{m + n}$ as $z \mapsto (I-\tau \, A(z)) \,F(z)$, we obtain the following operator form of the SGA method:
\begin{equation}
\label{eq:sga_operator_Form}
z_{k+1} = z_k - \eta \, \Gamma_\tau(z_k).
\end{equation}

A main disadvantage in the SGA method is 
the computation of the mixed derivatives $\partial_{xy}^2 f_k$ and $\textcolor{revblue}{\partial_{yx}^2 g_k}$. 
For large-size Nash games, the computation of these derivatives 
can be problematic, especially in the context of neural networks where 
back-propagation is involved. \textcolor{revblue}{Our purpose is to investigate how to efficiently}
determine matrices 
$M_k \in \R^{m \times n}$ and $N_k \in \R^{n \times m}$ that 
approximate $ \partial_{xy}^2 f_k$ and $ \partial_{yx}^2 g_k$, respectively, 
such that the convergence properties of the SGA method are preserved or even 
enhanced. 

For this purpose, we develop a least-change secant update method 
\cite{DennisSchnabel1979,WalkerWatson1990} to approximate the \textcolor{revblue}{second-order} derivatives $D \partial_x f$ and $D \partial_y g$ and use those parts of these Jacobians, corresponding to $\partial^2_{xy} f$ and $\partial^2_{yx} g$, for the computation of the new iterate $w_{k+1} = (x_{k+1}, y_{k+1})$.

The rank-one least-change secant update method applied to $D \partial_x f$ and $D \partial_y g$ 
gives the updates:
\begin{align}
    \label{eq:UpdateMixedDerivativesF}
    \mu_{k+1} &= \mu_k +
    \frac{(\partial_x f(w_{k+1}) - \partial_x f(w_k) - \mu_k (w_{k+1} - w_k))(w_{k+1} - w_k)^T}
    {(w_{k+1} - w_k)^T (w_{k+1} - w_k)}, \\
    \label{eq:UpdateMixedDerivativesG}
    \nu_{k+1} &= \nu_k
    + \frac{(\partial_y g(w_{k+1}) - \partial_y g(w_k) - \nu_k (w_{k+1} - w_k))(w_{k+1} - w_k)^T}
    {(w_{k+1} - w_k)^T (w_{k+1} - w_k)}.
\end{align}
\textcolor{revblue}{If \(w_{k+1}=w_k\), the secant quotient is not formed because
the iteration has already reached a fixed point of the LRSGA map (see Proposition~\ref{prop:lrsga-fixed-point-zero}). In numerical
implementations, the update is skipped whenever
\(\|w_{k+1}-w_k\|_2\) is below a prescribed tolerance, keeping
\(\mu_{k+1}=\mu_k\) and \(\nu_{k+1}=\nu_k\).}

Now, consider the block structures
$\mu_k = \begin{pmatrix} (\mu_k)_1 & (\mu_k)_2 \end{pmatrix} \in \mathbb{R}^{m \times (m + n)}$ with $(\mu_k)_1 \in \mathbb{R}^{m \times m}$ and $(\mu_k)_2 \in \mathbb{R}^{m \times n}$, and $\nu_k = \begin{pmatrix} (\nu_k)_1 & (\nu_k)_2 \end{pmatrix} \in \mathbb{R}^{n \times (m + n)}$ with $(\nu_k)_1 \in \mathbb{R}^{n \times m}$, $(\nu_k)_2 \in \mathbb{R}^{n \times n}$. Then, by inspection, one can see that the block $(\mu_k)_2$ corresponds 
to $\partial_{xy}^2 f$, and the block $(\nu_k)_1$ corresponds to $\partial_{yx}^2 g$. 
Therefore, we set $M_k = (\mu_k)_2 \in \mathbb{R}^{m \times n}$ and $N_k = (\nu_k)_1 \in \mathbb{R}^{n \times m}$ to define our low-rank SGA method in the 
form \eqref{modifiedGD}, where the matrix $\cM$ is given by
\begin{equation}
    \label{MLRSGA}
\cM_{LRSGA, k}= \begin{pmatrix} I_m & - \frac{\tau}{2}\, (M_k - N_k^T)\\ 
    - \frac{\tau}{2}\, (N_k - M_k^T) & I_n \end{pmatrix} .
\end{equation}
Clearly, we can also write the LRSGA method as follows: 
\begin{equation}
    \label{eq:def_sga_approx}
        w_{k+1} = w_k- \eta \, (I - \tau \, \alpha_k) \, F(w_k), 
\end{equation}
where
$$
        \alpha_k = \begin{pmatrix} 0 &  \textcolor{revblue}{\frac{1}{2}}\, (M_k - N_k^T)\\ 
            \textcolor{revblue}{\frac{1}{2}}\, (N_k - M_k^T) & 0 \end{pmatrix}.
$$
For the purpose of initialization, one can set 
$M_0=\partial_{xy}^2 f(x_0,y_0)$ and $N_0=\partial_{yx}^2 g(x_0,y_0)$, 
or choose random matrix initialization.

In comparison to the SGA method, our LRSGA approach has the great advantage 
of constructing $M_k$ and $N_k$, $k=1,2, \ldots$, by means of rank-one updates (with $\mathcal{O}((m+n)^2)$ additions and multiplications) instead of computing second-order derivatives at the current iterate.
\textcolor{revviolet}{In contrast, the explicit computation of mixed second derivatives
typically requires Hessian-vector products or second-order
backpropagation, whose cost can significantly exceed that of the rank-one update in large neural-network applications.} In particular, in the case of game problems involving 
neural networks, where the computation of second-order derivatives involves 
backpropagation processes, this advantage results in CPU times significantly smaller than in the SGA method. 

On the other hand, the LRSGA method requires more storage, since the matrices $\mu_k$ and $\nu_k$ contain also approximations to the derivatives $\partial_{xx}^2 f$ and $\partial_{yy}^2 g$, which are not used in the SGA method, but could be used in other iterative schemes. 

\textcolor{revblue}{The term low-rank is used here in the update sense: each
least-change secant correction is rank one. The matrices \(\mu_k\) and
\(\nu_k\) are nevertheless stored as dense matrices, so LRSGA should not be
interpreted as a low-rank storage method unless an additional factorized
implementation is introduced.}

\textcolor{revblue}{It is also useful to distinguish LRSGA from Newton and
quasi-Newton methods for the nonlinear system \(F(w)=0\). A Newton step solves
\(H(w_k)d_k=-F(w_k)\), which requires \(\mathcal O((m+n)^3)\) operations for a
dense linear solve, whereas SGA and LRSGA use the explicit fixed-point steps
\[
d_k=-\eta(I-\tau A(w_k))F(w_k),
\qquad
d_k=-\eta(I-\tau\alpha_k)F(w_k),
\]
respectively. Thus LRSGA uses secant information only to approximate the
skew-symmetric correction in this preconditioned gradient-type iteration, not
to approximate the full Jacobian or its inverse. Accordingly, the deterministic
experiments in this work compare first-order dynamics (GDA, its optimistic
variant, and the extragradient method), exact mixed-derivative dynamics (SGA
and linearized CGD), and the secant-based skew-corrected dynamics of LRSGA;
comparisons with GDA-BFGS and other full quasi-Newton
game optimizers are outside the scope of the present study and are left for
future work.}

\section{Convergence analysis of the SGA method}
\label{sec:SGAtheory}

\begin{color}{revblue}
In this section we develop the convergence analysis of the SGA method in two
steps. We first set up an abstract metric framework for the zeros of a nonlinear
operator in a real Hilbert space, isolating the metric properties
(monotonicity, cocoercivity, Lipschitz continuity) that drive the convergence of
the associated fixed point iteration. We then specialize this framework to the
finite-dimensional game on \(\mathbb R^{m+n}\), where it yields a global
convergence result for quadratic games.

The convergence analysis below is organized around a metric fixed point
principle for zero problems. Given an operator \(B\), the equation
\[
Bz=0
\]
can be studied through the explicit fixed point map
\[
T_\eta:=\operatorname{Id}-\eta B,
\qquad \eta>0,
\]
because
\[
z\in\operatorname{Zer}(B)
\quad\Longleftrightarrow\quad
z\in\operatorname{Fix}(T_\eta).
\]
Thus convergence of the zero-finding scheme 
\(
z_{k+1}=T_\eta z_k
\)
 is typically obtained by proving that \(T_\eta\) is nonexpansive, or
more strongly averaged or contractive, on a suitable invariant set. Without
such a metric property the Picard iteration may fail to converge even when the
zeros of the underlying operator have a meaningful variational interpretation.
For the classical gradient method applied to a convex minimization problem,
this metric mechanism is supplied by the Baillon--Haddad theorem: the gradient
of a convex function with Lipschitz continuous gradient is cocoercive, and
cocoercivity of \(B\) implies averaged nonexpansiveness of
\(\operatorname{Id}-\eta B\) for suitable stepsizes; see, e.g.,
\cite[Chapters~17--18]{bauschke2017}.

The Nash equilibrium setting is more delicate. Saddle functions and zero-sum game
problems naturally give rise to monotone, indeed maximal monotone, operators
in the sense of Rockafellar~\cite{rockafellar1970monotone}. However,
monotonicity alone is not a metric convergence principle for the explicit
fixed point map \(\operatorname{Id}-\eta B\). The analysis therefore cannot
stop at monotonicity of the game gradient or of its skew-corrected variant:
one needs cocoercivity, strong monotonicity combined with Lipschitz continuity,
or another condition implying nonexpansiveness/averagedness of the actual
iteration map.

We first study an abstract problem in a real Hilbert space: the computation of
zeros of a nonlinear operator
\[
\mathcal F:U\subset\mathcal H\to\mathcal H,
\qquad
\operatorname{Zer}(\mathcal F):=\{z\in U:\mathcal F(z)=0\}.
\]
The operator \(\mathcal F\) is assumed to be Fr\'echet differentiable on
\(U\), and we denote its Fr\'echet derivative by
\[
\mathcal J(z):=D\mathcal F(z)\in\mathcal L(\mathcal H).
\]
This infinite-dimensional point of view is useful because the convergence
mechanism behind SGA is a metric property of a skew-corrected operator, rather
than a specifically finite-dimensional matrix calculation. Given a
skew-adjoint correction
\[
\mathcal A(z)^*=-\mathcal A(z),
\qquad z\in U,
\]
we introduce
\[
\Gamma_\tau(z):=(\operatorname{Id}-\tau\mathcal A(z))\mathcal F(z).
\]
The corresponding fixed point map which we will study is
$T_{\eta,\tau}:=\operatorname{Id}-\eta\Gamma_\tau.$\\
Since \(\mathcal A(z)^*=-\mathcal A(z)\), we have
\[
\|(\operatorname{Id}-\tau\mathcal A(z))h\|^2
=
\|h\|^2+\tau^2\|\mathcal A(z)h\|^2 \quad\forall h\in\mathcal H,
\]
so \(\operatorname{Id}-\tau\mathcal A(z)\) is injective for every \(z\). Hence
the zero sets coincide:
\begin{equation}\label{eq:ZerGammaF}
\operatorname{Zer}(\Gamma_\tau)=\operatorname{Zer}(\mathcal F).
\end{equation}
Therefore zeros of \(\mathcal F\) can be computed as fixed points of
\(T_{\eta,\tau}\), provided one can prove metric properties of
\(T_{\eta,\tau}\). We obtain such properties from assumptions on
\(\Gamma_\tau\), especially cocoercivity and coercivity/Lipschitz estimates.
We recall the general
Hilbert-space tools first; then we show how the SGA results for finite
dimensional games follow by taking \(\mathcal H=\mathbb R^{m+n}\),
\(\mathcal F=F\), and \(\mathcal J=H=DF\). See, e.g.,
\cite{bauschke2017} for the fixed point background.

\begin{definition}
\label{def:metric-operators-hilbert}
Let \(\mathcal H\) be a real Hilbert space, let \(C\subseteq\mathcal H\), and
let \(T:C\to\mathcal H\). We use the following terminology.
Each property is said to hold locally at \(\bar z\in C\) if it holds on
\(C\cap U\) for some neighbourhood \(U\) of \(\bar z\).
\begin{itemize}
    \item[(i)] \(T\) is monotone if
    \[
    \langle Tz-Tv,z-v\rangle\geq0,
    \qquad z,v\in C.
    \]
    \item[(ii)] \(T\) is \(h\)-strongly monotone, or \(h\)-coercive, with \(h>0\), if
    \[
    \langle Tz-Tv,z-v\rangle\geq h\|z-v\|^2,
    \qquad z,v\in C.
    \]
    \item[(iii)] \(T\) is \(L\)-Lipschitz continuous, with \(L\geq0\), if
    \[
    \|Tz-Tv\|\leq L\|z-v\|,
    \qquad z,v\in C.
    \]
    In particular, the case \(L=1\) is called nonexpansive, while the case
    \(L<1\) is called contractive.
    \item[(iv)] \(T\) is \(\alpha\)-averaged nonexpansive, with
    \(\alpha\in(0,1)\), if there exists a nonexpansive operator
    \(R:C\to C\) such that
    \[
    T=(1-\alpha)\operatorname{Id}+\alpha R .
    \]
    \item[(v)] \(T\) is \(\beta\)-cocoercive, or
    \(\beta\)-inversely strongly monotone, with \(\beta>0\), if
    \[
    \langle Tz-Tv,z-v\rangle\geq
    \beta\|Tz-Tv\|^2,
    \qquad z,v\in C.
    \]
\end{itemize}
\end{definition}

\begin{proposition}
\label{prop:averaged-basic-facts}
Let \(C\) be a nonempty closed convex subset of a real Hilbert space
\(\mathcal H\), and let \(T:C\to C\) be \(\alpha\)-averaged, with
\(\alpha\in(0,1)\). Then \(T\) is nonexpansive and, for every
\(p\in\operatorname{Fix}(T)\),
\[
\|Tx-p\|^2
\leq
\|x-p\|^2-\frac{1-\alpha}{\alpha}\|x-Tx\|^2,
\qquad x\in C.
\]
Consequently, the Picard iteration \(x_{k+1}=Tx_k\) is Fej\'er monotone with
respect to \(\operatorname{Fix}(T)\), whenever this set is nonempty.
\end{proposition}

\textcolor{revblue}{This is a standard property of averaged nonexpansive mappings;
see, e.g., \cite[Proposition~4.25]{bauschke2017}.}

\begin{proposition}
\label{prop:cocoercive-averaged}
Let \(B:C\to\mathcal H\) be \(\beta\)-cocoercive on \(C\). Then, for every
\(\eta\in(0,2\beta)\) such that the forward step
\[
T_\eta:=\operatorname{Id}-\eta B
\]
maps \(C\) into \(C\), the operator \(T_\eta\) is \(\alpha\)-averaged on \(C\),
with
\[
\alpha=\frac{\eta}{2\beta}.
\]
In particular, \(T_\eta\) is nonexpansive. Moreover,
\[
\operatorname{Fix}(T_\eta)=\operatorname{Zer}(B).
\]
\end{proposition}

\textcolor{revblue}{Here we use the usual forward-step averaging result for
cocoercive operators, together with the fixed-point identity for
\(\operatorname{Id}-\eta B\); see, e.g., \cite[Chapter~4]{bauschke2017}.}

\begin{proposition}
\label{prop:avnonexp}
Let \(B:C\to\mathcal H\) be \(L\)-Lipschitz continuous and
\(h\)-strongly monotone. Then, for every
\[
\eta\in\left(0,\frac{2h}{L^2}\right),
\]
the operator \(T_\eta:=\operatorname{Id}-\eta B\) is contractive. More
precisely,
\begin{equation}
\label{eq:contractivity-coefficient}
\|T_\eta z-T_\eta v\|
\leq
\sqrt{1-\eta(2h-\eta L^2)}\,\|z-v\|,
\qquad z,v\in C.
\end{equation}
Consequently, if \(C\subset \mathcal H\) is nonempty, closed, and convex, and
\(T_\eta\) maps \(C\) into itself, then, for every \(x_0\in C\), the iteration
\(x_{k+1}=T_\eta x_k\) converges strongly and linearly to the unique fixed point
of \(T_\eta\) in \(C\), equivalently to the unique zero of \(B\) in \(C\).
\end{proposition}

\textcolor{revblue}{The estimate follows by combining strong monotonicity and
Lipschitz continuity, while the convergence statement is an application of
the Banach fixed point theorem; see, e.g., \cite{kirk}.}

\begin{theorem}
\label{thm:averaged-convergence-rate}
Let \(C\) be a nonempty closed convex subset of a real Hilbert space
\(\mathcal H\), and let \(T:C\to C\) be \(\alpha\)-averaged, with
\(\alpha\in(0,1)\). Assume that \(\operatorname{Fix}(T)\neq\varnothing\),
and define \(x_{k+1}=Tx_k\). Then \(\{x_k\}\) converges weakly to a point
\(p\in\operatorname{Fix}(T)\). Moreover,  the following sharp pointwise residual estimate holds for every
\(N\geq1\),
\begin{equation}
\label{eq:bravo-cominetti-residual}
\|x_N-Tx_N\|
\leq
\frac{2}{\sqrt{\pi}}\,
\frac{\alpha\,d(x_0,\operatorname{Fix}(T))}
{\sqrt{N\alpha(1-\alpha)}}.
\end{equation}
The universal constant \(1/\sqrt{\pi}\) is best possible (see
\cite{bravo2018sharp}). If \(\mathcal H\) is finite-dimensional, then the
convergence is strong.
\end{theorem}

We now state the local estimates for the skew-corrected zero problem in the
same Hilbert space setting. Let \(\mathcal H\) be a real Hilbert space, let
\(U\subset \mathcal H\) be open, and let
\[
\mathcal F:U\to\mathcal H,
\qquad
\mathcal A:U\to\mathcal L(\mathcal H),
\qquad
\mathcal A(z)^*=-\mathcal A(z).
\]
We assume that \(\mathcal F\) is Fr\'echet differentiable and write
\[
\mathcal J(z):=D\mathcal F(z).
\]
The variable skew-corrected operator is
\[
\Gamma_\tau(z):=(\operatorname{Id}-\tau \mathcal A(z))\mathcal F(z).
\]
The finite-dimensional SGA setting will be recovered later by taking
\(\mathcal F=F\), \(\mathcal J=H=DF\), and
\[
H(w)=S(w)+A(w),
\qquad
S(w):=\frac{H(w)+H(w)^T}{2},
\qquad
A(w):=\frac{H(w)-H(w)^T}{2}.
\]

Before applying the metric fixed point tools to this skew-corrected operator,
we record a caution specific to the game setting: the usual second-order
semidefiniteness conditions for Nash equilibria, even together with positive
semidefiniteness of the game Hessian in quadratic-form sense, do not by
themselves imply local cocoercivity of \(\Gamma_\tau\).

\begin{remark}
\label{rem:variable-skew-warning}
The assumptions \(\mathcal S(z)\succeq0\), \(\mathcal A(z)^*=-\mathcal A(z)\),
and \(\mathcal F(z_*)=0\) do not by themselves imply local monotonicity of the
nonlinear field
\(\Gamma_\tau(z)=(\operatorname{Id}-\tau\mathcal A(z))\mathcal F(z)\). In the
variable-skew case the derivative
contains the additional perturbation
\[
D\Gamma_\tau(z)[h]
=(\operatorname{Id}-\tau\mathcal A(z))D\mathcal F(z)h
-\tau D\mathcal A(z)[h]\mathcal F(z),
\]
	and the term \(D\mathcal A(z)[h]\mathcal F(z)\) must be controlled. This obstruction already
	occurs in finite-dimensional examples. Indeed, the following example shows that the structural conditions
		\(\partial_{xx}^2f\succeq0\), \(\partial_{yy}^2g\succeq0\), and
		positive semidefiniteness of the game Hessian in the sense of its
		quadratic form do not imply local cocoercivity of the skew-corrected
		field with variable correction.
	\end{remark}

		\begin{example}
		\label{ex:counterexample-variable-skew}
		Consider the two-player game
		\[
		f(x,y)=\frac{x^2}{2}+xy,
		\qquad
		g(x,y)=xy-\frac{x^3y}{3}+\frac{y^2}{2}.
		\]
		Its game gradient \(\mathcal F:\mathbb R^2\to\mathbb R^2\) is
		\[
		\mathcal F(x,y)=
		\begin{pmatrix}
		x+y\\
		x-\frac{x^3}{3}+y
		\end{pmatrix},
		\qquad \mathcal F(0,0)=0,
		\]
		since \(\partial_x f=\mathcal F_1\) and
		\(\partial_y g=\mathcal F_2\). Moreover, for \(y=0\),
		\(f(x,0)=x^2/2\), and for \(x=0\), \(g(0,y)=y^2/2\). Hence
		\((0,0)\) is a Nash equilibrium in the sense of
		Definition~\ref{DefNE}. Also,
		\[
		\partial_{xx}^2f(x,y)=1,
		\qquad
		\partial_{yy}^2g(x,y)=1.
		\]
		In particular, \(\partial_{xx}^2f(x,y)\succeq0\) and
		\(\partial_{yy}^2g(x,y)\succeq0\).
	The Fr\'echet derivative is
	\[
	\mathcal J(x,y):=D\mathcal F(x,y)=
	\begin{pmatrix}
	1&1\\
	1-x^2&1
	\end{pmatrix}.
	\]
	Writing \(\mathcal J=\mathcal S+\mathcal A\), one has
	\[
	\mathcal S(x,y)=
	\begin{pmatrix}
	1&1-\frac{x^2}{2}\\
	1-\frac{x^2}{2}&1
	\end{pmatrix},
	\qquad
	\mathcal A(x,y)=
	\begin{pmatrix}
	0&\frac{x^2}{2}\\
	-\frac{x^2}{2}&0
	\end{pmatrix}.
	\]
	The eigenvalues of \(\mathcal S(x,y)\) are
	\[
	2-\frac{x^2}{2},
	\qquad
	\frac{x^2}{2},
	\]
	so \(\mathcal S(x,y)\succeq0\) near the origin, and
	\(\mathcal A(x,y)^T=-\mathcal A(x,y)\). Since the skew part does not
	contribute to the quadratic form, for every \(v\in\mathbb R^2\),
	\[
	\langle \mathcal J(x,y)v,v\rangle
	=
	\langle \mathcal S(x,y)v,v\rangle .
	\]
	Hence \(\mathcal J(x,y)\) is positive semidefinite, in the sense of its
	quadratic form, for all \((x,y)\) with \(|x|\leq2\). Fix
	\(\tau>0\) and set
	\[
	\Gamma_\tau(x,y):=(I-\tau \mathcal A(x,y))\mathcal F(x,y).
	\]
	Choose \(k>\frac{1}{2\tau}-1\), define \(u_x=(x,kx)\), and take
	\(h=(1,-1)\). Then \(u_x\to0\) as \(x\to0\), while the product rule gives
	\[
	D\Gamma_\tau(u)[h]
	=(I-\tau \mathcal A(u))\mathcal J(u)h
	-\tau D\mathcal A(u)[h]\mathcal F(u).
	\]
	A direct computation yields
	\[
	\mathcal J(u_x)h=
	\begin{pmatrix}
	0\\
	-x^2
	\end{pmatrix},
	\qquad
	(I-\tau \mathcal A(u_x))\mathcal J(u_x)h=
	\begin{pmatrix}
	\frac{\tau x^4}{2}\\
	-x^2
	\end{pmatrix},
	\]
	and
	\[
	D\mathcal A(u_x)[h]\mathcal F(u_x)=
	\begin{pmatrix}
	(1+k)x^2-\frac{x^4}{3}\\
	-(1+k)x^2
	\end{pmatrix}.
	\]
	Hence
	\[
	\left\langle D\Gamma_\tau(u_x)[h],h\right\rangle
	=
	x^2\bigl(1-2\tau(1+k)\bigr)+\frac56\tau x^4 .
	\]
	Since \(k>\frac{1}{2\tau}-1\), the coefficient
	\(1-2\tau(1+k)\) is negative, and therefore
	\[
	\left\langle D\Gamma_\tau(u_x)[h],h\right\rangle<0
	\]
	for all sufficiently small nonzero \(x\). Thus, in every neighbourhood of
	the origin, the differential of \(\Gamma_\tau\) has a negative quadratic
	direction. Because a \(C^1\) monotone map must satisfy
	\(\langle D\Gamma_\tau(u)[h],h\rangle\geq0\), \(\Gamma_\tau\) is not locally
	monotone near the origin. Consequently, it is not locally cocoercive. This
	shows that one cannot hope to guarantee local cocoercivity without controlling the term
	\(D\mathcal A(z)[h]\mathcal F(z)\).
	\end{example}

We now apply the metric fixed point tools recalled above to
\(\Gamma_\tau\). We first record a basic Lipschitz estimate, and then state
additional sufficient assumptions under which the skew-corrected field becomes
cocoercive or coercive.

\begin{lemma}
\label{lem:local-lipschitz-gammatau}
Assume that \(\mathcal F\) and \(\mathcal A\) are Lipschitz continuous on \(U\),
with constants \(L_{\mathcal F}\) and \(L_{\mathcal A}\), respectively.
Set
\[
M_{\mathcal F}:=\sup_{z\in U}\|\mathcal F(z)\|,
\qquad
M_{\mathcal A}:=\sup_{z\in U}\|\mathcal A(z)\|_{\mathcal L(\mathcal H)},
\]
and assume that \(M_{\mathcal F},M_{\mathcal A}<+\infty\). Then
\(\Gamma_\tau\) is Lipschitz continuous on \(U\). More precisely,
\[
\|\Gamma_\tau(z)-\Gamma_\tau(v)\|
\leq L_{\Gamma}\|z-v\|,
\qquad z,v\in U,
\]
where
\[
L_{\Gamma}:=(1+\tau M_{\mathcal A})L_{\mathcal F}
+\tau L_{\mathcal A} M_{\mathcal F}.
\]
\end{lemma}

\begin{proof}
For \(z,v\in U\), adding and subtracting
\((\operatorname{Id}-\tau \mathcal A(z))\mathcal F(v)\), we obtain
\[
\Gamma_\tau(z)-\Gamma_\tau(v)
=(\operatorname{Id}-\tau \mathcal A(z))(\mathcal F(z)-\mathcal F(v))
-\tau(\mathcal A(z)-\mathcal A(v))\mathcal F(v).
\]
Taking norms gives the assertion.
\end{proof}

\textcolor{revblue}{In view of Example~\ref{ex:counterexample-variable-skew},
local cocoercivity of the skew-corrected field is not a consequence of the
basic second-order semidefiniteness conditions for a Nash equilibrium. 
Thus, as previously stated, one needs to control the perturbation term \(D\mathcal A(z)[h]\mathcal F(z)\) and the following lemma gives sufficient conditions for local cocoercivity and coercivity of \(\Gamma_\tau\) in terms of relative bounds on this perturbation.}

\begin{lemma}
\label{lem:variable-skew-local-cocoercivity-relative}
Let \(\mathcal H\) be a real Hilbert space, let \(U\subset\mathcal H\) be
open, and let \(\mathcal F\in C^1(U;\mathcal H)\). Set
\[
\mathcal J(z):=D\mathcal F(z).
\]
For \(z\in U\), define
\[
\mathcal S(z):=\frac{\mathcal J(z)+\mathcal J(z)^*}{2},
\qquad
\mathcal A(z):=\frac{\mathcal J(z)-\mathcal J(z)^*}{2},
\]
and assume that \(\mathcal A\in C^1(U;\mathcal L(\mathcal H))\) and
\[
\mathcal S(z)\succeq0,\qquad
\mathcal A(z)^*=-\mathcal A(z),
\qquad z\in U.
\]
Set
\[
\Gamma_\tau(z):=(\operatorname{Id}-\tau\mathcal A(z))\mathcal F(z).
\]
Let
\[
s_U:=\sup_{u\in U}\|\mathcal S(u)\|_{\mathcal L(\mathcal H)},
\qquad
a_U:=\sup_{u\in U}\|\mathcal A(u)\|_{\mathcal L(\mathcal H)},
\]
and assume that \(s_U,a_U<+\infty\). Choose \(0<\tau<2/s_U\) when
\(s_U>0\), while \(\tau>0\) when \(s_U=0\). For
\(z\in U\), define
\[
M_\tau(z):=(\operatorname{Id}-\tau\mathcal A(z))\mathcal J(z).
\]
Assume that there exist constants \(0\leq\theta<1\) and \(\rho\geq0\) such
that, for every \(z\in U\) and every \(h\in\mathcal H\),
\begin{equation}
\label{eq:relative-perturbation-inner}
\tau\left|\langle \mathcal F(z),D\mathcal A(z)[h]h\rangle\right|
\leq
\theta\,\langle M_\tau(z)h,h\rangle,
\end{equation}
and
\begin{equation}
\label{eq:relative-perturbation-norm}
\tau\|D\mathcal A(z)[h]\mathcal F(z)\|
\leq
\rho\|M_\tau(z)h\|.
\end{equation}
Then \(\Gamma_\tau\) is cocoercive on every convex subset of \(U\). More precisely,
for every \(z,v\in U\) such that \([v,z]\subset U\),
\[
\langle \Gamma_\tau(z)-\Gamma_\tau(v),z-v\rangle
\geq
\beta_\Gamma
\|\Gamma_\tau(z)-\Gamma_\tau(v)\|^2,
\]
where
\[
\beta_\Gamma
=
\frac{(1-\theta)\tau}
{2(1+\rho)^2(1+\tau^2a_U^2)}.
\]
In particular, if \(z_*\in U\) is a zero of \(\mathcal F\) and there exist
\(r>0\) with \(B_r(z_*)\subset U\) and \(\mu_r>0\) such that
\[
\langle \mathcal S(z)h,h\rangle
\geq
\mu_r\|h\|^2,
\qquad
z\in B_r(z_*),\ h\in\mathcal H,
\]
then, after possibly reducing \(r\), the two relative perturbation assumptions
are automatically satisfied on \(B_r(z_*)\), and therefore \(\Gamma_\tau\) is
locally cocoercive on a sufficiently small neighbourhood of \(z_*\). Moreover,
for \(r>0\) small enough, with
\[
M_r:=\sup_{u\in B_r(z_*)}\|\mathcal J(u)\|_{\mathcal L(\mathcal H)},
\qquad
L_r:=\sup_{u\in B_r(z_*)}\|D\mathcal A(u)\|,
\]
and
\[
L_rM_rr<\frac{\mu_r^2}{2},
\]
one has the explicit coercivity estimate
\[
\langle \Gamma_\tau(z)-\Gamma_\tau(v),z-v\rangle
\geq
m_{\Gamma,r}\|z-v\|^2,
\qquad z,v\in B_r(z_*),
\]
where
\[
m_{\Gamma,r}:=\tau\left(\frac{\mu_r^2}{2}-L_rM_rr\right)>0.
\]
\end{lemma}

\begin{proof}
Fix \(z\in U\) and \(h\in\mathcal H\). Since
\[
\Gamma_\tau(z)=(\operatorname{Id}-\tau\mathcal A(z))\mathcal F(z),
\]
we have
\[
D\Gamma_\tau(z)[h]
=
(\operatorname{Id}-\tau\mathcal A(z))\mathcal J(z)h
-\tau D\mathcal A(z)[h]\mathcal F(z)
=M_\tau(z)h-\tau D\mathcal A(z)[h]\mathcal F(z).
\]
Taking the scalar product with \(h\), and using that
\(D\mathcal A(z)[h]^*=-D\mathcal A(z)[h]\), gives
\[
\begin{aligned}
\langle D\Gamma_\tau(z)[h],h\rangle
&=
\langle M_\tau(z)h,h\rangle
-\tau\langle D\mathcal A(z)[h]\mathcal F(z),h\rangle\\
&=
\langle M_\tau(z)h,h\rangle
+\tau\langle \mathcal F(z),D\mathcal A(z)[h]h\rangle .
\end{aligned}
\]
By \eqref{eq:relative-perturbation-inner},
\[
\langle D\Gamma_\tau(z)[h],h\rangle
\geq
(1-\theta)\langle M_\tau(z)h,h\rangle .
\]
Writing
\(\mathcal J(z)=\mathcal S(z)+\mathcal A(z)\) and using
\(\mathcal A(z)^*=-\mathcal A(z)\), so that
\(\langle\mathcal A(z)h,h\rangle=0\) and
\(\langle\mathcal A(z)^2h,h\rangle=-\|\mathcal A(z)h\|^2\), we obtain
\[
\langle M_\tau(z)h,h\rangle
=\langle\mathcal S(z)h,h\rangle
+\tau\langle\mathcal S(z)h,\mathcal A(z)h\rangle
+\tau\|\mathcal A(z)h\|^2 .
\]
If \(\mathcal S(z)=0\), then \(\mathcal J(z)=\mathcal A(z)\) and
\(\langle M_\tau(z)h,h\rangle=\tau\|\mathcal A(z)h\|^2
=\tau\|\mathcal J(z)h\|^2\) for every \(\tau>0\). If
\(\mathcal S(z)\neq0\), then
\(\langle\mathcal S(z)h,h\rangle\geq
\|\mathcal S(z)h\|^2/\|\mathcal S(z)\|_{\mathcal L(\mathcal H)}\).
Indeed, since \(\mathcal S(z)\) is self-adjoint and positive semidefinite, its
spectrum is contained in
\([0,\|\mathcal S(z)\|_{\mathcal L(\mathcal H)}]\). Hence
\(\mathcal S(z)^2\preceq
\|\mathcal S(z)\|_{\mathcal L(\mathcal H)}\mathcal S(z)\), and therefore
\[
    \|\mathcal S(z)h\|^2
    =
    \langle \mathcal S(z)^2h,h\rangle
    \leq
    \|\mathcal S(z)\|_{\mathcal L(\mathcal H)}
    \langle\mathcal S(z)h,h\rangle .
\]
Moreover, since \(0<\tau<2/s_U\) and
\(\|\mathcal S(z)\|_{\mathcal L(\mathcal H)}\leq s_U\), we have
\(0<\tau<2/\|\mathcal S(z)\|_{\mathcal L(\mathcal H)}\). Thus,
\[
\langle M_\tau(z)h,h\rangle
\geq
\frac{\tau}{2}\|\mathcal S(z)h+\mathcal A(z)h\|^2
+\Bigl(\tfrac{1}{\|\mathcal S(z)\|_{\mathcal L(\mathcal H)}}
-\tfrac{\tau}{2}\Bigr)\|\mathcal S(z)h\|^2
+\frac{\tau}{2}\|\mathcal A(z)h\|^2
\geq
\frac{\tau}{2}\|\mathcal J(z)h\|^2 .
\]
Finally, since \(\mathcal A(z)^*=-\mathcal A(z)\),
\[
\|M_\tau(z)h\|^2
=\|\mathcal J(z)h\|^2+\tau^2\|\mathcal A(z)\mathcal J(z)h\|^2
\leq(1+\tau^2a_U^2)\|\mathcal J(z)h\|^2,
\]
and therefore
\[
\langle M_\tau(z)h,h\rangle
\geq
\frac{\tau}{2}\|\mathcal J(z)h\|^2
\geq
\frac{\tau}{2(1+\tau^2a_U^2)}\|M_\tau(z)h\|^2.
\]
On the other hand,
\[
\|D\Gamma_\tau(z)[h]\|
\leq
\|M_\tau(z)h\|
+\tau\|D\mathcal A(z)[h]\mathcal F(z)\|
\leq
(1+\rho)\|M_\tau(z)h\|.
\]
Combining these estimates yields
\[
\langle D\Gamma_\tau(z)[h],h\rangle
\geq
\frac{(1-\theta)\tau}
{2(1+\rho)^2(1+\tau^2a_U^2)}
\|D\Gamma_\tau(z)[h]\|^2
=
\beta_\Gamma\|D\Gamma_\tau(z)[h]\|^2.
\]
Now let \(z,v\in U\) be such that \([v,z]\subset U\), and set \(h:=z-v\).
By the fundamental theorem of calculus,
\[
\Gamma_\tau(z)-\Gamma_\tau(v)
=
\int_0^1D\Gamma_\tau(v+th)[h]\,dt.
\]
Therefore,
\[
\begin{aligned}
\langle \Gamma_\tau(z)-\Gamma_\tau(v),z-v\rangle
&=
\int_0^1
\langle D\Gamma_\tau(v+th)[h],h\rangle\,dt\\
&\geq
\beta_\Gamma
\int_0^1\|D\Gamma_\tau(v+th)[h]\|^2\,dt .
\end{aligned}
\]
By Jensen's inequality,
\[
\int_0^1\|D\Gamma_\tau(v+th)[h]\|^2\,dt
\geq
\left\|
\int_0^1D\Gamma_\tau(v+th)[h]\,dt
\right\|^2
=
\|\Gamma_\tau(z)-\Gamma_\tau(v)\|^2.
\]
This proves the cocoercivity of \(\Gamma_\tau\) on every convex subset of \(U\).

It remains to justify the final assertion. Let \(z_*\in U\) be such that
\(\mathcal F(z_*)=0\), and assume
\[
\langle \mathcal S(z)h,h\rangle\geq\mu_r\|h\|^2,
\qquad z\in B_r(z_*),\ h\in\mathcal H.
\]
Then
\[
\|\mathcal J(z)h\|\,\|h\|
\geq
\langle \mathcal J(z)h,h\rangle
=
\langle \mathcal S(z)h,h\rangle
\geq
\mu_r\|h\|^2,
\]
so
\[
\|\mathcal J(z)h\|\geq\mu_r\|h\|.
\]
Since \(\mathcal F(z_*)=0\) and \(\mathcal F\in C^1(U;\mathcal H)\),
after possibly reducing \(r>0\), there exists
\[
M_r:=\sup_{u\in B_r(z_*)}\|\mathcal J(u)\|_{\mathcal L(\mathcal H)}<+\infty
\]
such that
\[
\|\mathcal F(z)\|
=
\|\mathcal F(z)-\mathcal F(z_*)\|
\leq
M_r\|z-z_*\|
\leq
M_rr.
\]
Moreover, since
\(\mathcal A\in C^1(U;\mathcal L(\mathcal H))\), after
possibly reducing \(r\), there exists
\[
a_r:=\sup_{u\in B_r(z_*)}\|\mathcal A(u)\|_{\mathcal L(\mathcal H)}<+\infty,
\qquad
L_r:=\sup_{u\in B_r(z_*)}\|D\mathcal A(u)\|<+\infty .
\]
Since \(D\mathcal A(z)[h]\) is skew-adjoint,
\[
\left|\langle \mathcal F(z),D\mathcal A(z)[h]h\rangle\right|
\leq
\|\mathcal F(z)\|\,\|D\mathcal A(z)[h]h\|
\leq
L_rM_rr\|h\|^2.
\]
Using \(\|\mathcal J(z)h\|\geq\mu_r\|h\|\), we obtain
\[
\left|\langle \mathcal F(z),D\mathcal A(z)[h]h\rangle\right|
\leq
\frac{L_rM_rr}{\mu_r^2}\|\mathcal J(z)h\|^2.
\]
Using the same estimate as in the first part of the proof, we get
\[
\langle M_\tau(z)h,h\rangle
\geq
\frac{\tau}{2(1+\tau^2a_r^2)}\|M_\tau(z)h\|^2
\geq
\frac{\tau}{2(1+\tau^2a_r^2)}\|\mathcal J(z)h\|^2,
\]
and hence
\[
\tau\left|\langle \mathcal F(z),D\mathcal A(z)[h]h\rangle\right|
\leq
\frac{2(1+\tau^2a_r^2)L_rM_rr}{\mu_r^2}
\langle M_\tau(z)h,h\rangle.
\]
Thus the first relative perturbation condition holds with
\(\theta_r:=2(1+\tau^2a_r^2)L_rM_rr/\mu_r^2\), which is \(<1\)
for \(r\) small enough.

For the second condition, we estimate
\[
\|D\mathcal A(z)[h]\mathcal F(z)\|
\leq
L_r\|h\|\,\|\mathcal F(z)\|
\leq
L_rM_rr\|h\|
\leq
\frac{L_rM_rr}{\mu_r}\|\mathcal J(z)h\|.
\]
Since \(\mathcal A(z)^*=-\mathcal A(z)\),
\[
\|M_\tau(z)h\|
=
\|(\operatorname{Id}-\tau\mathcal A(z))\mathcal J(z)h\|
\geq
\|\mathcal J(z)h\|.
\]
Therefore
\[
\tau\|D\mathcal A(z)[h]\mathcal F(z)\|
\leq
\rho_r\|M_\tau(z)h\|,
\qquad
\rho_r:=\frac{\tau L_rM_rr}{\mu_r}.
\]
Thus both relative perturbation assumptions are satisfied on \(B_r(z_*)\) for
\(r>0\) sufficiently small. Applying the first part of the lemma on
\(B_r(z_*)\), with \(\theta=\theta_r\) and \(\rho=\rho_r\), gives local
cocoercivity. The stated coercivity constant follows directly from the
differential estimate: the elementary fixed-\(z\) algebraic estimate gives
\[
\left\langle(\operatorname{Id}-\tau\mathcal A(z))\mathcal J(z)h,h\right\rangle
\geq
\frac{\tau}{2}\|\mathcal J(z)h\|^2,
\]
while
\[
\left|\langle D\mathcal A(z)[h]\mathcal F(z),h\rangle\right|
\leq
L_rM_rr\|h\|^2.
\]
Since \(\|\mathcal J(z)h\|\geq\mu_r\|h\|\), we get
\[
\langle D\Gamma_\tau(z)[h],h\rangle
\geq
\tau\left(\frac{\mu_r^2}{2}-L_rM_rr\right)\|h\|^2
=m_{\Gamma,r}\|h\|^2.
\]
Integrating along the segment joining \(v\) and \(z\) yields the estimate with
constant \(m_{\Gamma,r}\).
Thus \(\Gamma_\tau\) is locally coercive on \(B_r(z_*)\).
\end{proof}

\textcolor{revblue}{The next theorem applies the metric fixed-point result to
\(\Gamma_\tau\) under the relative perturbation assumptions
\eqref{eq:relative-perturbation-inner}--\eqref{eq:relative-perturbation-norm}.}

	\begin{theorem}
	\label{thm:variable-skew-cocoercive-weak-convergence}
	In the setting and notation of
	Lemma~\ref{lem:variable-skew-local-cocoercivity-relative}, assume that the
	hypotheses in the first part of that lemma hold on \(U\), including
	\eqref{eq:relative-perturbation-inner} and
	\eqref{eq:relative-perturbation-norm}, with constants
	\(0\leq\theta<1\) and \(\rho\geq0\). Let
	\[
	a_U:=\sup_{u\in U}\|\mathcal A(u)\|_{\mathcal L(\mathcal H)}
	\]
	and set
	\[
	\beta_\Gamma
	=
	\frac{(1-\theta)\tau}
	{2(1+\rho)^2(1+\tau^2a_U^2)}.
	\]
	Let \(C\subset U\) be nonempty, closed, and convex, and assume that
	\[
	\textcolor{revblue}{\operatorname{Zer}(\mathcal F)\cap C\neq\varnothing.}
	\]
	For every \(0<\eta<2\beta_\Gamma\), set
	\[
	T_{\eta,\tau}:=\operatorname{Id}-\eta\Gamma_\tau.
	\]
	If \(T_{\eta,\tau}(C)\subset C\), then \(T_{\eta,\tau}\) is averaged
	nonexpansive on \(C\). Consequently, for every \(z_0\in C\), the
	iteration
	\[
	z_{k+1}=T_{\eta,\tau}z_k
	\]
	converges weakly to a point
	\(
	\bar z\in\textcolor{revblue}{\operatorname{Zer}(\mathcal F)}\cap C.
	\) Moreover, for every \(N\geq1\) the pointwise estimate
	\begin{equation}
        \label{eq:residual}
	\|\mathcal F (z_N)\|
	\leq
	\frac{2d(z_0,\operatorname{Zer}(\mathcal F)\cap C)}{\sqrt{\eta\,\pi(2\beta_\Gamma-\eta)N}}
	\end{equation}
	holds. If \(\mathcal H\) is finite-dimensional, then the convergence is
	strong.
	\end{theorem}

	\begin{proof}
	\textcolor{revblue}{By the zero-set identity~\eqref{eq:ZerGammaF},
	\(\operatorname{Zer}(\Gamma_\tau)=\operatorname{Zer}(\mathcal F)\); in
	particular the assumption \(\operatorname{Zer}(\mathcal F)\cap C\neq\varnothing\)
	is equivalent to \(\operatorname{Zer}(\Gamma_\tau)\cap C\neq\varnothing\), and
	the two zero sets may be used interchangeably below.}
	By Lemma~\ref{lem:variable-skew-local-cocoercivity-relative},
	\(\Gamma_\tau\) is \(\beta_\Gamma\)-cocoercive on \(C\).
	Therefore, by Proposition~\ref{prop:cocoercive-averaged},
	\(T_{\eta,\tau}\) is \(\alpha\)-averaged on \(C\), with
	\[
	\alpha=\frac{\eta}{2\beta_\Gamma},
	\]
	and
	\[
	\operatorname{Fix}(T_{\eta,\tau})
	=
	\operatorname{Zer}(\Gamma_\tau).
	\]
	The convergence follows from Theorem~\ref{thm:averaged-convergence-rate}.
	Since
	\[
	z_k-T_{\eta,\tau}z_k
	=
	\eta\Gamma_\tau(z_k),
	\]
	the stated bound \eqref{eq:residual} follows from
	\eqref{eq:bravo-cominetti-residual}, divided by \(\eta\), with
	\(\alpha=\eta/(2\beta_\Gamma)\) and the fact that $\|\Gamma_\tau(z)\|\geq\|\mathcal F(z)\|$.
	Finally, since \(\Gamma_\tau(\bar z)=0\), the zero-set identity
	\eqref{eq:ZerGammaF} gives \(\mathcal F(\bar z)=0\).
	\end{proof}
    \begin{remark}
        Estimate \eqref{eq:residual} is optimized by choosing
        \(
                \eta=\beta_\Gamma .
        \)
        Indeed, for such value of the stepsize, the
        estimate reduces to
        \[
                \|\mathcal F(z_N)\|
                \leq
                \frac{2d(z_0,\operatorname{Zer}(\mathcal F)\cap C)}{\beta_\Gamma\sqrt{\pi N}},
                \qquad N\geq 1 .
        \]
       
    \end{remark}
	\begin{theorem}
	\label{thm:local-convergence-cocoercive-sga}
	In the setting and notation of Lemma~\ref{lem:variable-skew-local-cocoercivity-relative},
	assume that the hypotheses of its local coercivity alternative hold at
	\(z_*\), and let \(m_{\Gamma,r}>0\) be the coercivity
	constant defined there. Thus
	\[
	\Gamma_\tau(z):=(\operatorname{Id}-\tau\mathcal A(z))\mathcal F(z)
	\]
	is \(m_{\Gamma,r}\)-coercive on \(B_r(z_*)\).
	Let \(C\subset B_r(z_*)\) be nonempty, closed, and convex. Assume that
	\(\Gamma_\tau\) is \(L_\Gamma\)-Lipschitz on \(C\).
	For every
	\[
	0<\eta<\frac{2m_{\Gamma,r}}{L_\Gamma^2},
	\]
	set
	\[
	T_{\eta,\tau}:=\operatorname{Id}-\eta\Gamma_\tau.
	\]
	If \(T_{\eta,\tau}(C)\subset C\), then \(T_{\eta,\tau}\) is a contraction on
	\(C\). Consequently, for every \(z_0\in C\), the iteration
	\[
	z_{k+1}=T_{\eta,\tau}z_k
	\]
	converges strongly and linearly to $z_*$,
	\textcolor{revblue}{that is, the unique zero of \(\mathcal F\) in \(C\).}
	\end{theorem}

	\begin{proof}
	\textcolor{revblue}{By the zero-set identity~\eqref{eq:ZerGammaF},
	\(\operatorname{Zer}(\Gamma_\tau)=\operatorname{Zer}(\mathcal F)\); in
	particular  \(\operatorname{Zer}(\Gamma_\tau)=\{z_*\}\) on \(B_r(z_*)\), and
	the two zero sets may be used interchangeably below.}
	By the local coercivity alternative in
	Lemma~\ref{lem:variable-skew-local-cocoercivity-relative}, applied at
	\(z_*\), we have
	\[
	\langle \Gamma_\tau(z)-\Gamma_\tau(v),z-v\rangle
	\geq
	m_{\Gamma,r}\|z-v\|^2,
	\qquad z,v\in B_r(z_*).
	\]
	Since \(\Gamma_\tau\) is \(L_\Gamma\)-Lipschitz on \(C\), Proposition
	\ref{prop:avnonexp} gives
	\[
	\|T_{\eta,\tau}z-T_{\eta,\tau}v\|
	\leq
	\sqrt{1-\eta(2m_{\Gamma,r}-\eta L_\Gamma^2)}\|z-v\|,
	\qquad z,v\in C.
	\]
    Thus \(T_{\eta,\tau}\) is a contraction on \(C\). Since \(C\) is closed, the Banach fixed point theorem yields strong linear convergence to a unique fixed point of \(T_{\eta,\tau}\) in \(C\).
	Moreover, \(\Gamma_\tau(z_*)=0\) and, finally, \(\mathcal F(z_*)=0\) by~\eqref{eq:ZerGammaF}.
	\end{proof}

\begin{theorem}
\label{thm:frozen-skew-local-convergence-hilbert}
Let \(U\subset\mathcal H\) be open, let
\(\mathcal F:U\to\mathcal H\) be \(C^1\), let
\(\mathcal A:U\to\mathcal L(\mathcal H)\) satisfy
\(\mathcal A(z)^*=-\mathcal A(z)\), and let
\(z_*\in U\) be such that \(\mathcal F(z_*)=0\). Set
\(\mathcal A_*:=\mathcal A(z_*)\) and define the
frozen-skew operator
\[
\Gamma_{\tau,*}(z):=(\operatorname{Id}-\tau\mathcal A_*)\mathcal F(z),
\qquad
T_{\eta,\tau}^*z:=z-\eta\Gamma_{\tau,*}(z).
\]
Assume that
\[
\sigma:=\|DT_{\eta,\tau}^*(z_*)\|_{\mathcal L(\mathcal H)}<1 .
\]
Then there exist \(R>0\) and \(q\in(\sigma,1)\) such that
\(\overline B_R(z_*)\subset U\) and
\begin{equation}\label{eq:qcontraction}
\|T_{\eta,\tau}^*z-T_{\eta,\tau}^*v\|\leq q\|z-v\|,
\qquad z,v\in \overline B_R(z_*).
\end{equation}
Consequently, the iteration
\[
v_{k+1}=T_{\eta,\tau}^*v_k
\]
converges linearly to \(z_*\) for every \(v_0\in \overline B_R(z_*)\).
\end{theorem}

\begin{proof}
Since \(\mathcal A_*=\mathcal A(z_*)\) is fixed, the frozen-skew field has
derivative
\[
D\Gamma_{\tau,*}(z)
=
(\operatorname{Id}-\tau\mathcal A_*)D\mathcal F(z),
\]
and hence
\[
DT_{\eta,\tau}^*(z)
=
\operatorname{Id}
-\eta(\operatorname{Id}-\tau\mathcal A_*)D\mathcal F(z).
\]
Moreover, \(\mathcal F(z_*)=0\) gives
\(\Gamma_{\tau,*}(z_*)=0\), and therefore
\(T_{\eta,\tau}^*z_*=z_*\). Since \(T_{\eta,\tau}^*\) is \(C^1\) and
\(\|DT_{\eta,\tau}^*(z_*)\|=\sigma<1\), there are \(R>0\), with
\(\overline B_R(z_*)\subset U\), and \(q\in(\sigma,1)\) such that
\[
\|DT_{\eta,\tau}^*(z)\|_{\mathcal L(\mathcal H)}\leq q,
\qquad z\in \overline B_R(z_*).
\]
For \(z,v\in \overline B_R(z_*)\), the whole segment
\(v+t(z-v)\), \(t\in[0,1]\), is contained in \(\overline B_R(z_*)\). Thus the
mean-value estimate gives
\[
\|T_{\eta,\tau}^*z-T_{\eta,\tau}^*v\|
\leq
\int_0^1
\|DT_{\eta,\tau}^*(v+t(z-v))(z-v)\|\,dt
\leq q\|z-v\|.
\]
Taking \(v=z_*\) and using \(T_{\eta,\tau}^*z_*=z_*\), we get, for every
\(z\in \overline B_R(z_*)\),
\[
\|T_{\eta,\tau}^*z-z_*\|
\leq q\|z-z_*\|
\leq qR<R.
\]
Thus
\[
T_{\eta,\tau}^*: \overline B_R(z_*)\to \overline B_R(z_*),
\qquad
T_{\eta,\tau}^*z:=z-\eta\Gamma_{\tau,*}(z),
\]
is well-defined on \(\overline B_R(z_*)\) and maps it into itself. Since
\(\overline B_R(z_*)\) is a closed subset of the Hilbert space \(\mathcal H\),
it is complete. The Banach fixed point theorem applied to \(T_{\eta,\tau}^*\)
yields a unique
fixed point in \(\overline B_R(z_*)\) and linear convergence of the iterates to
it. Since \(z_*\) is already a fixed point, this unique fixed point is \(z_*\).
\end{proof}

\end{color}

\begin{color}{revblue}
We now specialize the results proved above to a quadratic game on the
finite-dimensional space \(\mathcal H=\mathbb R^{m+n}\), with \(\mathcal F=F\) and
\(\mathcal J=H=DF\).  For a quadratic game the game Hessian is
constant and the game gradient is affine, \(F(w)=Hw+b\); hence the skew-corrected
operator \(\Gamma_\tau(w)=(I-\tau A)F(w)\) is affine, and the metric estimates of
the previous lemmas, established pointwise, hold globally. Specializing those
estimates, Lemma~\ref{lma:metricproperties} records the cocoercivity and
strong monotonicity of the linear map \(z\mapsto(I-\tau A)Hz\), and
Proposition~\ref{prop:quadratic} then gives global linear convergence of the SGA
method from any initial point in \(\mathbb R^{m+n}\). In this part \(H=H(w)\)
denotes the game Hessian at a fixed point \(w\in\mathbb R^{m+n}\), and
\[
H=S+A,
\qquad
S=\frac{H+H^T}{2},
\qquad
A=\frac{H-H^T}{2}.
\]

\begin{lemma}
\label{lma:metricproperties}
Assume that \(H\succeq0\). The following properties hold.
\begin{itemize}
\item[(a)] If
\[
0<\tau<\frac{2}{\|S\|_2}
\]
when \(S\neq0\), and \(\tau>0\) when \(S=0\), then the linear map
\[
z\mapsto (I-\tau A)Hz
\]
is \(\kappa\)-cocoercive
with
\[
\kappa=\frac{\tau}{2(1+\tau^2\|A\|_2^2)}.
\]
\item[(b)] Assume, in addition, that either
\textcolor{revblue}{\(\lambda_{\min}(S)>0\)  and
\(0<\tau<\dfrac{2\lambda_{\min}(S)}{\|S\|_2^2}\), or \(S=0\), \(H\) is invertible
and \(\tau>0\).} Then
\[
z\mapsto (I-\tau A)Hz
\]
is \(h\)-strongly monotone with
\[
h=\frac{\tau}{2}\sigma_{\min}^2,
\]
where \textcolor{revblue}{\(\lambda_{\min}(S)\) is the smallest eigenvalue of the
symmetric matrix \(S\), and} \(\sigma_{\min}>0\) is the smallest singular value
of \(H\).
\end{itemize}
\end{lemma}

\begin{proof}
\textcolor{revblue}{Since \(A^T=-A\), we have \(\langle Az,z\rangle=0\) and
\(\langle A^2z,z\rangle=-\|Az\|_2^2\); hence, writing \(H=S+A\), the algebraic
decomposition gives}
\[
\langle (I-\tau A)Hz,z\rangle
=\langle Sz,z\rangle+\tau\langle Sz,Az\rangle+\tau\|Az\|_2^2 .
\]
\textcolor{revblue}{We use this identity in both parts.}
For part (a), if \(S=0\), then \(H=A\), and therefore
\[
\langle (I-\tau A)Hz,z\rangle
=\tau\|Az\|_2^2
=\tau\|Hz\|_2^2 .
\]
If \(S\neq0\), then \(S\succeq0\) gives
\(\langle Sz,z\rangle\geq\|Sz\|_2^2/\|S\|_2\). Hence
\[
\begin{aligned}
\langle (I-\tau A)Hz,z\rangle
&\geq
\frac{\tau}{2}\|Sz+Az\|_2^2
+\left(\frac{1}{\|S\|_2}-\frac{\tau}{2}\right)\|Sz\|_2^2
+\frac{\tau}{2}\|Az\|_2^2\\
&\geq
\frac{\tau}{2}\|Hz\|_2^2 .
\end{aligned}
\]
In both cases,
\[
\|(I-\tau A)Hz\|_2^2
\leq
(1+\tau^2\|A\|_2^2)\|Hz\|_2^2,
\]
and part (a) follows. \textcolor{revblue}{For part (b), we use the same identity.}
If \(S=0\), then \(H=A\) and
\[
\langle (I-\tau A)Hz,z\rangle=\tau\|Hz\|_2^2
\geq\tau\sigma_{\min}^2\|z\|_2^2.
\]
If \(S\neq0\), then $H$ is invertible and \(\sigma_{\min}>0\). Since \(S\succeq\lambda_{\min}(S)I\), we have
\[
\begin{aligned}
\langle (I-\tau A)Hz,z\rangle
&\geq
\left(\lambda_{\min}(S)-\frac{\tau}{2}\|S\|_2^2\right)\|z\|_2^2
+\frac{\tau}{2}\|Hz\|_2^2
+\frac{\tau}{2}\|Az\|_2^2  \\
&\geq
\frac{\tau}{2}\sigma_{\min}^2\|z\|_2^2,
\end{aligned}
\]
because \(\tau<2\lambda_{\min}(S)/\|S\|_2^2\). This proves the claim.
\end{proof}

As a first consequence, we recover the global convergence result for quadratic
games. In this case the game Hessian is constant and the game gradient is affine,
namely \(F(w)=Hw+b\).

\begin{proposition}\label{prop:quadratic}
Let \(f\) and \(g\) be quadratic functions such that the game Hessian \(H\) is
positive semidefinite and invertible\textcolor{revblue}{, and such that its symmetric
part \(S\) satisfies the assumption of Lemma~\ref{lma:metricproperties}(b)
(that is, \(S\) is positive definite, or \(S=0\))}. Then there exists a unique SNE point
\(w_*\). Moreover, for any \(\tau\) satisfying
Lemma~\ref{lma:metricproperties}(b), and for every
\[
\eta\in
\left(
0,\frac{\tau\sigma_{\min}^2}
{(1+\tau^2\|A\|_2^2)\|H\|_2^2}
\right),
\]
the sequence generated by the SGA method \eqref{eq:sga_operator_Form}, from any
initial point \(w_0\in\mathbb R^{m+n}\), converges linearly to \(w_*\).
\end{proposition}

\begin{proof}
Since \(F(w)=Hw+b\) and \(H\) is invertible, \(F(w)=0\) has a unique solution
\(w_*\). \textcolor{revblue}{Recall that \(H=DF\) has the block form
\(\left(\begin{smallmatrix}\partial_{xx}^2 f & \partial_{xy}^2 f\\
\partial_{yx}^2 g & \partial_{yy}^2 g\end{smallmatrix}\right)\), so its diagonal
blocks are the own-variable Hessians \(\partial_{xx}^2 f\) (player~1) and
\(\partial_{yy}^2 g\) (player~2). Since \(H\succeq0\), both diagonal blocks are
positive semidefinite, hence \(x\mapsto f(x,y_*)\) and \(y\mapsto g(x_*,y)\) are
convex. Together with \(F(w_*)=0\), this makes \(w_*\) a Nash equilibrium, and
the invertibility of \(H\) makes it isolated; therefore \(w_*\) is a SNE point.}
By Lemma~\ref{lma:metricproperties}(b),
\((I-\tau A)H\) is \(h\)-strongly monotone with
\(h=\tau\sigma_{\min}^2/2\). Moreover,
\[
L=\sqrt{1+\tau^2\|A\|_2^2}\,\|H\|_2
\]
is a Lipschitz constant for \(\Gamma_\tau(w)=(I-\tau A)F(w)\). Hence
Proposition~\ref{prop:avnonexp} applies to \(B=\Gamma_\tau\), and
\(T_{\eta,\tau}=I-\eta\Gamma_\tau\) is a contraction whenever
\(\eta\in(0,2h/L^2)\), which is the stated bound. The Banach fixed point theorem
then gives linear convergence to \(w_*\).
\end{proof}

\begin{remark}
The admissible bound
\[
\tau<\frac{2\lambda_{\min}(S)}{\|S\|_2^2}
\]
in Lemma~\ref{lma:metricproperties}(b) only involves the symmetric part
\(S\). Here \(\lambda_{\min}(S)\) is not the minimum of the real parts of the
eigenvalues of the generally nonsymmetric matrix \(H=S+A\). The quantity \(\lambda_{\min}(S)\) can therefore
be estimated directly from the spectrum of \(S\), or by standard Rayleigh
quotient and Gershgorin-type bounds; see, for instance, \cite{golub2013matrix}.
\end{remark}

\end{color}

\section{Convergence analysis of the LRSGA method}
\label{sec:LRSGAtheory}
\textcolor{revblue}{Returning to the finite-dimensional game on \(\mathbb R^{m+n}\),
this section is devoted to the convergence analysis} of the LRSGA
method given by \eqref{eq:def_sga_approx} with the low-rank updates
resulting from \eqref{eq:UpdateMixedDerivativesF} and
\eqref{eq:UpdateMixedDerivativesG}. \textcolor{revblue}{Unlike the global result for
quadratic games of the previous section, here we} provide a local convergence
result stating that the sequence of LRSGA iterates
converges to \textcolor{revblue}{an} SNE point if the \textcolor{revblue}{initial point} is sufficiently
close to this point and \textcolor{revblue}{the mixed second-order derivative approximations are sufficiently accurate}.
\textcolor{revblue}{The proof is based on controlling, along the iterations, the
distance between the low-rank approximation \(\alpha_k\) and the
antisymmetric part \(A(w_*)\) of the game Hessian at the equilibrium.}

\textcolor{revblue}{We first make explicit that a stationary LRSGA step is already
a zero of the game gradient, so the secant update need not be formed in this
case.}

\begin{proposition}
\label{prop:lrsga-fixed-point-zero}
\textcolor{revblue}{Let \(w_{k+1}\) be generated by the LRSGA iteration
\eqref{eq:def_sga_approx}. If \(w_{k+1}=w_k\), then \(F(w_k)=0\).}
\end{proposition}
\begin{proof}
\textcolor{revblue}{Set
\[
    B_k:=\frac{1}{2}(M_k-N_k^T)\in\mathbb R^{m\times n}.
\]
Then the matrix \(\alpha_k\) in \eqref{eq:def_sga_approx} can be written as
\[
    \alpha_k
    =
    \begin{pmatrix}
        0 & B_k\\
        -B_k^T & 0
    \end{pmatrix}.
\]
Indeed,
\[
    \frac{1}{2}(N_k-M_k^T)
    =
    -\left(\frac{1}{2}(M_k-N_k^T)\right)^T
    =
    -B_k^T .
\]
Therefore
\[
    \alpha_k^T
    =
    \begin{pmatrix}
        0 & -B_k\\
        B_k^T & 0
    \end{pmatrix}
    =
    -\alpha_k,
\]
so \(\alpha_k\) is skew-symmetric. Hence 
\eqref{eq:ZerGammaF} applies to the skew correction
\(\mathcal A\equiv\alpha_k\) and \(\mathcal F=F\).
If \(w_{k+1}=w_k\), then \eqref{eq:def_sga_approx} gives
\[
    (I-\tau\alpha_k)F(w_k)=0.
\]
Therefore \(F(w_k)=0\).}
\end{proof}

\textcolor{revblue}{We now turn to the estimates that keep this approximation error
under control. The first step is to bound the secant matrices themselves.}
In the following lemma, we provide bounds to the norms
$\specnorm{\mu_{k+1} - D \partial_x f(w_*)}$ and
$\specnorm{\nu_{k+1} - D \partial_y g(w_*)}$ in terms of
$\specnorm{\mu_{k} - D \partial_x f(w_*)}$,
$\specnorm{\nu_{k} - D \partial_y g(w_*)}$,\\
$\specnorm{w_{k+1} - w_*}$, and $\specnorm{w_k - w_*}$.

\begin{lemma}
    \label{lemma:BoundUpdates}
    Let $\Omega \subseteq \mathbb{R}^{m + n}$ be a convex domain 
    such that $w_k, w_{k+1}, w_* \in \Omega$.  Assume
    that $D \partial_{x} f$, and $D \partial_{y} g$ are Lipschitz
    continuous in $\Omega$ with Lipschitz constants $\lipDdxf$ and
    $\lipDdyg$, respectively. Let $\mu_k \in \mathbb{R}^{m \times (m + n)}$ and $\nu_k \in \mathbb{R}^{n \times (m + n)}$ be given and assume $\mu_{k+1}$ and $\nu_{k+1}$ are computed by \eqref{eq:UpdateMixedDerivativesF} and \eqref{eq:UpdateMixedDerivativesG}. 
    \textcolor{revblue}{Then the following estimates hold:}
    \begin{align}
        \label{eq:UpdatedMixedDerivativeUpperBoundF}
        \specnorm{\mu_{k+1} - D\partial_x f(w_*)}
        &\leq \specnorm{\mu_{k} - D \partial_x f(w_*)}\\ \nonumber
         &\quad + 2 \lipDdxf \max\{\specnorm{w_{k+1} - w_*}, \specnorm{w_k - w_*}\} \\
        \label{eq:UpdatedMixedDerivativeUpperBoundG}
        \specnorm{\nu_{k+1} - D \partial_y g(w_*)}
        &\leq \specnorm{\nu_{k} - D \partial_y g(w_*)}\\ \nonumber
        & \quad + 2 \lipDdyg \max\{\specnorm{w_{k+1} - w_*}, \specnorm{w_k - w_*}\}
    \end{align}
\end{lemma}
\begin{proof}
    \textcolor{revblue}{We show the inequality for \(\mu_{k+1}\); the bound for \(\nu_{k+1}\) is obtained in the same way.}
    The update \eqref{eq:UpdateMixedDerivativesF} of $\mu_k$ can be rewritten as follows: 
    \begin{align*}
        \mu_{k+1} &= \mu_k \left(I_{m+n} - \frac{(w_{k+1} - w_k) (w_{k+1} - w_k)^T}{(w_{k+1} - w_k)^T (w_{k+1} - w_k)} \right) + \frac{(\partial_x f(w_{k+1}) - \partial_x f(w_k)) (w_{k+1} - w_k)^T}{(w_{k+1} - w_k)^T (w_{k+1} - w_k)}.
    \end{align*}
    By adding zero-terms, we obtain
    \begin{align*}
        \specnorm{\mu_{k+1} - D \partial_x f(w_*)}
        &\leq \specnorm{\left(\mu_k - D \partial_x f(w_*) \right) \left(I_{m+n} - \frac{(w_{k+1} - w_k)(w_{k+1} - w_k)^T}{(w_{k+1} - w_k)^T (w_{k+1} - w_k)} \right)} \\
        &\quad+ \specnorm{\left(D \partial_x f(w_k) - D \partial_x f(w_*)\right) \frac{(w_{k+1} - w_k)(w_{k+1} - w_k)^T}{(w_{k+1} - w_k)^T (w_{k+1} - w_k)}} \\
        &\quad+ \specnorm{\frac{ \left(\partial_x f(w_{k+1}) - \partial_x f(w_k)
                - D \partial_x f(w_k) (w_{k+1} - w_k)
            \right)
            (w_{k+1} - w_k)^T}{(w_{k+1} - w_k)^T (w_{k+1} - w_k)}} \\
        &\leq \specnorm{\mu_k - D \partial_x f(w_*) } 
        + \lipDdxf \specnorm{w_k - w_*} 
        + \frac{1}{2} \lipDdxf \specnorm{w_{k+1} - w_k},
    \end{align*}
    where we have used $\specnorm{I_{m+n} - \frac{ss^T}{s^T s}} \leq 1$ for $s \in \mathbb{R}^{m+n}$, and the Lipschitz continuity of $D \partial_x f$.
    From $\specnorm{w_{k+1} - w_k} \leq \specnorm{w_{k+1} - w_*} + \specnorm{w_k - w_*} \leq 2 \max\{\specnorm{w_{k+1} - w_*}, \specnorm{w_k - w_*}\}$, we obtain (\ref{eq:UpdatedMixedDerivativeUpperBoundF}).
\end{proof}

Next, we obtain a bound for $\specnorm{\alpha_k - A(w_*)}$. 

\begin{lemma}
    \label{lemma:PropertiesSpectralNormBlockMatrices}
    Let $w_* \in \Omega$ and assume there is a $\delta > 0$ such that \(\specnorm{\mu_k - D \partial_x f(w_*)} \leq \delta\) and \(\specnorm{\nu_k - D \partial_y g(w_*)} \leq \delta\) for matrices $\mu_k \in \mathbb{R}^{m \times (m+n)}$ and $\nu_k \in \mathbb{R}^{n \times (m+n)}$. Then, the matrix $\alpha_k$ defined in (\ref{eq:def_sga_approx}), with $M_k \in \mathbb{R}^{m \times n}$ \textcolor{revblue}{being the right block of} $\mu_k$ and $N_k \in \mathbb{R}^{n \times m}$ \textcolor{revblue}{being the left block of} $\nu_k$, satisfies the inequality $\specnorm{\alpha_k - A(w_*)} \leq \delta$. 
\end{lemma}
\begin{proof}
    The spectral norm of a matrix $B \in \mathbb{R}^{n \times n}$ is defined by
    \begin{equation*}
        \specnorm{B}^2 = \varrho(B B^T) = \varrho(B^T B) = \max_{i = 1, \ldots, n} \lambda_i(B^T B),
    \end{equation*}
    where $\lambda_i(B^T B)$, $i = 1, \ldots, n$, are the eigenvalues of $B^T B$.

    Thus, for a matrix $C \in \mathbb{R}^{m \times n}$, \textcolor{revblue}{we have}
    \begin{align*}
        \specnorm{\begin{pmatrix} 0 & C \\ -C^T & 0 \end{pmatrix}}^2
        &= \varrho\!\left(\begin{pmatrix} 0 & -C \\ C^T & 0 \end{pmatrix}\begin{pmatrix} 0 & C \\ -C^T & 0 \end{pmatrix}\right)
        = \varrho\begin{pmatrix} C C^T & 0 \\ 0 & C^T C \end{pmatrix} \\
        &= \max\{\varrho(C C^T), \varrho(C^T C)\}
        = \varrho(C C^T) = \specnorm{C}^2,
    \end{align*}
    since the eigenvalues of a block-diagonal matrix are the eigenvalues of the diagonal blocks, and \(\varrho(CC^T)=\varrho(C^TC)\).

    Notice that \(\alpha_k - A(w_*)\) is an anti-symmetric matrix of the form
    \(\textcolor{revblue}{\begin{pmatrix} 0 & C \\ -C^T & 0 \end{pmatrix}}\) with
    $C = \frac{1}{2}(M_k - \partial_{xy}^2 f(w_*) - N_k^T + \partial_{yx}^2 g(w_*)^T)$. Thus, we have
    \begin{align*}
        \specnorm{\alpha_k - A(w_*)}
        &= \specnorm{\frac{1}{2}(M_k - \partial_{xy}^2 f(w_*) - N_k^T + \partial_{yx}^2 g(w_*)^T)} \\
        &\leq \frac{1}{2} (\specnorm{M_k - \partial_{xy}^2 f(w_*)} + \specnorm{N_k^T - \partial_{yx}^2 g(w_*)^T}) \\
        &\leq \frac{1}{2} (\specnorm{\mu_k - D \partial_{x} f(w_*)} + \specnorm{\nu_k^T - D \partial_{y} g(w_*)^T}) \\
        &\leq \frac{1}{2} (\delta + \delta) = \delta.
    \end{align*}
    \textcolor{revblue}{This proves the claim.}
\end{proof}

\begin{color}{revblue}
We now use these lemmas to obtain a local linear convergence result for the
LRSGA scheme. Here, the main object is the frozen SGA map
obtained by freezing the skew correction at the equilibrium. Indeed, let
\(w_*=(x_*,y_*)\) be a SNE point, and define
\(\fplrsga:\mathbb R^{m+n}\to\mathbb R^{m+n}\) by
\begin{equation}
    \label{eq:sga_approx_fixed_point_function}
    \fplrsga(v)
    :=
    v
    - \eta (I - \tau A(w_*) ) \, F(v) .
\end{equation}
The associated iteration is \(v_{k+1}=\fplrsga(v_k)\), which is exactly
the finite-dimensional instance of \(T_{\eta,\tau}^*\) in
Theorem~\ref{thm:frozen-skew-local-convergence-hilbert}, with
\(\mathcal F=F\) and \(\mathcal A_*=A(w_*)\), and \(w_*\) is a fixed point of
\(\fplrsga\).

The LRSGA sequence is instead the sequence \(\{w_k\}\) generated by
\eqref{eq:def_sga_approx}, where \(\alpha_k\) is produced by the secant
updates. 

More precisely, in the proof below, \(\fplrsga\) is used only as a comparison map:
the one-step discrepancy
\(\|w_{k+1}-\fplrsga(w_k)\|\) is controlled by
\(\|\alpha_k-A(w_*)\|\). Thus the contractivity property of the frozen map will supply the
metric part of the argument, while the low-rank update is treated as a
controlled perturbation of that map.
\end{color}

\begin{theorem}
    \label{theorem:sga-approx-conv-specnrmbound} 
    Let $\Omega \subseteq \mathbb{R}^{m + n}$ be a convex domain and 
    $w_* \in \Omega$ be \textcolor{revblue}{an} SNE point. Let the map 
    $\fplrsga$ defined in \eqref{eq:sga_approx_fixed_point_function}
    be in $C^1(\Omega, \mathbb{R}^{m + n})$.  
    Assume that the game gradient $F$ is Lipschitz continuous in $\Omega$
    with Lipschitz constant $L_F$, and $D \partial_{x} f$, and
    $D \partial_{y} g$ are Lipschitz continuous in $\Omega$ with
    Lipschitz constants $\lipDdxf$, and $\lipDdyg$, respectively. 
    Assume that the spectral norm of the Jacobian $D \fplrsga$ satisfies
    $\sigma = \specnorm{D \fplrsga(w_*)} < 1$.

    \textcolor{revblue}{Then there exist} parameters $\delta_0, R > 0$ such that
    \textcolor{revblue}{\(\overline B_R(w_*) \subseteq \Omega\)}, and for any initialization $w_0$, $\mu_0$, and $\nu_0$ satisfying
    $w_0 \in B_R(w_*)$,
    $\specnorm{\mu_0 - D \partial_x f(w_*)} \leq \delta_0$, and
    $\specnorm{\nu_0 - D \partial_y \textcolor{revblue}{g}(w_*)} \leq \delta_0$,
    the sequence $\{ w_k \}$ \textcolor{revblue}{generated by the LRSGA iteration
    \eqref{eq:def_sga_approx}, with \(\mu_k\) and \(\nu_k\) updated by
    \eqref{eq:UpdateMixedDerivativesF}--\eqref{eq:UpdateMixedDerivativesG},
    converges linearly to \(w_*\).}
\end{theorem}
\begin{proof}
    Since \(w_*\) is an SNE point, \(F(w_*)=0\) \textcolor{revblue}{ and we may assume that \(w_*\) is the unique zero in a sufficiently small neighborhood.}
    \textcolor{revblue}{ As a consequence,} \(w_*\) is \textcolor{revblue}{the unique} fixed
    point of the frozen map \(\fplrsga\). Moreover,
    \(I-\tau A(w_*)\) is invertible because \(A(w_*)\) is skew-symmetric, and
    the \(C^1\) regularity of \(\fplrsga\) gives the required \(C^1\)
    regularity of \(F\) near \(w_*\). Applying
    formula \eqref{eq:qcontraction} from Theorem~\ref{thm:frozen-skew-local-convergence-hilbert} with
    \(\mathcal H=\mathbb R^{m+n}\), \(\mathcal F=F\),
    \(\mathcal A_*=A(w_*)\), and \(z_*=w_*\), we obtain
    \(R_*>0\) and \(q_*\in(\sigma,1)\) such that
    \begin{equation}
        \label{eq:ProofLRSGAFrozenContraction}
        \specnorm{\fplrsga(w)-\fplrsga(v)}
        \leq
        q_*\specnorm{w-v},
        \qquad
        w,v\in \overline B_{R_*}(w_*).
    \end{equation}
    Choose \(q\in(q_*,1)\) and then \(\delta>0\) such that
    \[
        \eta\tau L_F\delta\le q-q_* .
    \]
    By reducing \(R_*\), if necessary, we can choose \(R,\delta_0>0\) such that
    \(0<R\le R_*\), \textcolor{revblue}{\(\overline B_R(w_*)\subseteq\Omega\)}, and
    \[
        \delta_0+\frac{2R}{1-q}\max\{\lipDdxf,\lipDdyg\}\le\delta .
    \]

    Now, we show that, for every $K \geq 0$, the following inequalities hold:
    \begin{equation}
        \label{eq:ProofLRSGAInductionProp}
        \specnorm{w_k - w_*}\leq q^k \specnorm{w_0 - w_*}
        \quad\text{and}\quad
        \specnorm{\alpha_k - A(w_*)}\leq \delta
    \end{equation}
    for $k = 0, \ldots, K$. The proof is by induction, assuming that the initial values $w_0$, $\mu_0$, and $\nu_0$ satisfy $\specnorm{w_0 - w_*} < R$,
    $\specnorm{\mu_0 - D \partial_x f(w_*)} \leq \delta_0$, and
    $\specnorm{\nu_0 - D \partial_y \textcolor{revblue}{g}(w_*)} \leq \delta_0$.

    For $K = 0$, we have
    $\specnorm{w_0 - w_*} \leq 1 \specnorm{w_0 - w_*} = q^0
    \specnorm{w_0 - w_*}$ and from
    Lemma~\ref{lemma:PropertiesSpectralNormBlockMatrices}, we have $\specnorm{\alpha_0 - A(w_*)} \leq \delta_0 \leq \delta$ since $\delta_0$ satisfies $\delta_0 \leq \delta$.

    Next, assume for a $K \geq 0$ the inequalities from (\ref{eq:ProofLRSGAInductionProp}) hold for $k = 0, \ldots, K$.
    Since $q < 1$, \textcolor{revblue}{we have} $w_k \in B_R(w_*)$, $k = 0, \ldots, K$. From the Lipschitz continuity of $F$, we have
    \begin{equation*}
        \specnorm{F(w_k)}
        = \specnorm{F(w_k) - F(w_*)}
        \leq L_F \specnorm{w_k - w_*}, \quad k = 0, \ldots, K.
    \end{equation*}
    \textcolor{revblue}{Thus}
    \begin{equation*}
        \specnorm{w_{k+1} - \fplrsga(w_k)} = \specnorm{\eta \tau (\alpha_k - A(w_*)) F(w_k)}
        \leq \eta \tau \delta L_F \specnorm{w_k - w_*}, \quad k = 0, \ldots, K.
    \end{equation*}
    Combining this upper bound with the contraction estimate
    \eqref{eq:ProofLRSGAFrozenContraction} from
    Theorem~\ref{thm:frozen-skew-local-convergence-hilbert}, and using
    \(\fplrsga(w_*)=w_*\), we obtain
    \begin{align*}
        \specnorm{w_{k+1} - w_*}
        &\leq \specnorm{w_{k+1} - \fplrsga(w_k)}
        + \specnorm{\fplrsga(w_k) - \fplrsga(w_*)} \\
        &\leq \eta \tau \delta L_F \specnorm{w_k - w_*}
        + q_*\specnorm{w_k - w_*} \\
        &\leq q \specnorm{w_k - w_*} \\
        &\leq q^{k+1} \specnorm{w_0 - w_*}, \quad k = 0, \ldots, K.
    \end{align*}
    \textcolor{revblue}{Thus, \(w_{k+1}\in B_R(w_*)\) as well.}

    Hence, from Lemma~\ref{lemma:BoundUpdates} for $k = 0, \ldots, K$, we have
    \begin{align*}
        \specnorm{\mu_{k+1} - D \partial_x f(w_*)}
        &\leq \specnorm{\mu_{k} - D \partial_x f(w_*)}\\
        & \quad + 2 \lipDdxf \max\{\specnorm{w_k - w_*}, \specnorm{w_{k+1} - w_*} \}\\
        &\leq \specnorm{\mu_{0} - D \partial_x f(w_*)} \\ 
        & \quad + 2 \lipDdxf \sum_{\ell = 0}^{k} \max\{\specnorm{w_\ell - w_*}, \specnorm{w_{\ell+1} - w_*} \}\\
        &\leq \specnorm{\mu_{0} - D \partial_x f(w_*)} + 2 \lipDdxf \sum_{\ell = 0}^{k} q^\ell \specnorm{w_{0} - w_*} \\
        &\leq \delta_0
        + 2 \lipDdxf \frac{R}{1-q} 
        \leq \delta. 
    \end{align*}
    Similarly, we have
    $\specnorm{\nu_{k+1} - D \partial_y g(w_*)} \leq \delta$ and thus,
    as in the case $K = 0$, the upper bound
    \[\specnorm{\alpha_{k+1} - A(w_*)} \leq \delta,\] for
    $k = 0, \ldots, K$, follows from
    Lemma~\ref{lemma:PropertiesSpectralNormBlockMatrices}. Together
    with the property for $K = 0$ the inequalities
    (\ref{eq:ProofLRSGAInductionProp}) hold for $k = 0, \ldots, K + 1$.

    By induction, we have that
    $\norm{w_k - w_*} \leq q^k \norm{w_0 - w_*}$ holds for all
    $k \geq 0$. Thus, the sequence $\{ w_k \}$ converges to $w_*$ at a
    linear rate.
    \textcolor{revblue}{Observe that, if \(w_{k+1}=w_k\) for some \(k\),
    Proposition~\ref{prop:lrsga-fixed-point-zero} gives \(F(w_k)=0\). Since
    \(w_k\in  \overline B_{R_*}(w_*)\) and being \(w_*\) the unique zero of \(F\) in \(\overline B_{R_*}(w_*)\), it follows that
    \(w_k=w_*\).}
\end{proof}

\begin{remark}
\textcolor{revblue}{The convergence statement above is local. The initial point
must belong to a neighbourhood of \(w_*\) where the frozen-skew comparison map
\(\fplrsga\) satisfies the contraction estimate
\eqref{eq:ProofLRSGAFrozenContraction}, and where the secant approximation error
\(\|\alpha_k-A(w_*)\|\) remains small enough to be absorbed in the contraction
estimate. The theorem therefore does not provide a global initialization
criterion or a globalization mechanism. The use of  strategies such as
line-search, or restarts of the
secant matrices requires a separate analysis and is left for future work.}
\end{remark}

\begin{corollary}
    \label{cor:sga_approx_convergence}
    Let $\Omega \subseteq \mathbb{R}^{m + n}$ be a convex domain,  let
    $f \in C^3(\Omega, \mathbb{R})$ and
    $g \in C^3(\Omega, \mathbb{R})$, and suppose that $w_* \in \Omega$ is \textcolor{revblue}{an} SNE point \textcolor{revblue}{whose symmetric part $S(w_*)$ is positive definite}.
    Then there exist $R > 0$ and $\delta_0 > 0$, such that
    choosing $w_0 \in B_R(w_*)$, and $\mu_0$, $\nu_0$ satisfying 
    $\specnorm{\mu_0 - D \partial_x f(w_0)} \leq \delta_0$ and
    $\specnorm{\nu_0 - D \partial_y g(w_0)} \leq \delta_0$, the
    iterates $\{w_k \}$ given by the LRSGA method \eqref{eq:def_sga_approx} converge 
    linearly to $w_*$ for appropriately chosen \(\eta, \tau > 0\).
\end{corollary}
\begin{proof}
\begin{color}{revblue}
Choose \(r_0>0\) such that \(B_{r_0}(w_*)\subset\Omega\). Since
\(f,g\in C^3(\Omega,\mathbb R)\), after reducing \(r_0\) if necessary, the
maps \(F\), \(D\partial_x f\), and \(D\partial_y g\) are Lipschitz continuous
on \(B_{r_0}(w_*)\), and \(\fplrsga\) is \(C^1\) there.

Set \(H_*:=H(w_*)\), \(S_*:=S(w_*)\), and \(A_*:=A(w_*)\). Since \(w_*\) is
an SNE point, \(F(w_*)=0\), \(H_*\succeq0\), and 
\(S_*\) is positive definite, Lemma~\ref{lma:metricproperties}(b) gives
\(\bar\tau>0\) such that, for every \(\tau\in(0,\bar\tau)\), the linear map
\[
    B_*:=(I-\tau A_*)H_*
\]
is \(h\)-strongly monotone. Moreover,
\[
    L_*:=\sqrt{1+\tau^2\|A_*\|_2^2}\,\|H_*\|_2
\]
is a Lipschitz constant for \(B_*\). Hence Proposition~\ref{prop:avnonexp}
implies that, for every \(\eta\in(0,2h/L_*^2)\),
\[
    \|I-\eta B_*\|_2<1 .
\]
Since
\[
    D\fplrsga(w_*)=I-\eta(I-\tau A(w_*))H(w_*)=I-\eta B_*,
\]
the spectral condition
\(\|D\fplrsga(w_*)\|_2<1\) in
Theorem~\ref{theorem:sga-approx-conv-specnrmbound} is satisfied.
Therefore Theorem~\ref{theorem:sga-approx-conv-specnrmbound}, applied on this
local ball, gives \(R_1>0\) and \(\delta_1>0\) for initial secant matrices
close to \(D\partial_x f(w_*)\) and \(D\partial_y g(w_*)\).

It remains only to convert the stated initialization at \(w_0\) into that
required by the theorem. Let \(L_x,L_y\) be Lipschitz constants for
\(D\partial_x f\) and \(D\partial_y g\) on \(B_{R_1}(w_*)\). Choose
\(R\le R_1\) and \(\delta_0>0\) such that
\(\delta_0+L_xR\le\delta_1\) and \(\delta_0+L_yR\le\delta_1\). Then, whenever
\(w_0\in B_R(w_*)\),
\begin{eqnarray*} 
\|\mu_0-D\partial_x f(w_*)\|_2
& \le &
\|\mu_0-D\partial_x f(w_0)\|_2
+\|D\partial_x f(w_0)-D\partial_x f(w_*)\|_2 \\
\ & \le & \delta_0+L_xR\le\delta_1,
\end{eqnarray*}
and the same argument gives
\(\|\nu_0-D\partial_y g(w_*)\|_2\le\delta_1\). Thus all hypotheses of
Theorem~\ref{theorem:sga-approx-conv-specnrmbound} are satisfied, and the
LRSGA iterates converge linearly to \(w_*\).
\end{color}
\end{proof}
\begingroup
\color{revgreen}
\let\MainJOTAOriginalTextcolor\textcolor
\renewcommand{\textcolor}[2]{#2}
\section{Numerical experiments}
\label{sec:numerical-experiments}

In this section, we report three numerical experiments designed to evaluate
different aspects of the proposed LRSGA method.  The first two experiments are
conducted on explicit deterministic games, where the game gradient and the
relevant equilibrium conditions can be computed and monitored directly.  These
experiments allow us to study the qualitative behavior of the method, the 
stationarity residual $\|F(w_k)\|$, and the distance from known equilibria.

The first experiment considers a low-dimensional nonconvex game with multiple
stationary points and two local Nash equilibria.  Its purpose is to illustrate
how different optimization methods, and in particular different Broyden
initializations in LRSGA, may select different local Nash equilibria.  The
second experiment considers a higher-dimensional nonlinear game with a known
stable Nash equilibrium.  This setting allows us to compare LRSGA, SGA, the
GDA method introduced in \eqref{eGD}, its optimistic variant OGDA, the
extragradient method EG, and linearized
CGD in terms of residual decay, distance to equilibrium,
number of iterations, and mean optimizer time per iteration.

The third experiment is a CLIP-inspired neural-network training toy example.
Unlike the first two experiments, its purpose is not to certify convergence to
a Nash equilibrium.  Rather, it is designed to compare the computational cost
of LRSGA with that of exact SGA in a neural training setting.  In particular,
we show that LRSGA can achieve test loss values comparable to exact SGA while
substantially reducing the average training time per epoch, because it avoids
the explicit computation of mixed derivatives.

\subsection{A low-dimensional game with multiple known stationary points}
\label{subsec:low-dimensional-multiple-equilibria}

We next consider a two-player deterministic game in which both players have
one scalar decision variable.  The purpose of this experiment is not to test large-scale performance, but
rather to evaluate the behavior of the methods in a low-dimensional nonconvex
setting with multiple stationary points, where different algorithms and
initializations may lead to different local Nash equilibria. The players minimize, respectively,
\begin{equation}
    f(x,y) = (1-x)^2 + 100(y-x^2)^2,
    \qquad
    g(x,y) = (x-1)^2 + (y-1)^2,
    \label{eq:lowdim-game-objectives}
\end{equation}
with $(x,y)\in\mathbb{R}^2$.  The associated game gradient is
\begin{equation}
    F(x,y)
    =
    \begin{pmatrix}
        \textcolor{black}{\partial_x f(x,y)}\\
        \textcolor{black}{\partial_y g(x,y)}
    \end{pmatrix}
    =
    \begin{pmatrix}
        2(x-1)-400x(y-x^2)\\
        2(y-1)
    \end{pmatrix}.
    \label{eq:lowdim-game-gradient}
\end{equation}
Since the feasible set is the whole space, every local Nash equilibrium lies
in the domain of the game and no boundary conditions arise. Therefore,
$F(x,y)=0$ is the standard first-order necessary condition for a local Nash equilibrium. In this experiment, however,
$\|F(x,y)\|$ is used only as a stationarity residual: it is not, by itself, a
sufficient certificate of a local Nash equilibrium.

The stationary points can be computed explicitly.  From
$\textcolor{black}{\partial_y g(x,y)}=0$ we obtain $y=1$.  Substituting in
$\textcolor{black}{\partial_x f(x,y)}=0$ gives $200x^3 - 199x - 1 = (x-1)(200x^2+200x+1) = 0.$
Thus the three stationary points are
\begin{equation}
    E_0 =
    \left(\frac{-10-7\sqrt{2}}{20},1\right),
    \qquad
    E_1 =
    \left(\frac{-10+7\sqrt{2}}{20},1\right),
    \qquad
    E_2 = (1,1).
    \label{eq:lowdim-stationary-points}
\end{equation}
Numerically, these points are approximately
$E_0=(-0.99497475,1)$, $E_1=(-0.00502525,1)$ and $E_2=(1,1)$.
The playerwise second-order quantities are
\begin{equation}
    f_{xx}(x,y)=2-400y+1200x^2,
    \qquad
    g_{yy}(x,y)=2.
\end{equation}
At the stationary points,
\begin{equation}
    f_{xx}(E_0)=196+420\sqrt{2}>0,
    \qquad
    f_{xx}(E_1)=196-420\sqrt{2}<0,
    \qquad
    f_{xx}(E_2)=802>0.
\end{equation}
Consequently, $E_0$ and $E_2$ are 
local Nash equilibria, while
$E_1$ is a stationary point but not a local Nash equilibrium.

The game Hessian, namely the Jacobian of the game gradient, is
\begin{equation}
    H_F(x,y)=\textcolor{black}{D F(x,y)}
    =
    \begin{pmatrix}
        2-400y+1200x^2 & \qquad -400x\\
        0 & \qquad 2
    \end{pmatrix}.
    \label{eq:lowdim-game-hessian}
\end{equation}
The mixed derivatives are therefore
\begin{equation}
    f_{xy}(x,y)=-400x,
    \qquad
    g_{yx}(x,y)=0.
    \label{eq:lowdim-mixed-derivatives}
\end{equation}

The functions in \eqref{eq:lowdim-game-objectives} are polynomials; hence they
are $C^\infty$ on the whole domain $\mathbb{R}^2$.
The determinant of
\eqref{eq:lowdim-game-hessian} is
\begin{equation}
    \det H_F(x,y) =
    2\bigl(2-400y+1200x^2\bigr).
\end{equation}

At the two local Nash equilibria $E_0$ and $E_2$, the game Hessian is
invertible, since
\[
    \det H_F(E_0)\approx 1579.94,
    \qquad
    \det H_F(E_2)\approx 1604.00.
\]
Therefore, \textcolor{black}{as recalled in Section~\ref{sec:Nash}}, $E_0$ and $E_2$ are isolated local Nash equilibria.

It remains to check whether these isolated Nash equilibria are stable in the sense of \textcolor{black}{Definition~\ref{DefSNEpoint}}.  Since $H_F$ is not symmetric, positive
semidefiniteness of $H_F$ is understood through the quadratic form
$\langle z,H_Fz\rangle$.  This quadratic form depends only on the symmetric
part, because the skew-symmetric part contributes zero to
$\langle z,H_Fz\rangle$.  The symmetric part is
\begin{equation}
    S_F(x,y)
    =
    \frac{H_F(x,y)+H_F(x,y)^\top}{2}
    =
    \begin{pmatrix}
        2-400y+1200x^2 & \quad-200x\\
        -200x & \quad 2
    \end{pmatrix}.
\end{equation}
Therefore, $H_F$ is positive semidefinite in the sense of
\textcolor{black}{Definition~\ref{DefSNEpoint}} if and only if $S_F$ is positive semidefinite.  \textcolor{black}{At the two local Nash equilibria,}
\[
    \det S_F(E_0)\approx -38019.05,
    \qquad
    \det S_F(E_2)\approx -38396.00.
\]
Thus $S_F$ is indefinite at both $E_0$ and $E_2$.  Consequently, although
$E_0$ and $E_2$ are isolated local Nash equilibria, they are not stable Nash
equilibria in the sense of \textcolor{black}{Definition~\ref{DefSNEpoint}}.
This experiment is therefore deliberately outside the hypotheses of Theorem~\ref{theorem:sga-approx-conv-specnrmbound}, which require an SNE point. The observed convergence of LRSGA and SGA towards $E_0$ and $E_2$ suggests that the SNE condition is sufficient but not necessary for convergence, and that the methods can perform well beyond the regime covered by the theory.

Before comparing the optimization methods, Fig.~\ref{fig:lowdim-stationary-overview}
shows the stationary structure of the game.  The three stationary points are
the intersections of the two first-order nullclines $f_x=0$ and $g_y=0$.

\begin{figure}[htbp]
\centering
\includegraphics[width=0.82\linewidth]{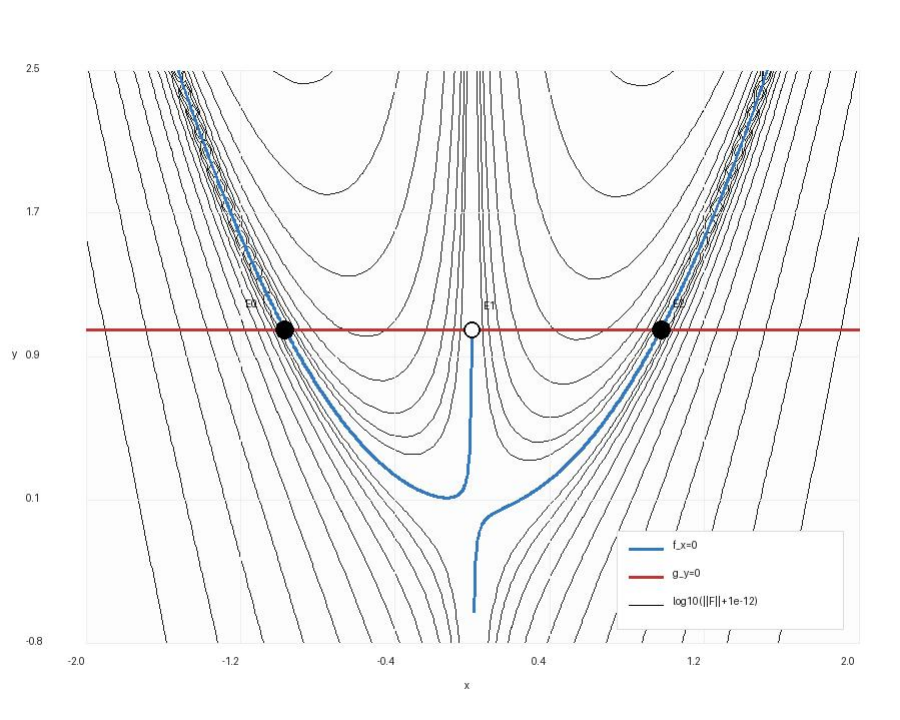}
\caption{Stationary structure of the low-dimensional game before applying any
optimization method.  The blue curve is the nullcline $f_x=0$, the red line is
the nullcline $g_y=0$, and the thin black curves are level curves of
$\log_{10}(\|F(x,y)\|+10^{-12})$. The filled black markers $E_0$ and $E_2$ are local Nash equilibria,
while the open marker $E_1$ is a stationary point that is not a local Nash
equilibrium.}
\label{fig:lowdim-stationary-overview}
\end{figure}

We compare seven deterministic methods: GDA, the optimistic GDA method
(OGDA), the extragradient method (EG), SGA, the linearized CGD method
as in \textcolor{black}{\eqref{MCGD}}, LRSGA, and LRSGAexact. \textcolor{black}{We recall} that SGA and CGD use the exact mixed derivatives, GDA denotes the iteration
$w_{k+1}=w_k-\eta F(w_k)$ (introduced in \eqref{eGD}),  LRSGA denotes the proposed low-rank method with
the random Broyden initialization used throughout the experiments.
OGDA denotes the optimistic gradient iteration
\[
w_{k+1}=w_k-\eta\bigl(2F(w_k)-F(w_{k-1})\bigr),
\]
with a plain gradient step at $k=0$ \cite{Daskalakis2018,rakhlin2013}, and EG
denotes the extragradient method of Korpelevich \cite{Korpelevich1976},
\[
\tilde w_k=w_k-\eta F(w_k),
\qquad
w_{k+1}=w_k-\eta F(\tilde w_k),
\]
which requires two gradient evaluations per iteration. Both are first-order
schemes that damp the rotational component of the game dynamics without using
second-order information, and therefore constitute the natural first-order
baselines for the proposed method.  In the
present implementation, the initial Broyden matrices in LRSGA are initialized
with independent uniform random entries and a fixed seed.  LRSGAexact denotes
the same method except that the initial Broyden matrices are set equal to the
exact Jacobian rows of $\textcolor{black}{\partial_x f}$ and $\textcolor{black}{\partial_y g}$ at the initial point:
\begin{equation}
    \mu_0 = \bigl(f_{xx}(x_0,y_0), f_{xy}(x_0,y_0)\bigr),
    \qquad
    \nu_0 = \bigl(g_{yx}(x_0,y_0), g_{yy}(x_0,y_0)\bigr).
\end{equation}
After this initialization, LRSGAexact uses the same Broyden update as LRSGA.
This comparison is included to assess the sensitivity of the low-rank Broyden
approximation to its initialization.  Since the purpose of LRSGA is precisely
to avoid the explicit computation of mixed derivatives, LRSGAexact is used only
as a diagnostic variant and not as the proposed implementation.

All methods are run from the same initial points, with the same fixed
step-size parameters $\eta=\tau=10^{-3}$ and for $5000$ iterations.  No
averaging over random seeds is performed in this experiment: for each initial
point, each method is run once, and all methods share exactly the same initial
condition.  Table~\ref{tab:lowdim-selected-equilibria} reports the stationary
point selected by each method, while Table~\ref{tab:lowdim-final-residuals}
reports the final residual $\|F(w_K)\|$.
The selected equilibrium is determined by the nearest stationary point to the
final iterate.

\begin{table}[t]
\centering
\scriptsize
\setlength{\tabcolsep}{3pt}
\begin{tabular}{c|ccccccc}
\hline
$w_0=(x_0,y_0)$ & LRSGA & LRSGAexact & GDA & OGDA & EG & SGA & CGD\\
\hline
$(-1.25,1.25)$ & $E_0$ & $E_0$ & $E_0$ & -- & -- & $E_0$ & $E_0$\\
$(1.25,1.25)$ & $E_2$ & $E_2$ & $E_2$ & -- & -- & $E_2$ & $E_2$\\
$(1.75,-0.06)$ & $E_2$ & $E_0$ & $E_0$ & -- & -- & $E_0$ & $E_0$\\
$(-1.65,0.25)$ & $E_2$ & $E_0$ & $E_2$ & -- & -- & $E_0$ & $E_2$\\
\hline
\end{tabular}
\caption{Equilibrium selected by each method in the low-dimensional
multi-equilibrium game.  The initializations $(1.75,-0.06)$ and $(-1.65,0.25)$ produce different local
Nash equilibrium selections across \textcolor{black}{methods.} A dash
indicates that the method did not approach any stationary point with the
common stepsize $\eta=10^{-3}$; see
Table~\ref{tab:lowdim-final-residuals} and the discussion below.}
\label{tab:lowdim-selected-equilibria}
\end{table}

\begin{table}[t]
\centering
\scriptsize
\setlength{\tabcolsep}{3pt}
\begin{tabular}{c|ccccccc}
\hline
$w_0=(x_0,y_0)$ & LRSGA & LRSGAexact & GDA & OGDA & EG & SGA & CGD\\
\hline
$(-1.25,1.25)$ & $4.01{\times}10^{-5}$ & $4.51{\times}10^{-5}$ & $2.52{\times}10^{-5}$ & $1.31{\times}10^{2}$ & $1.86{\times}10^{3}$ & $4.52{\times}10^{-5}$ & $2.26{\times}10^{-5}$\\
$(1.25,1.25)$ & $3.95{\times}10^{-5}$ & $4.44{\times}10^{-5}$ & $2.51{\times}10^{-5}$ & $1.37{\times}10^{2}$ & $1.88{\times}10^{3}$ & $4.47{\times}10^{-5}$ & $2.26{\times}10^{-5}$\\
$(1.75,-0.06)$ & $1.35{\times}10^{-4}$ & $4.65{\times}10^{-5}$ & $1.07{\times}10^{-4}$ & $2.20{\times}10^{15}$ & $2.22{\times}10^{15}$ & $5.54{\times}10^{-5}$ & $9.58{\times}10^{-5}$\\
$(-1.65,0.25)$ & $1.63{\times}10^{-4}$ & $4.92{\times}10^{-5}$ & $7.54{\times}10^{-5}$ & $1.25{\times}10^{14}$ & $1.86{\times}10^{3}$ & $4.54{\times}10^{-5}$ & $6.78{\times}10^{-5}$\\
\hline
\end{tabular}
\caption{Final stationarity residual $\|F(w_K)\|$ after $K=5000$ iterations.
The large values in the OGDA and EG columns correspond to divergent or
non-stationary cycling trajectories with the common stepsize $\eta=10^{-3}$.
}
\label{tab:lowdim-final-residuals}
\end{table}

\begin{figure}[htbp]
\centering
\begin{minipage}{0.48\linewidth}
    \centering
    \includegraphics[width=\linewidth]{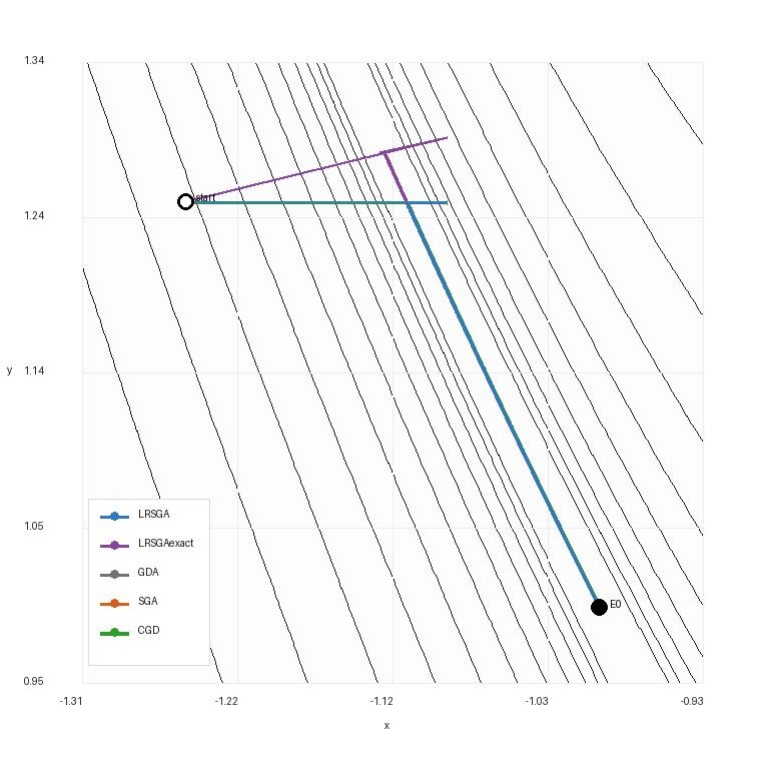}\\
    \small (a) $w_0=(-1.25,1.25)$
\end{minipage}\hfill
\begin{minipage}{0.48\linewidth}
    \centering
    \includegraphics[width=\linewidth]{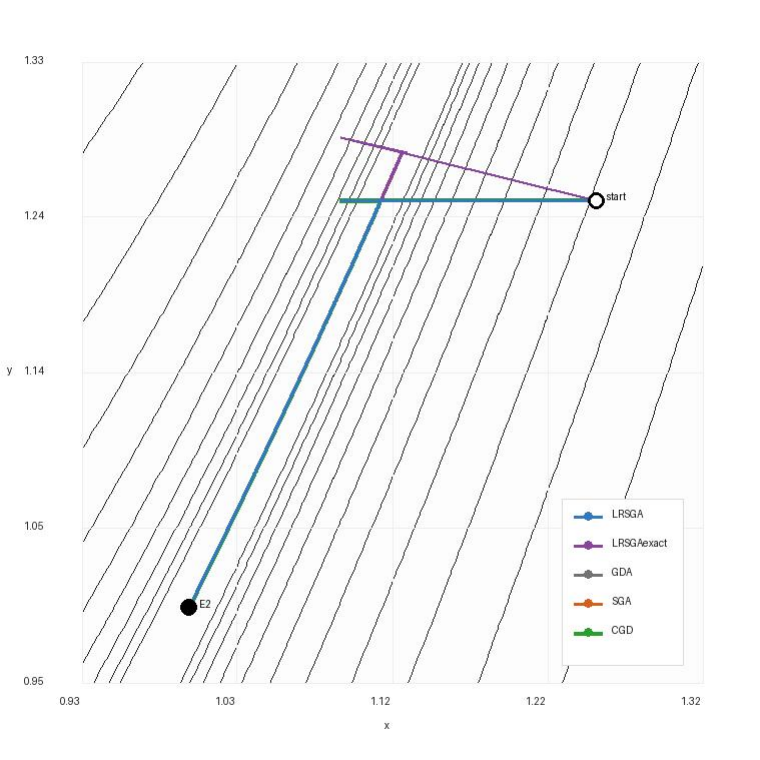}\\
    \small (b) $w_0=(1.25,1.25)$
\end{minipage}

\vspace{0.5em}

\begin{minipage}{0.48\linewidth}
    \centering
    \includegraphics[width=\linewidth]{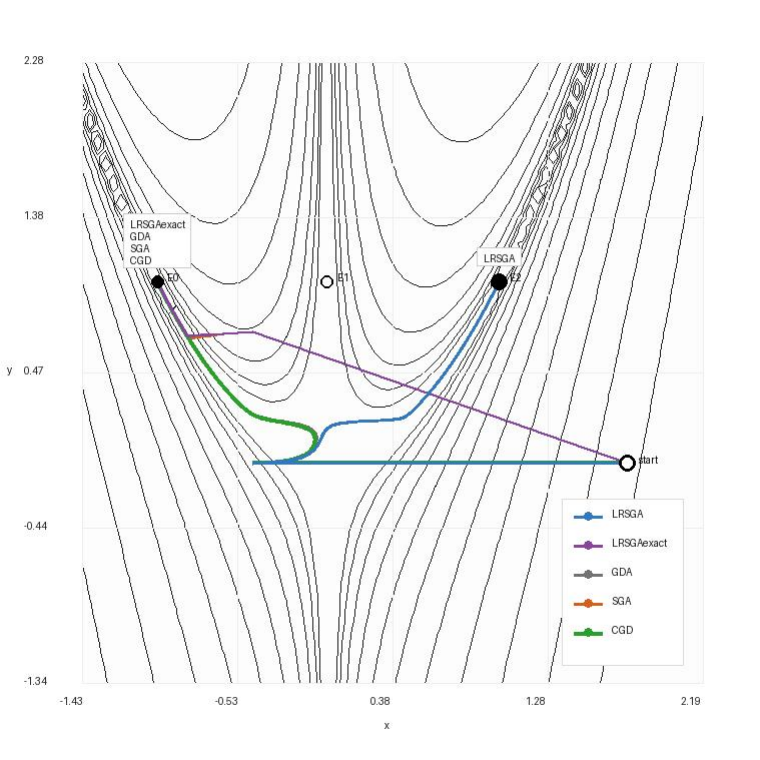}\\
    \small (c) $w_0=(1.75,-0.06)$
\end{minipage}\hfill
\begin{minipage}{0.48\linewidth}
    \centering
    \includegraphics[width=\linewidth]{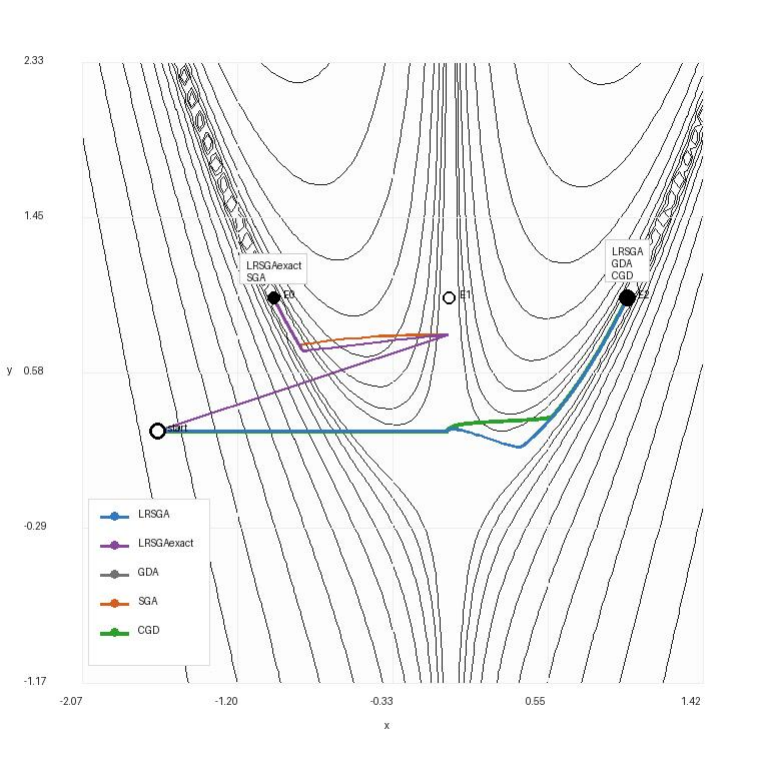}\\
    \small (d) $w_0=(-1.65,0.25)$
\end{minipage}
\caption{Local trajectories in the low-dimensional game.  The color convention
is the same in all panels: LRSGA is blue, LRSGAexact is purple, GDA is gray,
SGA is orange, and CGD is green.  The OGDA and EG trajectories are omitted
from these panels because, with the common stepsize $\eta=10^{-3}$, they
diverge or leave the displayed region; see
Table~\ref{tab:lowdim-final-residuals} and the related discussion.  The white circle labeled ``start'' denotes
the common initial point.  The filled black markers $E_0$ and $E_2$ are
local Nash equilibria, while the open marker $E_1$ is a stationary point.
The thin black curves are level curves of
$\log_{10}(\|F(x,y)\|+10^{-12})$ and are included only as a residual landscape;
they are not level curves of either objective function.  In panels (c) and
(d), the small labels above the equilibria list the methods that end at the
corresponding point.  These panels show that
different methods may select different local Nash equilibria.}
\label{fig:lowdim-trajectories}
\end{figure}

The four initializations were chosen to separate two effects.  The points
$(-1.25,1.25)$ and $(1.25,1.25)$ are in regions where all the convergent
methods select the
nearest local Nash equilibrium.  In contrast, the points $(1.75,-0.06)$ and $(-1.65,0.25)$ are initializations
for which the selected local Nash equilibrium depends on the method.  From
$(1.75,-0.06)$, LRSGA selects the right equilibrium $E_2$, whereas LRSGAexact,
GDA, SGA, and CGD select the left equilibrium $E_0$.  From
$(-1.65,0.25)$, LRSGA, GDA, and CGD select $E_2$, while LRSGAexact and SGA
select $E_0$. These results indicate that, in this example, the Broyden initialization can
affect which local Nash equilibrium is selected by the method, especially for
initializations where different update rules already lead to different
outcomes.  

The behavior of the two first-order baselines deserves a separate comment.
With the common stepsize $\eta=10^{-3}$, OGDA diverges from the
initializations $(1.75,-0.06)$ and $(-1.65,0.25)$, and stalls on
non-stationary configurations from the remaining two, while EG diverges from
$(1.75,-0.06)$ and settles on non-stationary cycles from the other three; see
the OGDA and EG columns of Table~\ref{tab:lowdim-final-residuals}.  This is
consistent with the fact that the optimistic and extrapolation corrections
reduce the stable stepsize range with respect to plain GDA, which is
restrictive in this stiff game, where the eigenvalues of the game Hessian
reach the order of $10^{3}$ near the equilibria.  As a control, we repeated
the experiment with all seven methods, the reduced stepsize
$\eta=\tau=2\times10^{-4}$, and $25000$ iterations: in this configuration all
methods, including OGDA and EG, converge to $E_0$ or $E_2$ from all four initializations, with
final residuals of the same order as in
Table~\ref{tab:lowdim-final-residuals}.  Thus the failures of
OGDA and EG in Table~\ref{tab:lowdim-final-residuals} are a stepsize-range
effect rather than a structural one; at the same time, the experiment shows
that the skew-corrected methods SGA and LRSGA tolerate the same stepsize as
GDA, whereas the first-order optimistic and extragradient corrections do not.

Overall, the experiment shows that the random Broyden initialization used by
LRSGA can influence the local Nash equilibrium selected in a nonconvex
multi-equilibrium game.  Importantly, this sensitivity does not prevent
local convergence to a Nash equilibrium: in all reported runs, the LRSGA methods end near either $E_0$ or $E_2$, with small final first-order residuals.


\subsection{A high-dimensional nonlinear game with a known stable Nash equilibrium}
\label{subsec:high-dimensional-known-equilibrium}

We now consider a deterministic high-dimensional game with a known Nash
equilibrium.  The goal of this experiment is to evaluate the convergence
behavior of the methods by monitoring the stationarity residual
$\|F(w_k)\|$, the distance to the equilibrium $\|w_k-w_*\|$, the number
of optimizer iterations required to satisfy the stopping criterion, and the
mean optimizer time per iteration. 


Let $x,y\in\mathbb{R}^d$.  For each random seed, we generate diagonal matrices
$Q_x=\operatorname{diag}(q_x)$ and $Q_y=\operatorname{diag}(q_y)$ with diagonal
entries sampled independently from the interval $[1,2]$.  We also generate two dense random
matrices $C,D\in\mathbb{R}^{d\times d}$ and normalize them so that
$\|C\|_2=\|D\|_2=1$.  With elementwise sine and cosine, the two players
minimize
\begin{equation}
    f(x,y)
    =
    \frac{1}{2}x^\top Q_x x
    +
    \alpha \sin(\omega x)^\top C \sin(\omega y),
    \qquad
    g(x,y)
    =
    \frac{1}{2}y^\top Q_y y
    +
    \beta \sin(\omega y)^\top D \sin(\omega x).
    \label{eq:highdim-game-objectives}
\end{equation}
In all the experiments below we use
$\alpha=\beta=0.04375$ and $\omega=4$, so that
$\alpha\omega^2=\beta\omega^2=0.7$.  The associated game gradient is
\begin{equation}
    F(x,y)
    =
    \begin{pmatrix}
        \textcolor{black}{\partial_x f(x,y)}\\
        \textcolor{black}{\partial_y g(x,y)} 
    \end{pmatrix}=
    \begin{pmatrix}
         Q_xx +\alpha\omega\, \cos(\omega x)\odot \bigl(C\sin(\omega y)\bigr)\\
        Q_yy + \beta\omega\, \cos(\omega y)\odot \bigl(D\sin(\omega x)\bigr)
    \end{pmatrix},
    \label{eq:highdim-game-gradient-definition}
\end{equation}
where $\odot$ denotes the componentwise product.  Since
$\sin(0)=0$, the point $    w_*=(x_*,y_*)=(0,0)$
satisfies $F(w_*)=0$.  
The functions in \eqref{eq:highdim-game-objectives} are 
$C^\infty$ on the whole domain $\mathbb{R}^{2d}$.
\textcolor{black}{The game Hessian} has the block form
\begin{equation}
    H_F(x,y)
    =
    \textcolor{black}{D F(x,y)}
    =
    \begin{pmatrix}
        H_{xx}(x,y) & H_{xy}(x,y)\\
        H_{yx}(x,y) & H_{yy}(x,y)
    \end{pmatrix},
    \label{eq:highdim-game-hessian}
\end{equation}
with
\begin{align}
    H_{xx}(x,y)
    &=
    Q_x
    -
    \alpha\omega^2
    \operatorname{diag}
    \left(
        \sin(\omega x)\odot C\sin(\omega y)
    \right),
    \label{eq:highdim-A-block}\\
    H_{xy}(x,y)
    &=
    \alpha\omega^2
    \operatorname{diag}\bigl(\cos(\omega x)\bigr)
    C
    \operatorname{diag}\bigl(\cos(\omega y)\bigr),
    \label{eq:highdim-B-block}\\
    H_{yx}(x,y)
    &=
    \beta\omega^2
    \operatorname{diag}\bigl(\cos(\omega y)\bigr)
    D
    \operatorname{diag}\bigl(\cos(\omega x)\bigr),
    \label{eq:highdim-G-block}\\
    H_{yy}(x,y)
    &=
    Q_y
    -
    \beta\omega^2
    \operatorname{diag}
    \left(
        \sin(\omega y)\odot D\sin(\omega x)
    \right).
    \label{eq:highdim-R-block}
\end{align}

The diagonal blocks are the playerwise Hessians \(H_{xx}(x,y)=\textcolor{black}{\partial_{xx}^2 f(x,y)}\) and \( H_{yy}(x,y)=\textcolor{black}{\partial_{yy}^2 g(x,y)}.\) At $w_*=(0,0)$, since $\sin(0)=0$, we have
\[
    \textcolor{black}{\partial_{xx}^2 f(w_*)}=Q_x,
    \qquad
    \textcolor{black}{\partial_{yy}^2 g(w_*)}=Q_y.
\]
Because the diagonal entries of $Q_x$ and $Q_y$ belong to the interval $[1,2]$, both
matrices are positive definite.  Therefore, together with
$F(w_*)=0$, the playerwise second-order sufficient conditions hold, and
$w_*$ is a local Nash equilibrium.

The mixed derivatives are the blocks $H_{xy}(x,y)=\textcolor{black}{\partial_{x y}^2 f(x,y)}$ and $H_{yx}(x,y)=\textcolor{black}{\partial_{y x}^2 g(x,y)}$. We emphasize that they are generally nonzero because the
random matrices $C$ and $D$ are dense, and they are not constant because of the
factors $\operatorname{diag}(\cos(\omega x))$ and $\operatorname{diag}(\cos(\omega y))$. 

Hence, this experiment tests LRSGA in a setting where the Broyden updates
approximate varying mixed-derivative blocks rather than constant or identically
zero matrices.

At $w_*=(0,0)$, the game Hessian becomes
\begin{equation}
    H_F(w_*)
    =
    \begin{pmatrix}
        Q_x & \gamma C\\
        \gamma D & Q_y
    \end{pmatrix},
    \qquad
    \gamma=\alpha\omega^2=\beta\omega^2=0.7 .
    \label{eq:highdim-hessian-at-equilibrium}
\end{equation}
The symmetric part at $w_*$ is
\begin{equation}
    S_F(w_*)
    =
    \frac{H_F(w_*)+H_F(w_*)^\top}{2}
    =
    \begin{pmatrix}
        Q_x & \frac{\gamma}{2}(C+D^\top)\\
        \frac{\gamma}{2}(C^\top+D) & Q_y
    \end{pmatrix}.
    \label{eq:highdim-symmetric-part}
\end{equation}

Since $\lambda_{\min}(Q_x)\geq 1$, $\lambda_{\min}(Q_y)\geq 1$, and
$\|C\|_2=\|D\|_2=1$, for every $u,v\in\mathbb{R}^d$ we have
\begin{align}
    \begin{pmatrix}u\\v\end{pmatrix}^{\!\top}
    S_F(w_*)
    \begin{pmatrix}u\\v\end{pmatrix}
    &=
    u^\top Q_x u
    +
    v^\top Q_y v
    +
    \gamma\, u^\top (C+D^\top)v \nonumber\\
    &\geq
    \|u\|^2+\|v\|^2
    -
    \gamma\, \|u\|\,\|C+D^\top\|_2\,\|v\| \nonumber\\
    &\geq
    \|u\|^2+\|v\|^2
    -
    2\gamma\,\|u\|\,\|v\| \nonumber\\
    &\geq
    (1-\gamma)\bigl(\|u\|^2+\|v\|^2\bigr) \nonumber\\
    &=
    0.3\bigl(\|u\|^2+\|v\|^2\bigr).
    \label{eq:highdim-symmetric-lower-bound}
\end{align}
Thus $S_F(w_*)$ is positive definite.  Moreover, writing
$H_F(w_*)=S_F(w_*)+A_F(w_*)$, with
$A_F(w_*)^\top=-A_F(w_*)$, we have, for every
$z\in\mathbb{R}^{2d}$,
\(
    z^\top H_F(w_*)z
    =
    z^\top S_F(w_*)z
    +
    z^\top A_F(w_*)z
    =
    z^\top S_F(w_*)z,
\)
because the quadratic form associated with a skew-symmetric matrix is zero.
Hence, for every $z\neq 0$,
\[
    z^\top H_F(w_*)z
    =
    z^\top S_F(w_*)z
    >
    0.
\]
In particular, $H_F(w_*)$ is positive definite in the quadratic-form sense.  Moreover, this strict positivity implies that
$H_F(w_*)$ is invertible: indeed, if $H_F(w_*)z=0$ for some
$z\neq 0$, then $z^\top H_F(w_*)z=0$, a contradiction.  Since
$w_*$ is a local Nash equilibrium, it follows that $w_*$ is a stable
Nash equilibrium.  Moreover, by the isolation criterion recalled in
Section~\ref{sec:Nash}, $w_*$ is isolated.

Also here, we compare GDA, OGDA, EG, SGA, the linearized CGD method, LRSGA,
and LRSGAexact.  To test the sensitivity to the initial point, we
consider several initial radii $r\in\{0.75,3,10,35\}$, where $r$ denotes the
initial distance of each player variable from its equilibrium component.
For each dimension $d\in\{50,100\}$ and
each radius $r$, we run five seeds. For every seed, $x_0$ and $y_0$ are
sampled from a standard Gaussian distribution and then rescaled to satisfy
$\|x_0-x_*\|=\|y_0-y_*\|=r$. All methods share exactly the same initial point.

We use $\eta=0.1$, $\tau=0.5$, and a maximum of $3000$ iterations.  The stopping
criterion is satisfied when both $\|\textcolor{black}{\partial_x f(x_k,y_k)}\|<10^{-8}$ and
$\|\textcolor{black}{\partial_y g(x_k,y_k)}\|<10^{-8}$ hold for five consecutive iterations.
The reported time is the optimizer time per iteration, averaged first over the
iterations of each run and then over the five random seeds.

\begin{table}[htbp]
\centering
\scriptsize
\setlength{\tabcolsep}{3pt}
\begin{tabular}{cc|ccccccc}
\hline
$d$ & $r$ & LRSGA & LRSGAexact & GDA & OGDA & EG & SGA & CGD\\
\hline
50 & 0.75 & $148.2\pm 4.2$ & $186.4\pm 7.4$ & $181.6\pm 4.2$ & $200.4\pm 4.4$ & $200.8\pm 4.5$ & $187.4\pm 7.5$ & $177.4\pm 4.4$ \\
50 & 3 & $159.0\pm 4.3$ & $192.8\pm 7.2$ & $193.4\pm 4.8$ & $214.0\pm 4.7$ & $214.2\pm 4.7$ & $200.4\pm 7.7$ & $189.4\pm 4.7$ \\
50 & 10 & $166.2\pm 4.2$ & $202.0\pm 5.8$ & $202.6\pm 5.1$ & $224.6\pm 5.1$ & $225.0\pm 5.2$ & $210.6\pm 7.4$ & $199.2\pm 5.0$ \\
50 & 35 & $175.4\pm 4.3$ & $211.8\pm 4.8$ & $213.2\pm 5.2$ & $236.2\pm 5.2$ & $236.8\pm 5.1$ & $221.4\pm 7.4$ & $209.4\pm 5.0$ \\
100 & 0.75 & $149.6\pm 2.9$ & $186.0\pm 2.3$ & $180.6\pm 2.1$ & $198.6\pm 2.3$ & $198.8\pm 2.4$ & $186.2\pm 2.2$ & $176.4\pm 2.3$ \\
100 & 3 & $160.4\pm 3.4$ & $193.4\pm 2.4$ & $193.6\pm 2.2$ & $213.8\pm 2.2$ & $213.8\pm 2.2$ & $200.6\pm 2.3$ & $189.6\pm 2.2$ \\
100 & 10 & $169.4\pm 2.3$ & $201.8\pm 1.9$ & $203.4\pm 2.2$ & $224.8\pm 2.2$ & $224.8\pm 2.2$ & $210.8\pm 1.9$ & $199.6\pm 1.8$ \\
100 & 35 & $178.0\pm 3.7$ & $214.8\pm 3.6$ & $214.4\pm 2.2$ & $236.8\pm 2.2$ & $237.2\pm 2.5$ & $221.4\pm 2.2$ & $210.6\pm 2.2$ \\
\hline
\end{tabular}
\caption{Optimizer iterations required to satisfy the stopping criterion.
Each entry reports mean $\pm$ standard deviation over five seeds.}
\label{tab:highdim-iterations}
\end{table}

\begin{table}[htbp]
\centering
\scriptsize
\setlength{\tabcolsep}{3pt}
\begin{tabular}{cc|ccccccc}
\hline
$d$ & $r$ & LRSGA & LRSGAexact & GDA & OGDA & EG & SGA & CGD\\
\hline
50 & 0.75 & $0.041\pm 0.003$ & $0.038\pm 0.002$ & $0.013\pm 0.001$ & $0.015\pm 0.001$ & $0.023\pm 0.001$ & $0.034\pm 0.002$ & $0.029\pm 0.002$ \\
50 & 3 & $0.041\pm 0.001$ & $0.041\pm 0.002$ & $0.014\pm 0.001$ & $0.015\pm 0.000$ & $0.023\pm 0.000$ & $0.035\pm 0.002$ & $0.030\pm 0.001$ \\
50 & 10 & $0.042\pm 0.002$ & $0.041\pm 0.002$ & $0.014\pm 0.001$ & $0.016\pm 0.001$ & $0.024\pm 0.000$ & $0.035\pm 0.002$ & $0.032\pm 0.002$ \\
50 & 35 & $0.041\pm 0.001$ & $0.041\pm 0.001$ & $0.014\pm 0.000$ & $0.016\pm 0.001$ & $0.025\pm 0.001$ & $0.035\pm 0.002$ & $0.030\pm 0.001$ \\
100 & 0.75 & $0.094\pm 0.003$ & $0.089\pm 0.001$ & $0.019\pm 0.000$ & $0.020\pm 0.000$ & $0.033\pm 0.000$ & $0.072\pm 0.000$ & $0.060\pm 0.002$ \\
100 & 3 & $0.097\pm 0.005$ & $0.091\pm 0.001$ & $0.020\pm 0.001$ & $0.022\pm 0.001$ & $0.035\pm 0.002$ & $0.072\pm 0.001$ & $0.062\pm 0.002$ \\
100 & 10 & $0.098\pm 0.004$ & $0.092\pm 0.001$ & $0.020\pm 0.001$ & $0.022\pm 0.001$ & $0.035\pm 0.000$ & $0.076\pm 0.003$ & $0.061\pm 0.003$ \\
100 & 35 & $0.096\pm 0.001$ & $0.094\pm 0.002$ & $0.020\pm 0.000$ & $0.021\pm 0.000$ & $0.035\pm 0.001$ & $0.072\pm 0.003$ & $0.060\pm 0.000$ \\
\hline
\end{tabular}
\caption{Mean optimizer time per iteration in milliseconds.  These timings
measure only the optimizer update computations and do not include diagnostic
quantities such as trajectory storage or mixed-derivative error monitoring.
All methods are timed on the same machine, so the values are directly
comparable across columns.}
\label{tab:highdim-times}
\end{table}

\begin{table}[htbp]
\centering
\scriptsize
\setlength{\tabcolsep}{3pt}
\begin{tabular}{cc|ccccccc}
\hline
$d$ & $r$ & LRSGA & LRSGAexact & GDA & OGDA & EG & SGA & CGD\\
\hline
50 & 0.75 & $7.54\times 10^{-9}$ & $8.83\times 10^{-9}$ & $8.72\times 10^{-9}$ & $9.06\times 10^{-9}$ & $8.95\times 10^{-9}$ & $8.90\times 10^{-9}$ & $8.65\times 10^{-9}$ \\
50 & 3 & $7.17\times 10^{-9}$ & $8.57\times 10^{-9}$ & $8.57\times 10^{-9}$ & $9.02\times 10^{-9}$ & $9.12\times 10^{-9}$ & $8.90\times 10^{-9}$ & $8.46\times 10^{-9}$ \\
50 & 10 & $7.68\times 10^{-9}$ & $8.23\times 10^{-9}$ & $8.70\times 10^{-9}$ & $9.07\times 10^{-9}$ & $9.02\times 10^{-9}$ & $8.91\times 10^{-9}$ & $8.28\times 10^{-9}$ \\
50 & 35 & $7.45\times 10^{-9}$ & $8.77\times 10^{-9}$ & $8.60\times 10^{-9}$ & $9.19\times 10^{-9}$ & $8.98\times 10^{-9}$ & $9.13\times 10^{-9}$ & $8.74\times 10^{-9}$ \\
100 & 0.75 & $7.58\times 10^{-9}$ & $8.89\times 10^{-9}$ & $8.83\times 10^{-9}$ & $9.16\times 10^{-9}$ & $9.21\times 10^{-9}$ & $9.14\times 10^{-9}$ & $8.81\times 10^{-9}$ \\
100 & 3 & $7.33\times 10^{-9}$ & $8.43\times 10^{-9}$ & $8.96\times 10^{-9}$ & $8.97\times 10^{-9}$ & $9.19\times 10^{-9}$ & $9.03\times 10^{-9}$ & $8.79\times 10^{-9}$ \\
100 & 10 & $7.62\times 10^{-9}$ & $8.75\times 10^{-9}$ & $9.01\times 10^{-9}$ & $9.14\times 10^{-9}$ & $9.38\times 10^{-9}$ & $8.97\times 10^{-9}$ & $8.74\times 10^{-9}$ \\
100 & 35 & $7.51\times 10^{-9}$ & $8.62\times 10^{-9}$ & $8.77\times 10^{-9}$ & $9.14\times 10^{-9}$ & $9.08\times 10^{-9}$ & $9.24\times 10^{-9}$ & $8.64\times 10^{-9}$ \\
\hline
\end{tabular}
\caption{Final stationarity residual $\|F(w_K)\|$, averaged over five seeds.
All methods reach the requested first-order tolerance.}
\label{tab:highdim-final-residuals}
\end{table}

\begin{table}[htbp]
\centering
\scriptsize
\setlength{\tabcolsep}{3pt}
\begin{tabular}{cc|ccccccc}
\hline
$d$ & $r$ & LRSGA & LRSGAexact & GDA & OGDA & EG & SGA & CGD\\
\hline
50 & 0.75 & $6.46\times 10^{-9}$ & $1.01\times 10^{-8}$ & $9.30\times 10^{-9}$ & $9.68\times 10^{-9}$ & $9.58\times 10^{-9}$ & $1.02\times 10^{-8}$ & $9.31\times 10^{-9}$ \\
50 & 3 & $5.91\times 10^{-9}$ & $9.31\times 10^{-9}$ & $9.11\times 10^{-9}$ & $9.61\times 10^{-9}$ & $9.70\times 10^{-9}$ & $1.02\times 10^{-8}$ & $9.06\times 10^{-9}$ \\
50 & 10 & $6.24\times 10^{-9}$ & $8.78\times 10^{-9}$ & $9.18\times 10^{-9}$ & $9.61\times 10^{-9}$ & $9.55\times 10^{-9}$ & $1.03\times 10^{-8}$ & $8.81\times 10^{-9}$ \\
50 & 35 & $6.20\times 10^{-9}$ & $9.39\times 10^{-9}$ & $9.06\times 10^{-9}$ & $9.70\times 10^{-9}$ & $9.49\times 10^{-9}$ & $1.05\times 10^{-8}$ & $9.28\times 10^{-9}$ \\
100 & 0.75 & $5.82\times 10^{-9}$ & $1.05\times 10^{-8}$ & $9.74\times 10^{-9}$ & $1.01\times 10^{-8}$ & $1.02\times 10^{-8}$ & $1.08\times 10^{-8}$ & $9.81\times 10^{-9}$ \\
100 & 3 & $5.58\times 10^{-9}$ & $9.39\times 10^{-9}$ & $9.86\times 10^{-9}$ & $9.89\times 10^{-9}$ & $1.01\times 10^{-8}$ & $1.07\times 10^{-8}$ & $9.77\times 10^{-9}$ \\
100 & 10 & $5.81\times 10^{-9}$ & $9.60\times 10^{-9}$ & $9.87\times 10^{-9}$ & $1.00\times 10^{-8}$ & $1.03\times 10^{-8}$ & $1.06\times 10^{-8}$ & $9.67\times 10^{-9}$ \\
100 & 35 & $5.74\times 10^{-9}$ & $9.50\times 10^{-9}$ & $9.60\times 10^{-9}$ & $1.00\times 10^{-8}$ & $9.95\times 10^{-9}$ & $1.09\times 10^{-8}$ & $9.54\times 10^{-9}$ \\
\hline
\end{tabular}
\caption{Final distance to the known equilibrium,
$\|w_K-w_*\|$, averaged over five seeds.}
\label{tab:highdim-final-distances}
\end{table}

\begin{figure}[htbp]
\centering
\begin{minipage}{0.48\linewidth}
    \centering
    \includegraphics[width=\linewidth]{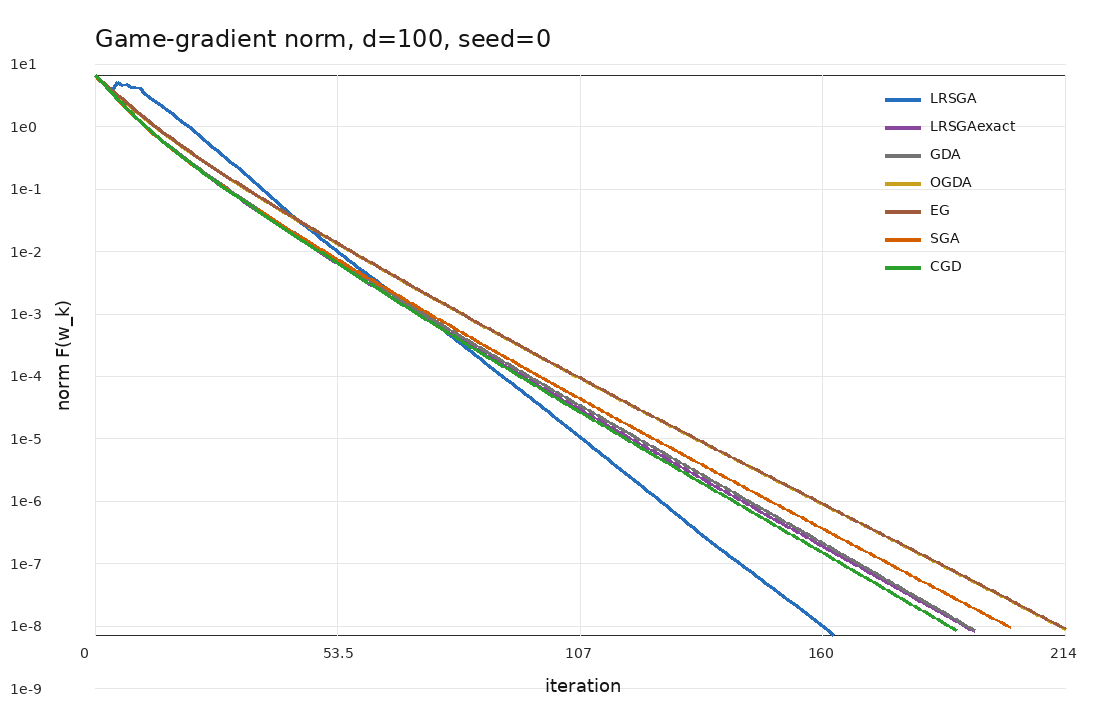}\\
    \small (a) Stationarity residual
\end{minipage}\hfill
\begin{minipage}{0.48\linewidth}
    \centering
    \includegraphics[width=\linewidth]{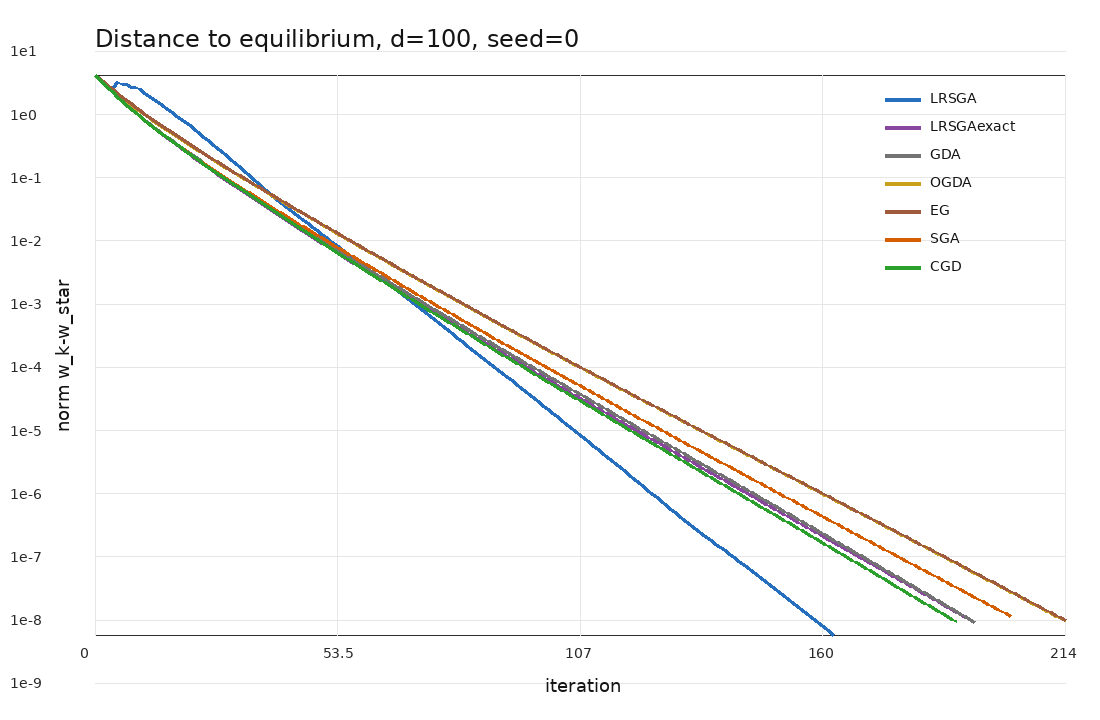}\\
    \small (b) Distance to $w_*$
\end{minipage}

\vspace{0.5em}

\begin{minipage}{0.58\linewidth}
    \centering
    \includegraphics[width=\linewidth]{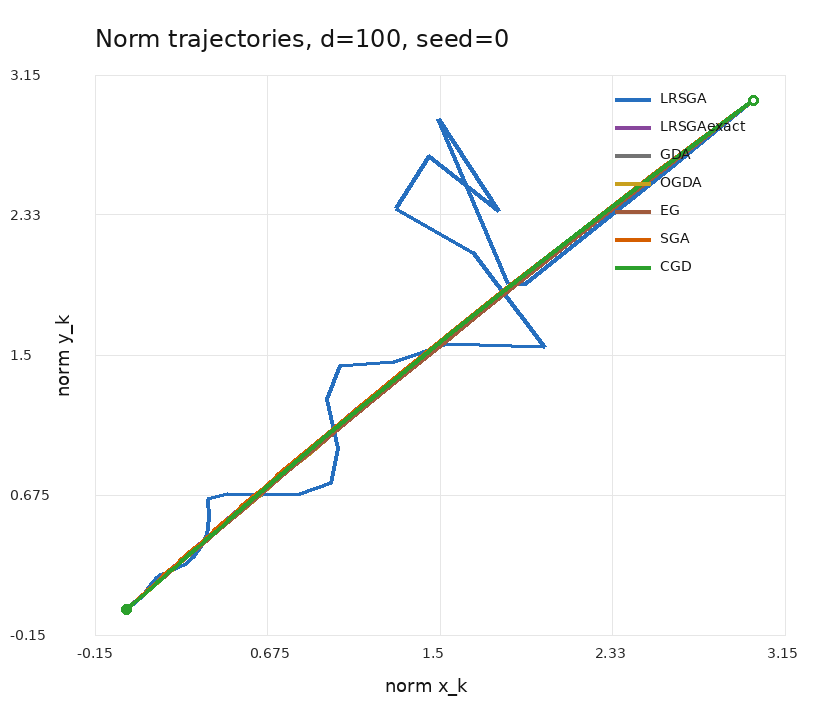}\\
    \small (c) Trajectory in the $(\|x_k\|,\|y_k\|)$ plane
\end{minipage}
\caption{Representative trajectories for $d=100$, initial radius $r=3$, and one fixed seed.  Panel (a) shows the game-gradient norm $\|F(w_k)\|$, panel (b)
shows the distance to the known equilibrium $\|w_k-w_*\|$, and panel (c)
shows the trajectory projected onto the two playerwise norms.  The color
convention is the same in all panels: LRSGA is blue, LRSGAexact is purple,
GDA is gray, OGDA is dark yellow, EG is brown, SGA is orange, and CGD is
green.}
\label{fig:highdim-trajectories}
\end{figure}

The numerical results show three consistent trends.  First, all methods
converge for all dimensions, radii, and seeds, and the final residuals in
Table~\ref{tab:highdim-final-residuals} are all of order $10^{-8}$ at least.  Second,
LRSGA with random Broyden initialization requires the smallest number of
iterations to satisfy the stopping criterion in every configuration in
Table~\ref{tab:highdim-iterations}.
This remains true even when compared with LRSGAexact, which starts from the
exact Jacobian blocks. This counterintuitive finding may be due to an implicit damping effect: an inexact skew correction in the early iterations effectively reduces the step taken in the antisymmetric direction, which can moderate oscillations and accelerate entry into the contractive neighborhood. A precise characterization of this phenomenon is beyond the scope of the present work and is left for future investigation.
The two first-order baselines OGDA and EG also converge in every
configuration, but require slightly more iterations than GDA and noticeably
more than LRSGA in all the settings of
Table~\ref{tab:highdim-iterations}.  Third, the mean per-iteration cost of LRSGA is
comparable to the other second-order-corrected methods in this deterministic
experiment, while GDA and OGDA are cheaper per iteration because they do not
compute or approximate mixed-derivative corrections; EG costs roughly twice
GDA, owing to the two gradient evaluations per iteration.
In this controlled problem GDA, OGDA, and EG also converge reliably. Nevertheless, as
discussed in the introduction through the motivating example in
Figure~\ref{figGD}, purely first-order game-gradient methods may exhibit
oscillatory behavior and this motivates the use of mixed-derivative information,
or approximations of it, as in SGA and LRSGA.


The advantage of LRSGA in this experiment is therefore primarily visible in
iteration count and residual decay, while its main computational motivation
remains the avoidance of explicit second-order derivatives in settings where
those derivatives are expensive or unavailable.  Although LRSGA and SGA can
exhibit similar trajectories, except for the effect of the Broyden
initialization in particularly delicate nonconvex cases, the computational
benefit of avoiding explicit mixed-derivative computations becomes more
pronounced in large-scale neural-network experiments.  For this reason, the
next experiment focuses on the per-epoch training time of a CLIP model.  That
experiment is intended as a direct timing comparison: at comparable training
performance, measured through the achieved loss values, LRSGA can substantially
reduce the computational cost with respect to SGA because it requires fewer
backward passes.

\subsection{Computational speedup of LRSGA with respect to SGA with explicitly assembled mixed-derivative blocks on a CLIP-inspired neural training toy example}
\label{subsec:clip-inspired-experiment}

To empirically evaluate the computational advantage of the proposed LRSGA method
with respect to the exact SGA optimization scheme, we designed a controlled
experimental protocol inspired by the CLIP architecture.  The purpose of this
experiment is not to verify convergence to a Nash equilibrium, but rather to
compare the computational cost of LRSGA and SGA when both methods are applied to
a neural-network training task with comparable achieved loss values.

Specifically, we consider a Contrastive Language-Image Pretraining (CLIP) model
\cite{radford2021learning}, due to its naturally coupled architecture: two
separate neural networks, one encoding images and the other encoding textual
descriptions, are jointly trained to learn a shared latent representation
space.  The training process of a CLIP architecture involves two contrastive
losses: the \textit{image-to-text loss}, which aligns image embeddings to their
corresponding text embeddings, and the \textit{text-to-image loss}, which aligns
text embeddings back to their corresponding image embeddings.

Unlike the classical approach, which minimizes the sum of the two losses, we
optimize each loss independently with respect to the parameters of its
corresponding encoder.  In this way, the image encoder minimizes the
image-to-text loss, while the text encoder minimizes the text-to-image loss.
This yields a nonzero-sum game-theoretic formulation suitable for SGA and
LRSGA.  Although the CLIP task is not inherently competitive in the strict
sense, nonzero-sum formulations naturally arise in competitive optimization
\cite{Balduzzi2018} and have also been explored in multitask learning under
game-theoretic frameworks \cite{Navon2022}.  This makes the experiment a useful
toy benchmark for studying game-based optimization methods in neural-network
training.

Specifically, our model comprises two neural networks, an ImageEncoder and a
TextEncoder, whose architectural details are summarized in
Figure~\ref{fig:clipStruttura}.

\begin{figure}[htbp]
    \centering
    \includegraphics[width=1\textwidth]{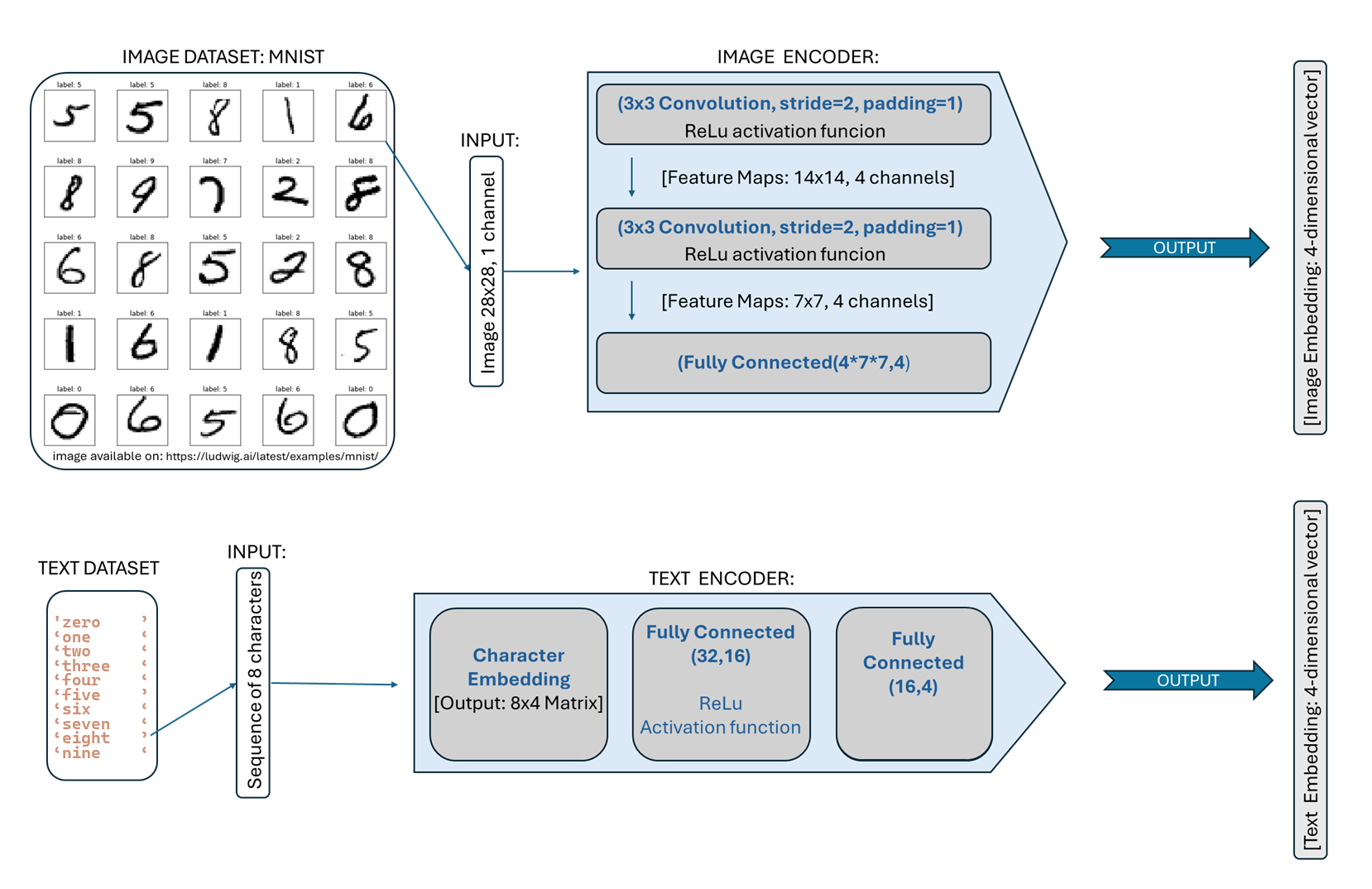}
    \caption{Architecture of the CLIP network used in the experiments.}
    \label{fig:clipStruttura}
\end{figure}

\paragraph{Dataset.}
Our experiments use a modified version of the MNIST dataset, where numerical
digit labels, from 0 to 9, are transformed into textual representations
(``zero'', ``one'', \ldots, ``nine'').  Each textual label is encoded into
numerical indices based on a defined character-to-index mapping.  We construct
a balanced subset containing 640 samples, with 64 samples per class, and divide
it into training, validation, and test sets, using proportions 60\%, 20\%, and
20\%, respectively.  During training, data are processed using mini-batches of
size 16, meaning that the model parameters are updated after computing
gradients on subsets of 16 samples at a time.

\paragraph{Neural architecture details.}
The detailed structures of the ImageEncoder and TextEncoder modules are as
follows:
\begin{itemize}
    \item \textbf{ImageEncoder}: a convolutional neural network composed of two
    convolutional layers followed by a fully connected embedding layer.  The
    output embedding dimension is fixed to 4, with ReLU activations after each
    convolutional layer, and the final embeddings are L2-normalized.
    \item \textbf{TextEncoder}: a neural network composed of a character-level
    embedding layer followed by two fully connected layers.  Textual sequences,
    fixed to 8 characters, are embedded into 4-dimensional vectors and
    L2-normalized at the output.  Image pixels are rescaled between 0 and 1.
\end{itemize}

\paragraph{Loss function.}

\begin{figure}[htbp]
    \centering
    \includegraphics[width=0.5\textwidth]{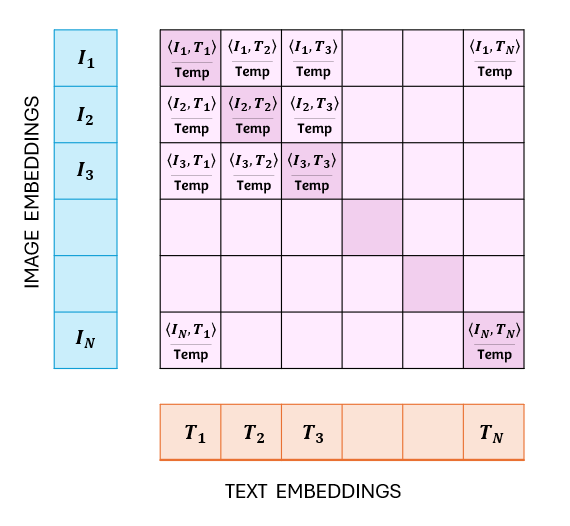}
    \caption{Illustration of the similarity matrix between image and text
    embeddings for a batch of size $N$.}
    \label{fig:similarity matrix}
\end{figure}

To train this CLIP-inspired model, we employ a symmetric contrastive loss as
in~\cite{radford2021learning}, composed of two terms:
\begin{itemize}
    \item \textbf{The image-to-text loss}, which encourages the model to assign
    a high probability to the correct text given an image.
    \item \textbf{The text-to-image loss}, which encourages the model to assign
    a high probability to the correct image given a text.
\end{itemize}

The similarity between image and text embeddings is computed through a
\emph{similarity matrix} $\texttt{logits}$, represented in
Figure~\ref{fig:similarity matrix}, whose entries are defined as
\[
\texttt{logits}_{i,j}
=
\frac{\langle \mathbf{I}_i, \mathbf{T}_j \rangle}{Temp},
\]
where $\mathbf{I}_i$ and $\mathbf{T}_j$ represent the normalized embeddings of
the $i$-th image and the $j$-th text, respectively, and $Temp$ is a temperature
hyperparameter used to scale the distribution.  In
\cite{radford2021learning}, this parameter is optimized during training.

The diagonal terms of the similarity matrix represent correct image-text
pairs, while the off-diagonal terms correspond to mismatched pairs.  To compute
the image-to-text loss for a batch of size $N$, each row of the similarity
matrix is passed through a softmax function, assigning a probability
distribution over all text candidates given an image.  The correct target for
each image is assumed to be the text at the same index position within the
batch.  Accordingly, the target vector is defined as
$\mathbf{y}=[1,\dots,N]$, which specifies that the $i$-th image is paired with
the $i$-th text.  This formulation enables the use of the standard
cross-entropy loss, where the diagonal entry of each row is treated as the
correct class.  The image-to-text loss is therefore
\[
{Loss}_{I}
=
\frac{1}{N}
\sum_{i=1}^N
-\log
\left(
\frac{\exp(\texttt{logits}_{i,i})}
{\sum_{j=1}^N \exp(\texttt{logits}_{i,j})}
\right).
\]
The text-to-image loss ${Loss}_{T}$ is computed analogously by transposing the
similarity matrix $\texttt{logits}$.

\paragraph{Experimental procedure.}
We compare the exact SGA optimizer \eqref{MSGA} with its low-rank variant,
LRSGA.  In this implementation, SGA explicitly assembles the
mixed-derivative blocks appearing in \eqref{MSGA}. This choice is deliberate:
these explicit blocks are precisely the matrix objects analyzed in
Sections~\ref{sec:SGAtheory}--\ref{sec:LRSGAtheory}, so that SGA and LRSGA differ only in how these blocks are
obtained, by exact computation or by rank-one secant updates, thereby isolating
the computational cost removed by the secant updates.  The aim is to assess whether LRSGA can reach test loss values
comparable to SGA while substantially reducing the training time required by
the optimizer.

To ensure statistical robustness and fairness, each configuration is evaluated
over multiple random seeds under identical training conditions.  In our
experiments, gradients are computed on mini-batches of the complete dataset.
We refer to one pass through the complete dataset as one epoch.  All models
are trained for the same number of epochs, using the same initialization scheme
and batch size.  The parameter $\eta$ varies in
$\{0.01,0.001,0.0001\}$, the temperature parameter is fixed at $0.09$, and the
auxiliary parameter $\tau$ is set to $\tau=\eta/100$, a conservative choice that keeps the secant correction small relative to the gradient step and proved empirically effective across all tested learning rates.

The optimizers are evaluated according to the following criteria:
\begin{itemize}
    \item \textbf{Computation time}: average CPU time per epoch, measured in
    seconds.
    \item \textbf{Environmental impact}: total CO$_2$ emissions, measured in
    grams, estimated with CodeCarbon\footnote{\url{https://codecarbon.io/}}.
    \item \textbf{Test losses}: image-to-text and text-to-image contrastive
    losses computed on the test set.
\end{itemize}

Since SGA with explicit matrix assembly requires high computation times, all training processes are halted
after 150 epochs.  Although this may be suboptimal from a training perspective,
it is sufficient for the purpose of this experiment, namely to show that LRSGA
achieves loss values comparable to SGA, see Figures~\ref{fig:TrainingSGA}
and~\ref{fig:LossesBoxPlots2}, while being notably faster and producing lower
estimated emissions, see Figure~\ref{fig:TimesAndEmissions2} and
Table~\ref{tab:results_summary2}.

\paragraph{Numerical results and discussion.}
We now present the detailed results of the comparison between LRSGA and SGA
according to the experimental protocol described above.

Figure~\ref{fig:TrainingSGA} shows the average training losses across epochs
for the two optimizers.  The curves correspond to the mean loss values per
epoch for each optimizer, where LRSGA is shown in blue and SGA in orange.
These plots highlight the similarity in training dynamics between the two
methods, with the LRSGA curves closely following those of SGA across all tested
values of~$\eta$.

\begin{figure}[htbp]
    \centering
    \begin{subfigure}{0.3\textwidth}
        \includegraphics[width=\linewidth]{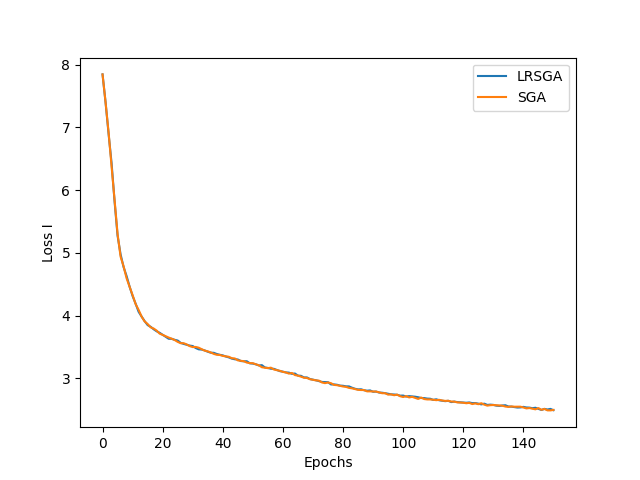}
        \caption{LossI values on training set for $\eta = 0.0001$.}
    \end{subfigure}
    \hfill
    \begin{subfigure}{0.3\textwidth}
        \includegraphics[width=\linewidth]{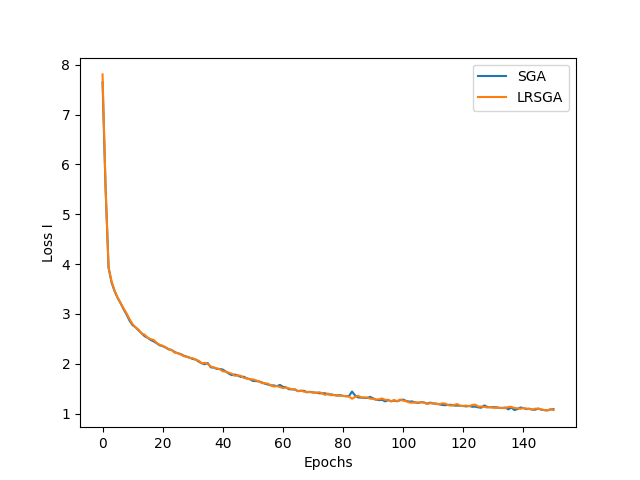}
        \caption{LossI values on training set for $\eta = 0.001$.}
    \end{subfigure}
    \hfill
    \begin{subfigure}{0.3\textwidth}
        \includegraphics[width=\linewidth]{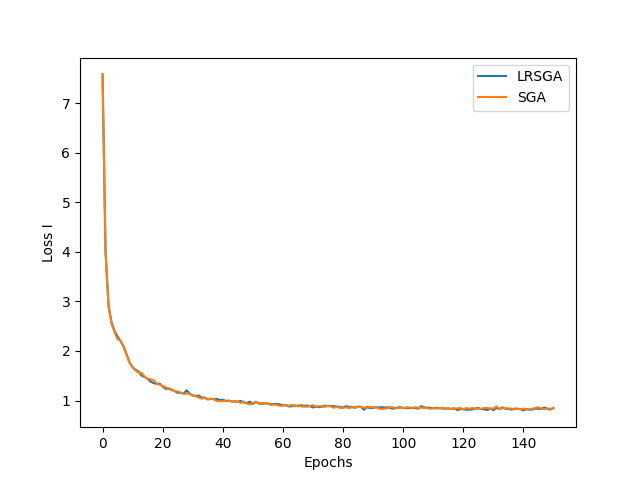}
        \caption{LossI values on training set for $\eta = 0.01$.}
    \end{subfigure}

    \vspace{1em}

    \begin{subfigure}{0.3\textwidth}
        \includegraphics[width=\linewidth]{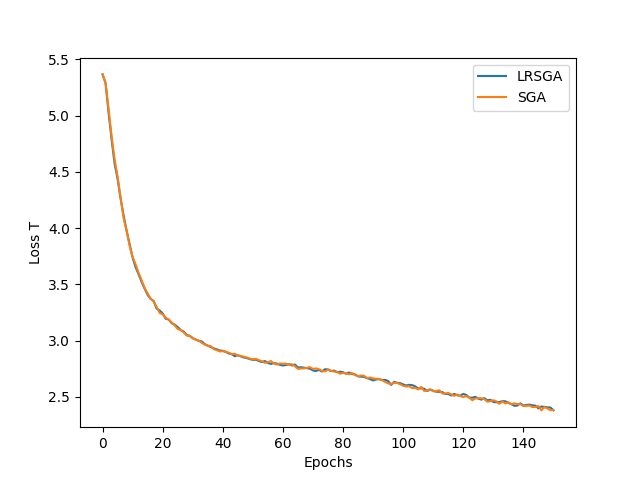}
        \caption{LossT values on training set for $\eta = 0.0001$.}
    \end{subfigure}
    \hfill
    \begin{subfigure}{0.3\textwidth}
        \includegraphics[width=\linewidth]{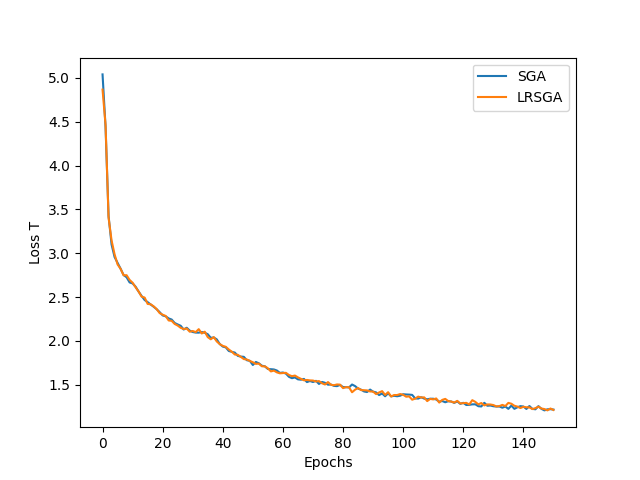}
        \caption{LossT values on training set for $\eta = 0.001$.}
    \end{subfigure}
    \hfill
    \begin{subfigure}{0.3\textwidth}
        \includegraphics[width=\linewidth]{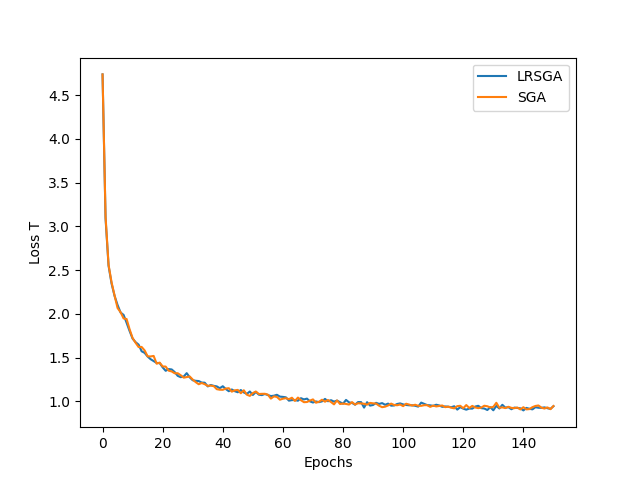}
        \caption{LossT values on training set for $\eta = 0.01$.}
    \end{subfigure}

    \caption{Training loss curves, LossI and LossT, for LRSGA and SGA across
    different values of $\eta$.}
    \label{fig:TrainingSGA}
\end{figure}

The visual similarity suggests that, during training, LRSGA reaches loss values
comparable to those of SGA.  In particular, both the image-to-text and
text-to-image loss curves of LRSGA almost overlap with those of SGA, indicating
that the two methods follow similar training dynamics.

To further validate that both optimizers achieve similar test losses, we
evaluate the contrastive losses on the test set, as illustrated in the box
plots in Figure~\ref{fig:LossesBoxPlots2}.  The interquartile range (IQR) for each box
extends from the 25th percentile to the 75th percentile, illustrating the
central distribution of the data.  The black horizontal line within each box
indicates the median value.  The whiskers extend up to 1.5 times the
interquartile range, and data points outside this range are displayed
individually as circles.  Statistical significance from pairwise comparisons
between the two optimizers is indicated above each box plot using standard
notation derived from independent-sample $t$-tests: \textit{ns}, not
statistically significant, $0.05 < p \leq 1.00$; $^{*}$, slight significance,
$0.01 < p \leq 0.05$; $^{**}$, moderate significance,
$0.001 < p \leq 0.01$; $^{***}$, high significance,
$0.0001 < p \leq 0.001$; and $^{****}$, very high significance,
$p \leq 0.0001$.

\begin{figure}[htbp]
    \centering
    \includegraphics[width=1\textwidth]{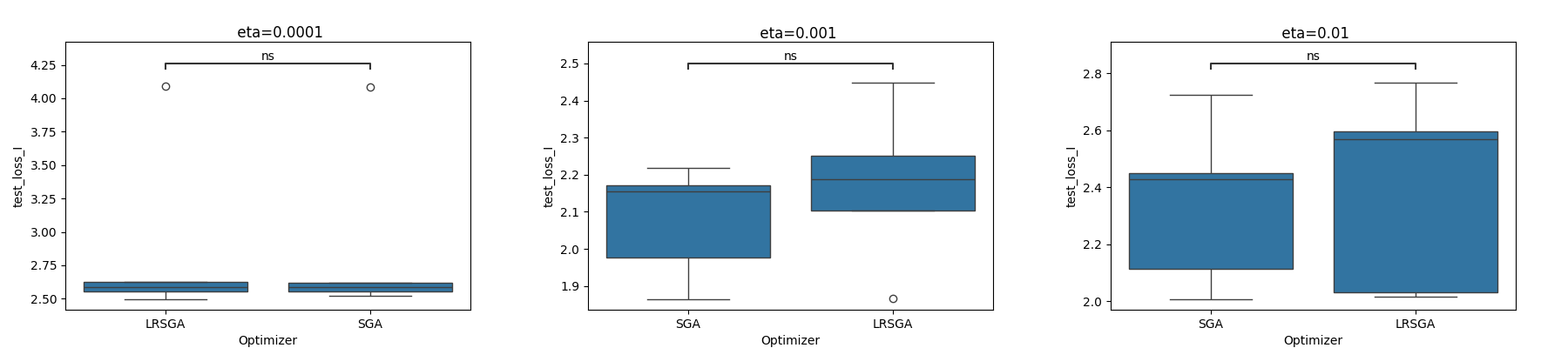}
    \includegraphics[width=1\textwidth]{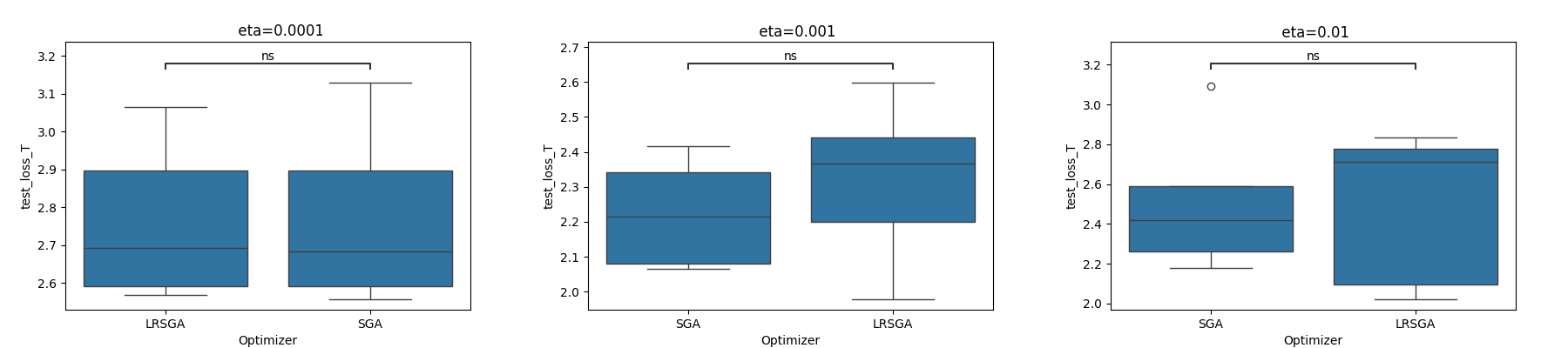}
    \caption{Box plots of image-to-text and text-to-image losses computed on
    the test set for each configuration.}
    \label{fig:LossesBoxPlots2}
\end{figure}

\begin{figure}[htbp]
    \centering
    \includegraphics[width=1\textwidth]{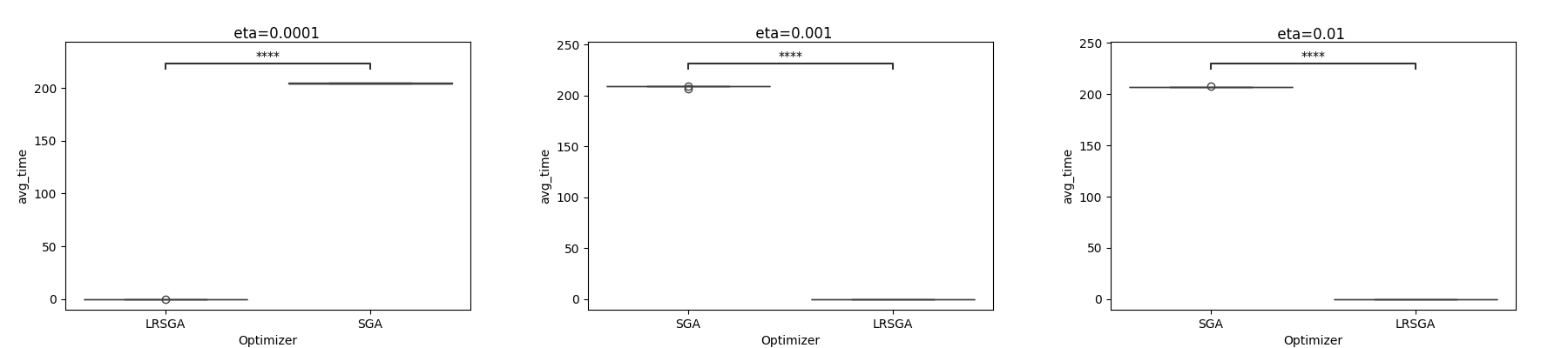}
    \includegraphics[width=1\textwidth]{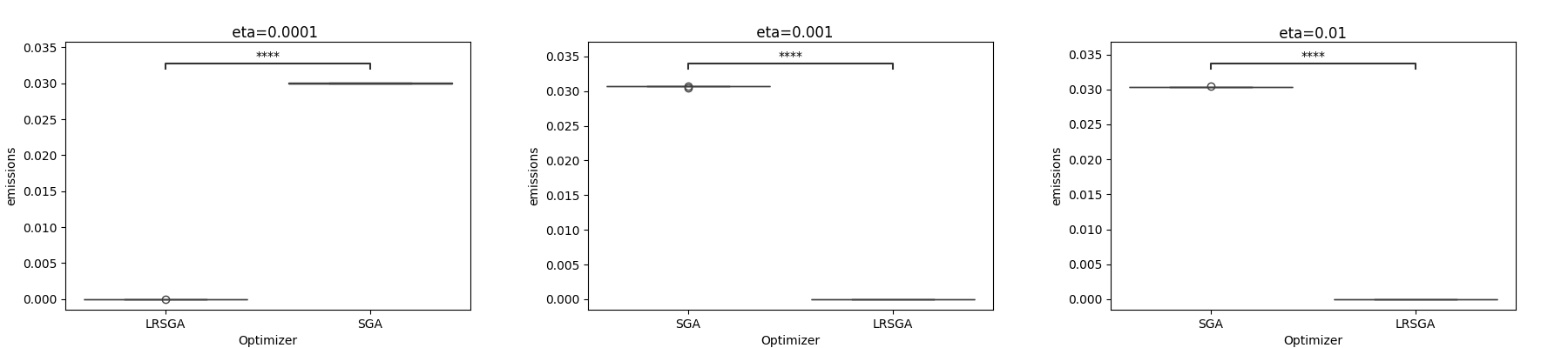}
    \caption{Box plots of average epoch time and total carbon emissions for
    each configuration.}
    \label{fig:TimesAndEmissions2}
\end{figure}

To assess LRSGA and SGA computational
efficiency, we also measure the average execution time per epoch and the total
carbon emissions produced during training.
This comparison is summarized in Figure~\ref{fig:TimesAndEmissions2}.

Table~\ref{tab:results_summary2} summarizes the numerical comparison
between LRSGA and SGA, reporting the mean and standard deviation for each
metric.  This tabular view complements the box plots and provides a direct
overview of the differences between the two methods.

\begin{table}[htbp]
\centering
\caption{Pairwise comparison of LRSGA against SGA in terms of test losses,
execution time per epoch, and estimated environmental impact.  Results are
reported as mean $\pm$ standard deviation.  Values in \textbf{bold} indicate
statistically significant differences in favor of LRSGA with respect to SGA.}
\label{tab:results_summary2}
\resizebox{\textwidth}{!}{%
\begin{tabular}{lccccc}
\toprule
Optimizer & $\eta$ & Test Image-to-Text Loss & Test Text-to-Image Loss & Avg. Epoch Time (s) & Carbon emission ($g$) \\
\midrule
SGA   & 0.01   & $2.3445 \pm 0.2863$ & $2.5088 \pm 0.3627$ & $207.1532 \pm 0.5153$ & $30.39 \pm 0.075$ \\
LRSGA & 0.01   & $2.3956 \pm 0.3471$ & $2.4886 \pm 0.3955$ & $\mathbf{0.1936 \pm 0.0085}$ & $\mathbf{0.030 \pm 1.3\times10^{-3}}$ \\
\midrule
SGA   & 0.001  & $2.0776 \pm 0.1504$ & $2.2235 \pm 0.1556$ & $208.8654 \pm 0.9728$ & $30.65 \pm 0.14$ \\
LRSGA & 0.001  & $2.1709 \pm 0.2120$ & $2.3173 \pm 0.2371$ & $\mathbf{0.1828 \pm 0.0093}$ & $\mathbf{0.030 \pm 1.4\times10^{-3}}$ \\
\midrule
SGA   & 0.0001 & $2.8739 \pm 0.6794$ & $2.7721 \pm 0.2390$ & $204.6436 \pm 0.0938$ & $30.03 \pm 0.015$ \\
LRSGA & 0.0001 & $2.8725 \pm 0.6844$ & $2.7625 \pm 0.2132$ & $\mathbf{0.2056 \pm 0.0103}$ & $\mathbf{0.030 \pm 1.6\times10^{-3}}$ \\
\bottomrule
\end{tabular}
}
\end{table}

As observed in both the box plots and Table~\ref{tab:results_summary2}, LRSGA
achieves test loss values comparable to those of SGA, while significantly
reducing training time and estimated carbon emissions.  Thus, in this
CLIP-inspired toy experiment, replacing explicit mixed-derivative computations
with Broyden approximations preserves the achieved loss values while producing
a substantial computational speedup with respect to SGA with explicitly assembled mixed-derivative blocks.

\paragraph{Scope and limitations of this experiment.}
The purpose of this CLIP-inspired experiment is not to verify convergence to a
Nash equilibrium, nor to certify the first-order residual $\|F(w_k)\|$ during
neural-network training. Moreover, because the encoders include ReLU
activations, the resulting parameter-to-loss maps are not \(C^3\) and therefore
fall outside the smooth deterministic setting of the convergence theory.
Rather, the experiment is designed as a computational
comparison between SGA, which uses explicit mixed-derivative computations, and
LRSGA, which replaces them with Broyden-type approximations.  The relevant
question in this experiment is therefore whether LRSGA can achieve comparable
training and test losses while substantially reducing the average time per
epoch and the associated estimated carbon emissions.

This experiment should be interpreted as a toy neural benchmark for the
following reasons.
\begin{enumerate}
    \item Training is performed with mini-batches.  This is standard in
    neural-network optimization, but it is not ideal for estimating
    second-order information through gradient differences between successive
    iterations, since the two consecutive gradients are computed on different
    mini-batches. 
    Stochastic quasi-Newton strategies, such as computing gradient
    differences on the same mini-batch or computing them on the intersection of two consecutive mini-batches \cite{berahas2016multi,doi:10.1137/140954362}, could be incorporated in
    future stochastic variants of LRSGA.

    \item The neural architecture of this experiment must remain very small because the
    current LRSGA formulation does not reduce the memory required to store the
    matrices approximating the mixed derivatives.  In SGA, the explicit mixed
    derivatives are expensive both computationally and in memory, since their
    dimensions scale with the number of parameters of the two neural networks.
    LRSGA addresses the computational cost associated with the explicit
    computation of these mixed derivatives, but not the memory cost. In the
    present implementation, in fact, it stores the full Broyden approximations
    \(\mu_k\approx D\partial_x f\in\mathbb{R}^{m\times(m+n)}\) and
    \(\nu_k\approx D\partial_y g\in\mathbb{R}^{n\times(m+n)}\), from which the
    mixed-derivative blocks are extracted. Hence, the memory requirement is
    even larger than that of storing only the exact mixed-derivative blocks.
    Therefore, the memory bottleneck remains a
    significant limitation for large neural models.  A natural continuation is
    to develop a limited-memory version of LRSGA, in the spirit of
    limited-memory methods such as L-BFGS, while taking into
    account that the present setting is a game-theoretic problem for computing
    Nash equilibria rather than the minimization of a single scalar objective as for L-BFGS.

    \item The timing comparison concerns SGA with explicitly assembled
    mixed-derivative blocks.  Matrix-free implementations of SGA can compute
    the correction \(A^\top \xi\) through Jacobian-vector products at the cost
    of two additional reverse-mode passes \cite[Section~1]{Balduzzi2018},
    \cite[Section~1.2 and Appendix~A]{Letcher2019}, thereby avoiding explicit
    block assembly.  A comparison with such matrix-free variants is
    methodologically paired with a limited-memory version of LRSGA, see Section~\ref{sec:FutureWork}, since both approaches abandon the explicit matrix representation on
    which the analysis in Sections~\ref{sec:SGAtheory}--\ref{sec:LRSGAtheory} is based; this comparison is the
    subject of a manuscript in preparation.  
\end{enumerate}

\endgroup

\section{Conclusions}
\label{sec:Conclusion}

A theoretical and numerical investigation 
of the symplectic gradient adjustment (SGA) method and of a low-rank SGA (LRSGA) 
method for efficiently solving \textcolor{revviolet}{optimization problems arising from two-player Nash games} was presented. The SGA method  
considerably improves upon the simple gradient method 
by including second-order mixed derivatives computed at each iterate. However, 
the use of second-order derivatives requires additional computational effort 
in back-propagation processes, which is the main disadvantage of this method. 
For this reason, \textcolor{revblue}{an} LRSGA method was developed where the approximation to
second-order mixed derivatives is obtained by rank-one updates. 
The theoretical analysis presented in this work focused on novel convergence 
estimates for the SGA method and the newly proposed LRSGA method, 
including parameter bounds. 
\textcolor{revgreen}{The numerical experiments complement the theoretical
analysis in three directions: a low-dimensional nonconvex game illustrates
local equilibrium selection in the presence of multiple Nash equilibria; a
high-dimensional nonlinear game with a known stable Nash equilibrium allows us
to monitor stationarity residuals, distances to equilibrium, iteration counts,
and per-iteration costs; and a CLIP-inspired neural training toy example shows
that LRSGA can preserve loss values comparable to exact SGA while substantially
reducing the training time associated with the explicit assembly of
mixed-derivative blocks.}
\section{Future Work}
\label{sec:FutureWork}
\textcolor{revblue}{As future work, we plan to conduct a more rigorous analysis of stochastic quasi-Newton strategies in order to develop a stochastic variant of LRSGA. 
Although LRSGA is effective in reducing the computational burden associated with the explicit evaluation of mixed second-order derivatives, 
it still requires storing the secant matrices used to build the low-rank correction. This memory side effect could be mitigated by low-memory variants;
 the design and analysis of such variants will be the subject of future work. 
 Furthermore, the selection of more appropriate values for the parameters \(\eta\) and \(\tau\), together with a deeper analysis of global convergence properties, remains a topic for future investigation.}

\begin{acknowledgements}

The authors would like to express their sincere gratitude to Professor Lionel Trojman from the Institut Supérieur d'Électronique de Paris (ISEP) for kindly granting access to the research server used during the experimental CLIP phase of this work, equipped with a 12th Gen Intel Core i9-12900K processor featuring 16 physical cores and 24 threads, with a maximum frequency of 5.2~GHz.

K. R. Foglia has been partly funded by
the PhD program in Mathematics and Computer Science at University of Calabria with the support of a scholarship financed by DM 351/2022 (CUP H23C22000440007), based on the NRRP funded by the European Union.\\

V. Colao has been partly funded by the Research project of MUR - Prin 2022 “Nonlinear differential problems with applications to real phenomena” (Grant Number: 2022ZXZTN2).
\end{acknowledgements}
\medskip

\section*{Conflict of interest statement}
This study does not have any conflicts to disclose.

\section*{Data availability statement}
The data that support the findings of this study are available from the corresponding author upon reasonable request.


\bibliographystyle{abbrv}
\bibliography{literature}

\end{document}